\documentclass[11pt]{article}
\usepackage{color}
\usepackage{amsmath}
\catcode`\@=11 \@addtoreset{equation}{section}

\catcode`\@=12

\usepackage{amssymb}
\usepackage{amsfonts}
\usepackage{xcolor}
\usepackage{graphics}
\usepackage{epsfig,psfrag,graphicx}
\usepackage{amssymb}
\usepackage{eepic,epic}
\usepackage{dsfont}

\usepackage{enumerate}

\usepackage{epsfig} 
\usepackage{amsmath} 
\usepackage{amssymb} 
\usepackage{amsbsy} 

\newtheorem{definition}{Definition}[section]
\newtheorem{theorem}[definition]{Theorem}
\newtheorem{lemma}[definition]{Lemma}

\newtheorem{proposition}[definition]{Proposition}
\newtheorem{remark}[definition]{Remark}

\newcommand{\nc}{\newcommand}
\nc{\qed}{\mbox{}\nolinebreak\hfill \rule{2mm}{2mm}} 
\nc{\weak}{\rightharpoonup}
\nc{\weakstar}{\stackrel{\ast}{\rightharpoonup}} 
\nc{\proof}{{\bf Proof: }} 
\renewcommand{\div}{{{\mathrm{div}}_x}\,}
\newcommand{\vrho}{\varrho}
\nc{\modular}[1]{{\stackrel{ #1}{\longrightarrow\,}}}

\def\bbbone{{\mathchoice {\rm 1\mskip-4mu l}
{\rm 1\mskip-4mu l} {\rm 1\mskip-4.5mu l} {\rm 1\mskip-5mu l}}}

\def\tens#1{\pmb{\mathsf{#1}}}
\def\vec#1{\boldsymbol{#1}}

\newcommand{\vr}		{\vrho}
\newcommand{\vre}		{\vr_\varepsilon}
\newcommand{\vrez}	{\vr_{0,\varepsilon}}

\newcommand{\vret}		{\wtilde{\vr}_\ep}
\newcommand{\ue}		{\vec{u}_\varepsilon}
\newcommand{\uez}		{\vec{u}_{0,\varepsilon}}
\newcommand{\ep}		{\varepsilon}
\newcommand{\temp}	{\vartheta}
\newcommand{\tem}		{\vartheta_\varepsilon}
\newcommand{\temz}	{\vartheta_{0,\varepsilon}}
\newcommand{\tems}	{\overline{\vartheta}}

\newcommand{\q}		{\vec{q}}
\newcommand{\n}		{\vec{n}}

\newcommand{\ess} 	{{\rm{ess}}}
\newcommand{\res}		{{\rm{res}}}
\newcommand{\dx}		{\,{\rm d}x}
\newcommand{\dt}		{\, {\rm d}t}
\newcommand{\dxdt}	{\, {\rm d}x{\rm d}t}

\newcommand{\U}		{\vec{U}}

\def\bbbone{{\mathchoice {\rm 1\mskip-4mu l}
{\rm 1\mskip-4mu l} {\rm 1\mskip-4.5mu l} {\rm 1\mskip-5mu l}}}

\renewcommand{\bbbone}{\mathds{1}}

\newcommand{\mbb}{\mathbb}

\newcommand{\mc}{\mathcal}
\newcommand{\mf}{\mathfrak}
\newcommand{\veps}{\varepsilon}
\newcommand{\vtheta}{\vartheta}
\newcommand{\what}{\widehat}
\newcommand{\wtilde}{\widetilde}
\newcommand{\vphi}{\varphi}
\newcommand{\oline}{\overline}

\newcommand{\ra}{\rightarrow}

\newcommand{\g}{\gamma}
\newcommand{\z}{\zeta}

\newcommand{\de}{\delta}

\newcommand{\lan}{\langle}
\newcommand{\ran}{\rangle}

\newcommand{\e}{\vec{e}}

\newcommand{\R}{\mathbb{R}}

\newcommand{\N}{\mathbb{N}}
\newcommand{\Z}{\mathbb{Z}}
\newcommand{\B}{\mathbb{B}}
\newcommand{\T}{\mathbb{T}^1}

\newcommand{\TT}{\mathbb{T}}

\newcommand{\h}{\mathbb{H}}

\renewcommand{\div}{{\rm div}\,}
\newcommand{\curl}{{\rm curl}\,}
\newcommand{\divh}{{\rm div}_h}
\newcommand{\curlh}{{\rm curl}_h}

\newcommand{\Supp}{{\rm Supp}\,}

\allowdisplaybreaks

\def\d{\partial}
\def\div{{\rm div}\,}

\title{\LARGE \bf{A multi-scale problem for\\ 
viscous heat-conducting fluids in fast rotation
}}
\author{ \textsl{Daniele Del Santo}$\,^1\;$, $\;$\textsl{Francesco Fanelli}$\,^2\;$, $\;$\textsl{Gabriele Sbaiz}$\,^{1,2}\;$,
$\;$\textsl{Aneta Wr\'oblewska-Kami\'nska}$\,^{3}$ \vspace{.2cm} \\
\footnotesize{$\,^1\;$ \textsc{Universit\`a degli Studi di Trieste}, \textit{Dipartimento di Matematica e Geoscienze},} \\
\footnotesize{Via Valerio 12/1, 34127 Trieste, Italy} \vspace{0.1cm} \\
\footnotesize{$\,^2\;$ \textsc{Univ. Lyon, Universit\'e Claude Bernard Lyon 1}, CNRS UMR 5208, \textit{Institut Camille Jordan},} \\
\footnotesize{43 blvd. du 11 novembre 1918, F-69622 Villeurbanne cedex, France} \vspace{0.1cm} \\
\footnotesize{$\,^3\;$ \textsc{Polish Academy of Sciences}, \textit{Institute of Mathematics,}} \\ {\footnotesize ul.\'Sniadeckich 8, 00-656 Warszawa, Poland}
\vspace{.2cm} \\
\footnotesize{\ttfamily{delsanto@units.it}$\,,\quad$  \ttfamily{fanelli@math.univ-lyon1.fr}$\,,\quad$\ttfamily{gabriele.sbaiz@phd.units.it}$\,,\quad$
\ttfamily{a.wroblewska@impan.pl}} \vspace{.1cm}
}

\date{\small \today}

\begin{document}
\maketitle

\abstract{In the present paper, we study the combined incompressible and fast rotation limits for the full Navier-Stokes-Fourier system with Coriolis, centrifugal and gravitational forces,
in the regime of small Mach, Froude and Rossby numbers and for general ill-prepared initial data.
We consider both the isotropic scaling (where all the numbers have the same order of magnitude) and the multi-scale case (where some effect is predominant with respect to the others).
In the case when the Mach number is of higher order than the Rossby number, we prove that the limit dynamics is described by an incompressible Oberbeck-Boussinesq system, where the velocity
field is horizontal (according to the Taylor-Proudman theorem), but vertical effects on the temperature equation are not negligible.
Instead, when the Mach and Rossby numbers have the same order of magnitude, and in absence of the centrifugal force, we show convergence to a quasi-geostrophic equation for a stream function
of the limit velocity field, coupled with a transport-diffusion equation for a new unknown, which links the target density and temperature profiles.

The proof of the convergence is based on a compensated compactness argument. The key point is to identify some compactness properties hidden in the system of acoustic-Poincar\'e waves.
Compared to previous results, our method enables first of all to treat the whole range of parameters in the multi-scale problem, and also
to consider a low Froude number regime with the somehow critical choice $Fr\,=\,\sqrt{Ma}$, where $Ma$ is the Mach number. This allows us to capture
some (low) stratification effects in the limit.}

\paragraph*{\small 2010 Mathematics Subject Classification:}{\footnotesize 35Q86 
(primary);
35B25, 
76U05, 
35B40, 
76M45 
(secondary).}


\section{Introduction}

In this paper, we are interested in the description of the dynamics of large-scale flows, like e.g. geophysical flows in the atmosphere and in the ocean.
From the physical viewpoint, three are the main features that the mathematical model has to retain (see \cite{C-R} and \cite{Val}, for instance):
(almost) incompressibility of the flow, stratification effects
(i.e. density variations, essentially due to the gravity) and the action of a strong Coriolis force (due to the fast rotation of the ambient system).
The importance of these effects is ``measured'' by introducing, correspondingly, three positive adimensional parameters: the Mach number $Ma$, linked with incompressibility,
the Froude number $Fr$, linked with stratification, and the Rossby number $Ro$, related to fast rotation.
Saying that the previous attributes are predominant in the dynamics corresponds to assuming that the values of those parameters are very small.
This is also the point of view that we adopt throughout all this paper.

In order to get a more realistic model, capable to capture also heat transfer processes in the dynamics, we focus on the full $3$-D Navier-Stokes-Fourier system with Coriolis,
centrifugal and gravitational forces, which we set in the physical domain
\begin{equation} \label{eq:domain}
\Omega\,:=\,\R^2\times\,]0,1[\,.
\end{equation}
Denote by $\vrho ,\, \vtheta\geq 0$ the density and the absolute temperature of the fluid, respectively, and by $\vec{u}\in \mathbb{R}^3$ its velocity field: in its
non-dimensional form, the system can be written (see e.g. \cite{F-N}) as
\begin{equation} \label{eq_i:NSF}
\begin{cases}
	\partial_t \vrho + \div (\vrho\vec{u})=0\  \\[2ex]
	\partial_t (\vrho\vec{u})+ \div(\vrho\vec{u}\otimes\vec{u}) + \dfrac{\e_3 \times \vrho\vec{u}}{Ro}\,  +    \dfrac{1}{Ma^2} \nabla_x p(\vrho,\vtheta) 
=\div \mbb{S}(\vtheta,\nabla_x\vec{u}) + \dfrac{\vrho}{Ro^2} \nabla_x F  + \dfrac{\vrho}{Fr^2} \nabla_x G  \\[2ex]
	 \partial_t \bigl(\vrho s(\vrho, \vtheta)\bigr)  + \div \bigl(\vrho s (\vrho,\vtheta)\vec{u}\bigr) + \div\left(\dfrac{\q(\vtheta,\nabla_x \vtheta )}{\vtheta} \right)
	= \sigma\,,
\end{cases}
\end{equation}
where $p$ denotes the pressure of the fluid, 
the functions $s,\vec{q},\sigma$ are the specific entropy, the heat flux and the
entropy production rate respectively, and $\mbb{S}$ is the viscous stress tensor, which satisfies Newton's rheological law.
We refer to Paragraphs \ref{sss:primsys} and \ref{sss:structural} below for details.
The two gradients $\nabla_xF$ and $\nabla_xG$ represent respectively the centrifugal and gravitational forces, while the term $\e_3 \times \vrho \vec{u}$ takes into account the Coriolis force.
Here $\e_3=(0,0,1)$ denotes the unit vector directed along the vertical axis and the symbol $\times$ stands for the usual external product of vectors in $\mathbb{R}^3$.
This is a very simple form of the Coriolis force, which is however physically well-justified at mid-latitudes (see e.g. \cite{C-R} and \cite{Ped} for details). In addition, this approximate
model already enables to capture several interesting physical phenomena related to the dynamics of geophysical flows: the so-called \emph{Taylor-Proudman theorem}, the formation
of \emph{Ekman layers} and the propagation of \emph{Poincar\'e waves}. We refer to \cite{C-D-G-G} for a more in-depth discussion. In the present paper, we deal only with the first and third
issues, which we will comment more in detail below. On the contrary, for simplicity of presentation, we avoid boundary layer effects, by imposing \emph{complete-slip} boundary conditions
(see conditions \eqref{bc1-2} and \eqref{bc3} below).

In \eqref{eq_i:NSF}, one can recognise the presence of the Mach, Froude and Rossby numbers. On the other hand, we have neglected other important characteristic numbers, such as
the Strouhal, Reynolds and P\'eclet numbers, which we have set equal to $1$.
Motivated by the initial discussion, our goal is to study the regime where $Ma$, $Fr$ and $Ro$ are small: given a parameter $\veps\in\,]0,1]$, we set
\begin{equation} \label{eq_i:scales}
Ro=\veps \, , \quad Ma=\veps^m \quad \text{and}\quad Fr=\veps^{m/2}\,, \qquad\qquad \mbox{ for some }\quad m\geq 1\,.
\end{equation}
This means that we will perform the incompressible, fast rotation and mild stratification limits all together. Moreover, the presence of the additional parameter $m$ allows us to
consider different regimes: the one where some effect is predominant with respect to the others (for $m>1$, which entails the presence of multiple scales in the system),
and the one where all the forces act at the same scale and their effects are in balance in the limit (when $m=1$).

\medbreak
This work enters into a general (and nowadays classical) research programme, consisting in taking singular limits for systems of PDEs related to fluid mechanics.
Concerning specifically models for geophysical flows, the study goes back to the pioneering works \cite{B-M-N_EMF}, \cite{B-M-N_AA}, \cite{B-M-N_IUM} of Babin, Mahalov and Nikolaenko
for the (homogeneous) incompressible Navier-Stokes equations with Coriolis force. We refer to \cite{C-D-G-G} for an overview of the (broad) literature on this subject. 
The case of compressible flows was considered for the first time in a general setting by Feireisl, Gallagher and Novotn\'y in \cite{F-G-N} for the barotropic Navier-Stokes
system (see also \cite{B-D-GV} for a preliminary study and \cite{G-SR_Mem} for the analysis of equatorial waves).

In the compressible case, because of both physical considerations and technical difficulties, it is natural to combine a low Rossby number regime (fast rotation limit) with a low Mach number regime
(incompressible limit). This opens the scenario to possible multi-scale analysis: if in \cite{F-G-N} the scaling focused on the specific choice $Ro=Ma=\veps$, i.e. $m=1$ in \eqref{eq_i:scales} above,
a more general instance was considered in \cite{F-G-GV-N} by the same authors together with G\'erard-Varet. There, the system under consideration was the same barotropic Navier-Stokes system
as in \cite{F-G-N}, with the addition of the centrifugal force term. Afterwards, Feireisl and Novotn\'y continued the multi-scale analysis for the same system, by considering
the effects of a low stratification (without the centrifugal force term, yet), see \cite{F-N_AMPA}, \cite{F-N_CPDE}.
In this context, we also mention \cite{F_MA} dealing with the Navier-Stokes-Korteweg system, \cite{F-L-N} which is, to the best of our knowledge, the only work concerning strongly stratified fluids
(the result holds for well-prepared data only), and \cite{F_2019} which deals with the case where the Mach number is large with respect to the Rossby number. 

The analysis for models presenting also heat transfer is much more recent, and has begun in work \cite{K-M-N} by Kwon, Maltese and Novotn\'y. In that paper, the authors considered a multi-scale
problem for the full Navier-Stokes-Fourier system with Coriolis and gravitational forces (in particular, $F=0$ therein). Notice that the scaling adopted in \cite{K-M-N} consists in taking
$$
Fr=\veps^n\,,\qquad\qquad\mbox{ with }\qquad m/2>n\geq1\,.
$$
In particular, $m$ has to be taken strictly larger than $2$, and the case $n=m/2$ was left open. Similar restrictions on the parameters can be found in \cite{F-N_CPDE}
for the barotropic model, and have to be ascribed to the techniques used for proving convergence, which are based on relative energy/relative entropy estimates
(notice that an even larger restriction for $m$ appears in \cite{F-G-GV-N}). On the other hand, it is worth
noticing that the relative energy methods allow to get a precise rate of convergence and to consider also inviscid and non-diffusive limits (one does not dispose of a uniform bound
on $\nabla_x\vtheta$ and on $\nabla_x\vec{u}$).
The isotropic scaling for the full system (i.e. the case $m=1$) was handled in the subsequent work \cite{K-N} by Kwon and Novotn\'y, by resorting to similar techniques of analysis
(see also \cite{K-L-S} for the case of the compressible MHD system in $2$-D). Notice however
that, in that work, the gravitational term is not penalised at all.

\medbreak
The main motivation of the present work is to shed some light on the multi-scale problem, by focusing on the full Navier-Stokes-Fourier system introduced in \eqref{eq_i:NSF}.
Our main concern is to remove the various restrictions on the different parameters, which appear to be a purely technical artefact. 
More precisely, as pointed out in \eqref{eq_i:scales}, we will consider the whole range of values of $m\geq1$, and we will perform the somehow critical choice $n=m/2$ for the Froude number
(see also \cite{F-N}, \cite{F-Scho} in this respect). Of course, we are still in a regime of low stratification, since $Ma/Fr\ra0$, but having $Fr=\sqrt{Ma}$ 
allows us to capture some additional qualitative properties on the limit dynamics, with respect to previous works.
In addition, we will add to the system the centrifugal force term $\nabla_x F$ (in the spirit of \cite{F-G-GV-N}), which is a source of technical troubles, due to its unboundedness.
Let us now comment all these issues in detail.

First of all, in absence of the centrifugal force, namely when $F=0$, we are able to perform incompressible, mild stratification and fast rotation limits for the \emph{whole range}
of values of $m\geq 1$, in the framework of \emph{finite energy weak solutions} to the Navier-Stokes-Fourier system \eqref{eq_i:NSF} and for general \emph{ill-prepared initial data}.
In the case $m>1$, the incompressibility and stratification effects are predominant with respect to the Coriolis force: then we prove convergence to the well-known \emph{Oberbeck-Boussinesq system},
giving a rigorous justification to this approximate model in the context of fast rotating fluids.
We point out that the target velocity field is $2$-dimensional, according to the celebrated Taylor-Proudman theorem in geophysics: in the limit of high rotation, the fluid motion tends to behave
like planar, it takes place on planes orthogonal to the rotation axis (i.e. horizontal planes in our model) and is essentially constant along the vertical direction. We refer to
\cite{C-R}, \cite{Ped} and \cite{Val} for more details on the physical side. Notice however that, although the limit dynamics is purely horizontal, the limit density and temperature variations,
$R$ and $\Theta$ respectively, appear to be stratified: this is the main effect of taking $n=m/2$ for the Froude number in  \eqref{eq_i:scales}.
This is also the main qualitative property which is new here, with respect to the previous studies.

When $m=1$, instead, all the forces act at the same scale, and then they balance each other asymptotically for $\veps\ra0$. As a result, the limit motion is described by the so-called
\emph{quasi-geostrophic equation} for a suitable function $q$, which is linked to $R$ and $\Theta$ (respectively, the target density and temperature variations) and to the gravity,
and which plays the role of a stream function for the limit velocity field. This quasi-geostrophic equation is coupled with a scalar transport-diffusion equation for a new quantity
$\Upsilon$, mixing  $R$ and $\Theta$. This is in the spirit of the result in \cite{K-N}, but once again, here we capture also gravitational effects in the limit, so that we cannot say
anymore that $R$ and $\Theta$ (and then $\Upsilon$) are horizontal; on the contrary, and somehow surprisingly, $q$ and the target velocity $\vec U$ are horizontal.

At this point, let us make a couple of remarks. First of all, we mention that, as announced above, we are able to add to the system the effects of the centrifugal force $\nabla_x F$.
Unfortunately, in this case the restriction $m\geq 2$ appears (which is still less severe than the ones imposed in \cite{F-G-GV-N}, \cite{F-N_CPDE} and \cite{K-M-N}).
However, we show that such a restriction is not of technical nature, but is now structural (see Proposition \ref{p:target-rho_bound} and comments after Proposition \ref{p:prop_5.2}).
As a matter of fact, the problem lies in the analysis of the static states, and it is not clear at all, for instance, whether or not the techniques of \cite{F_2019} might be of any help here.
The result for $F\neq 0$ is analogous to the one presented above for the case $F=0$ and $m>1$: when $m>2$, the anisotropy of scaling is too large in order to
see any effect due to $F$ in the limit, and no qualitative differences will appear with respect to the instance when $F=0$; when $m=2$, instead, additional terms, related to $F$,
will appear in the Oberbeck-Boussinesq system. In any case, the analysis will be considerably more complicated,
since $F$ is not bounded in $\Omega$ (defined in \eqref{eq:domain} above) and this will demand an additional localisation procedure (already employed in \cite{F-G-GV-N}).

We also point out that the classical existence theory of finite energy weak-solutions for \eqref{eq_i:NSF} requires the physical domain to be a Lipschitz \emph{bounded} subset of $\R^3$
 (see \cite{F-N} for a comprehensive study, see \cite{Poul} for the case of Lipschitz regularity). The theory was later extended
in \cite{J-J-N} to cover the case of unbounded domains, and this might appear suitable for us in view of \eqref{eq:domain}.
Nonetheless, the notion of weak solutions developed in \cite{J-J-N} is somehow milder than the classical one (the authors speak in fact of \emph{very weak solutions}), inasmuch as the
usual weak formulation of the entropy balance, i.e. the third equation in \eqref{eq_i:NSF}, has to be replaced by an inequality in the sense of distributions.
Now, such a formulation is not convenient for us, because, when deriving the system of acoustic-Poincar\'e waves (see more details about the proof here below), we need to combine the mass conservation
and the entropy balance equations together. In particular, this requires to have true equalities, satisfied in the (classical) weak sense.
In order to overcome this problem, we resort to the technique of \emph{invading domains} (see e.g. Chapter 8 of \cite{F-N}, \cite{F-Scho} and \cite{WK}):
namely, for each $\veps\in\,]0,1]$, we solve system \eqref{eq_i:NSF}, with the choice \eqref{eq_i:scales}, in a bounded Lipschitz domain $\Omega_\veps$, where $\big(\Omega_\veps\big)_\veps$ converges
(in a suitable sense) to $\Omega$ when $\veps\ra0$, with a rate higher than the wave propagation speed (which is proportional to $\veps^{-m}$).
Such an ``approximation procedure'' will need some extra work.

In order to prove our results, and get the improvement on the values of the different parameters, we propose a unified approach, which actually works both for the case $m>1$
(allowing us to treat the anisotropy of scaling quite easily) and for the case $m=1$ (allowing us to treat the more complicate singular perturbation operator).
This approach is based on \emph{compensated compactness} arguments, firstly employed by Lions and Masmoudi in \cite{L-M} for dealing with the incompressible limit of the barotropic
Navier-Stokes equations, and later adapted  by Gallagher and Saint-Raymond in \cite{G-SR_2006} to the case of fast rotating (incompressible homogeneous) fluids. More recent applications
of that method in the context of geophysical flows can be found in \cite{F-G-GV-N}, \cite{F_JMFM}, \cite{Fan-G} and \cite{F_2019}.
The basic idea is to use the structure of the equations, in order to find special algebraic cancellations and relations which, in turn, allow to take the limit in some non-linear quantities,
namely (in our case) the Coriolis and convective terms. Still, in dealing with the latter term, one cannot avoid the presence of bilinear expressions: for taking the limit,
some strong convergence properties are required. These strong convergence properties are by no means evident, because the singular terms are responsible for strong
time oscillations (the so-called acoustic-Poincar\'e waves) of the solutions, which may finally prevent the convergence of the non-linearities.
Nonetheless, a fine study of the system for acoustic-Poincar\'e waves actually reveals compactness (for any $m\geq1$ if $F=0$, for $m\geq 2$ if $F\neq 0$) of a special quantity
$\g_\veps$, which combines (roughly speaking) the vertical averages of the momentum $\vec{V}_\veps=\vrho_\veps\vec{u}_\veps$ (of its vorticity, in fact) and of another function
$Z_\veps$, obtained as a linear combination of density and temperature variations. Similar compactness properties have been highlighted in \cite{Fan-G} for incompressible density-dependent fluids
in $2$-D (see also \cite{C-F}), and in \cite{F_2019} for treating a multi-scale problem at ``large'' Mach numbers.
In the end, the strong convergence of $\big(\g_\veps\big)_\veps$ turns out to be enough to take the limit in the convective term,
and to complete the proof of our results.

To conclude this part, let us mention that we expect the same technique to enable us to treat also the case $m=1$ and $F\neq0$ (this was the case in \cite{F-G-GV-N}, for barotropic flows).
Nonetheless, the presence of heat transfer deeply complicates the wave system, and new technical difficulties arise in the study of the singular perturbation operator and in the analysis
of the convective term. For those reasons, here we are not able to handle that case, which still remains open.

\medbreak
Let us now give an overview of the paper.
In Section \ref{s:result} we collect our assumptions and we state our main results. In Section \ref{s:sing-pert} we study the singular perturbation part of the equations,
stating uniform bounds on our family of weak solutions and establishing constraints that the limit points have to satisfy. Section \ref{s:proof} is devoted to the proof of the convergence
result for $m\geq2$ and $F\neq0$. In the last Section \ref{s:proof-1}, we prove the convergence result for $m=1$ and $F=0$, but the same argument shows convergence also for any $m>1$,
in absence of the centrifugal force.

\paragraph*{Some notation and conventions.} \label{ss:notations}

Let $B\subset\R^n$. Throughout the whole text, the symbol $\bbbone_B$ denotes the characteristic function of $B$.
The symbol $C_c^\infty (B)$ denotes the space of $\infty$-times continuously differentiable functions on $\R^n$ and having compact support in $B$. The dual space $\mc D^{\prime}(B)$ is the space of
distributions on $B$. 
Given $p\in[1,+\infty]$, by $L^p(B)$ we mean the classical space of Lebesgue measurable functions $g$, where $|g|^p$ is integrable over the set $B$ (with the usual modifications for the case $p=+\infty$).
We use also the notation $L_T^p(L^q)$ to indicate the space $L^p\big([0,T];L^q(B)\big)$ with $T>0$.
Given $k \geq 0$, we denote by $W^{k,p}(B)$ the Sobolev space of functions which belongs to $L^p(B)$ together with all their derivatives up to order $k$. When $p=2$, we alternately use the
notation $W^{k,2}(B)$ and  $H^k(B)$.
Moreover, we denote by $\mathcal{D}^{k,p}(B)$ the corresponding homogeneous Sobolev spaces, i.e. $\mathcal{D}^{k,p}(B) = \{ g \in L^1_{\rm loc}(B)\, : \, D^\alpha g \in L^p(B),\ |\alpha| = k \}$.
Recall that $\mc D^{k,p}$ is the completion of $C^\infty_c(\overline{B})$ with respect to the $L^p$ norm of gradients.
The symbol $\mc{M}^+(B)$ denotes the cone of non-negative Borel measures on $B$. For the sake of simplicity, we will omit from the notation the set $B$, that we will explicitly point out if needed.

In the whole paper, the symbols $c$ and $C$ will denote generic multiplicative constants, which may change from line to line, and which do not depend on the small parameter $\veps$.
Sometimes, we will explicitly point out the quantities that these constants depend on, by putting them inside brackets.

Let $\big(f_\veps\big)_{0<\veps\leq1}$ be a sequence of functions in a normed space $Y$. If this sequence is bounded in $Y$,  we use the notation $\big(f_\veps\big)_{\veps} \subset Y$.
 
\medbreak
As we will see below, one of the main features of our asymptotic analysis is that the limit-flow
will be \emph{two-dimensional} and \emph{horizontal} along the plane orthogonal to the rotation axis.
Then, let us introduce some notation to describe better this phenomenon.

Let $\Omega$ be a domain in $\R^3$. We will always decompose $x\in\Omega$
into $x=(x^h,x^3)$, with $x^h\in\R^2$ denoting its horizontal component. Analogously,
for a vector-field $v=(v^1,v^2,v^3)\in\R^3$, we set $v^h=(v^1,v^2)$ and we define the differential operators
$\nabla_h$ and $\div_{\!h}$ as the usual operators, but acting just with respect to $x^h$.
In addition, we define the operator $\nabla^\perp_h\,:=\,\bigl(-\d_2\,,\,\d_1\bigr)$.
Finally, the symbol $\h$ denotes the Helmholtz projector onto the space of solenoidal vector fields in $\Omega$, while $\h_h$ denotes the Helmholtz projection in $\R^2$.
Observe that, in the sense of Fourier multipliers, one has $\h_h\vec f\,=\,-\nabla_h^\perp(-\Delta_h)^{-1}\curlh\vec f$.

Moreover, since we will deal with a periodic problem in the $x^{3}$-variable, we also introduce the following decomposition: for a vector-field $X$, we write
\begin{equation} \label{eq:decoscil}
X(x)=\langle X\rangle (x^{h})+\widetilde{X}(x)\quad\qquad
 \text{ with }\quad \langle X\rangle(x^{h})\,:=\,\frac{1}{\left|\T\right|}\int_{\T}X(x^{h},x^{3})\, dx^{3}\,,
\end{equation}
where $\mbb{T}^1\,:=\,[-1,1]/\sim$ is the one-dimensional flat torus (here $\sim$ denotes the equivalence relation which identifies $-1$ and $1$)
and $\left|\T\right|$ denotes its Lebesgue measure.
Notice that $\widetilde{X}$ has zero vertical average, and therefore we can write $\widetilde{X}(x)=\d_{3}\widetilde{Z}(x)$ with $\widetilde{Z}$ having zero vertical average as well.

\subsection*{Acknowledgements}
{\small 
The work of the second author has been partially supported by the LABEX MILYON (ANR-10-LABX-0070) of Universit\'e de Lyon, within the program ``Investissement d'Avenir''
(ANR-11-IDEX-0007), by the project BORDS (ANR-16-CE40-0027-01) and by the project SingFlows (ANR-18-CE40-0027), all operated by the French National Research Agency (ANR). 
}

\section{Setting of the problem and main results} \label{s:result}

In this section, we formulate our working hypotheses (see Subsection \ref{ss:FormProb}) and we state our main results
(in Subsection \ref{ss:results}).

\subsection{Formulation of the problem} \label{ss:FormProb}

In this Subsection, we present the rescaled Navier-Stokes-Fourier system with Coriolis, centrifugal and gravitational forces, which we are going to consider in our study.
We formulate the main working hypotheses, and recall the result in \cite{F-N} which guarantees the existence of \emph{finite energy weak solutions}.

The most part of this subsection is somehow classical: unless otherwise specified, we refer to \cite{F-N} for details.
Paragraph \ref{sss:equilibrium} contains some original contributions: there, we will establish fundamental properties for the equilibrium states, under our hypotheses on
the specific form of the centrifugal and gravitational forces.

 \subsubsection{Primitive system}\label{sss:primsys}
To begin with, let us  introduce the ``primitive system'', i.e. a rescaled compressible Navier-Stokes-Fourier system with small Mach $Ma$, Rossby $Ro$ and Froude $Fr$ numbers.
In particular, given a small parameter $\veps\in\,]0,1]$ which we will let go to $0$, we introduce the scaling
$$
Ma = \ep^m\,,\quad Fr = \ep^{m/2}\,,\quad Ro = \ep\,,
$$
for some $m\geq1$. Our system consists of
the continuity equation (conservation of mass), the momentum equation, the entropy balance and the total energy balance: respectively,
\begin{align}
&	\partial_t \vre + \div (\vre\ue)=0\,, \label{ceq}\tag{NSF$_{\ep}^1$} \\
&	\partial_t (\vre\ue)+ \div(\vre\ue\otimes\ue) + \frac{1}{\ep}\,\e_3 \times \vre\ue +    \frac{1}{\ep^{2m}} \nabla_x p(\vre,\tem) 
	= \label{meq}\tag{NSF$_{\ep}^2$} \\
&	\qquad\qquad\qquad\qquad\qquad\qquad\qquad\qquad
=\div \mbb{S}(\tem,\nabla_x\ue) + \frac{\vre}{\ep^2} \nabla_x F  + \frac{\vre}{\ep^m} \nabla_x G\,, \nonumber \\
&	 \partial_t \bigl(\vre s(\vre, \tem)\bigr)  + \div \bigl(\vre s (\vre,\tem)\ue\bigr) + \div\left(\frac{\q(\tem,\nabla_x \tem )}{\tem} \right)
	= \sigma_\ep\,, \label{eiq}\tag{NSF$_{\ep}^3$} \\
&	\frac{d}{dt} \int_{\Omega_\veps} 
	\left( \frac{\ep^{2m}}{2} \vre|\ue|^2 +  \vre e(\vre,\tem) - {\ep^m} \vre G - {\ep^{2(m-1)}} \vre F \right) dx = 0\,. \label{eeq}\tag{NSF$_{\ep}^4$}
	\end{align}
The unknowns are the fluid mass density $\vre=\vre(t,x)\geq0$ of the fluid, its velocity field $\ue=\ue(t,x)\in\R^3$ and
its absolute temperature $\tem=\tem(t,x)\geq0$, $t\in \;  ]0,T[$ , $x\in \Omega_\veps$ which fills, for $\veps \in \; ]0,1]$ fixed, the bounded domain
\begin{equation}\label{dom}
	\Omega_\ep  :=    {B}_{L_\veps} (0) \times\,]0,1[\;,\qquad\qquad\mbox{ where }\qquad L_\veps\,:=\,\frac{1}{\ep^{m+ \delta}}\,L_0
	\end{equation}
for $\delta >0$ and for some $L_0>0$ fixed, and where we have denoted by ${B}_{l}(x_0)$  the Euclidean ball of center $x_0$ and radius $l$ in  $\R^2$. Namely, roughly speaking we have the property
$$\Omega_\ep\, \longrightarrow \, \Omega := \R^2 \times\,]0,1[\, \quad \mbox{ as } \ep \to 0\,.$$ 
The pressure $p$, the specific internal energy 
$e$ and the specific entropy $s$ are given scalar valued functions of $\vr$ and $\temp$ which are related through 
Gibbs' equation	
	\begin{equation}\label{gibbs}
	\temp D s = D e + p D \left( \frac{1}{\vr}\right),	
	\end{equation}
where the symbol $D$ stands for the differential with respect to the variables $\vr$ and $\vartheta$. The viscous stress tensor in \eqref{meq} is given by Newton's rheological law
	\begin{equation}\label{S}
	\mbb{S}(\tem,\nabla_x \ue) = \mu(\tem)\left( \nabla_x\ue + \nabla_x^T \ue  - \frac{2}{3}\div \ue \tens{Id} \right)
	+ \eta(\tem) \div\ue \tens{Id}\,,
	\end{equation}
for two suitable coefficients $\mu$ and $\eta$ (we refer to Paragraph \ref{sss:structural} below for the precise hypotheses), and the entropy production rate $\sigma_\ep$ in \eqref{eiq} satisfies 
	\begin{equation}\label{ss}
	 \sigma_\ep \geq \frac{1}{\tem} \left({\ep^{2m}} \mbb{S}(\tem,\nabla_x\ue) : \nabla_x \ue
	 - \frac{\q(\tem,\nabla_x \tem )\cdot \nabla_x \tem}{\tem}  \right). 
	\end{equation}
The heat flux $\q$ in \eqref{eiq} is determined by Fourier's law
	\begin{equation}\label{q}
	\q(\tem,\nabla_x \tem)= - \kappa(\tem) \nabla_x \tem ,
	\end{equation}
where $\kappa>0$ is the heat-conduction coefficient. Finally, the term $\e_3\times\vrho_\veps\vec u_\veps$ takes into account the (strong) Coriolis force acting on the fluid.

Next, let us turn our attention to centrifugal and gravitational forces, namely $F$ and $G$ respectively.	
For the existence theory of weak solutions to our system, it would be enough to assume $F\in W^{1,\infty}_{\rm loc}\bigl(\Omega\bigr)$ to satisfy
$$ 
	F(x) \geq 0, \qquad F(x^1,x^2,- x^3) = F(x^1,x^2,x^3), 
	\qquad
	 | \nabla_x F(x) | \leq c(1 + |x^h| )\qquad\qquad \mbox{ for all } x\in \Omega\,,
$$
and $G\in W^{1,\infty}\bigl(\Omega\bigr)$.
However, for the sequel (see Paragraph \ref{sss:equilibrium} below),
it is useful to give a precise form of $F$ and $G$: motivated by physical considerations, we assume that they are of the form
	\begin{equation}\label{assFG}
	F(x) = \left|x^h\right|^2\qquad\mbox{ and }\qquad G(x)= -x^3\,.
	\end{equation}

The system is supplemented  with \emph{complete slip boundary conditions}, namely
	\begin{equation}\label{bc1-2}
	\big(\ue \cdot \n_\veps\big) _{|\partial \Omega_\veps} = 0,
	\quad \mbox{ and } \quad
	\bigl([ \mbb{S} (\tem, \nabla_x \ue) \n_\veps ] \times \n_\veps\bigr)_{|\d\Omega_\veps} = 0\,,
	\end{equation}
where $\vec{n}_\veps$ denotes the outer normal to the boundary $\d\Omega_\veps$.
We also suppose that the boundary of physical space is \emph{thermally isolated}, i.e. one has
	\begin{equation}\label{bc3}
	\big(\q\cdot\n_\veps\big)_{|\partial {\Omega_\veps}}=0\,.
	\end{equation}	

\begin{remark} \label{r:speed-waves}
Let us notice that, as $\delta>0$ in \eqref{dom} and the speed of sound is proportional to $\ep^{-m}$, hypothesis \eqref{dom} guarantees that the part $\d B_{L_\veps}(0)\times\,]0,1[\,$
of the outer boundary $\d\Omega_\veps$ of $\Omega_\ep$ becomes irrelevant when one considers the behaviour of acoustic waves on some compact set of the physical space. We refer to
Subsections \ref{ss:acoustic} and \ref{ss:unifbounds_1} for details about this point. Here, we have set $\d B_{L_\veps}(0)\,=\,\left\{x^h\in\R^2\,:\;\left|x^h\right|^2=L_\veps^2\right\}$.
\end{remark}

\subsubsection{Structural restrictions} \label{sss:structural}
Now we need to  impose structural restrictions on the thermodynamical functions $p$, $e$, $s$ as well as on  
the diffusion coefficients $\mu$, $\eta$, $\kappa$. We start by setting, for some real number $a>0$,
	\begin{equation}\label{pp1}
	p(\vrho,\vtheta)= p_M(\vrho,\vtheta) + \frac{a}{3} \vtheta^{4}\,,\qquad\qquad\mbox{ where }\qquad p_M(\vrho,\vtheta)\,=\,\vtheta^{5/2} P\left(\frac{\vrho}{\vtheta^{3/2}}\right)\,.
	\end{equation}
The first component $p_M$ in \eqref{pp1} corresponds  to the standard molecular pressure of a general monoatomic gas, while the second one represents thermal radiation. Here above,
	\begin{equation}\label{pp2}
	P\in C^1 [0,\infty)\cap C^2(0,\infty),\qquad P(0)=0,\qquad P'(Z)>0\quad \mbox{ for all }\,Z\geq 0\,,
	\end{equation}
which in particular implies the positive compressibility condition
	\begin{equation}\label{p_comp}
	\partial_\vr  p (\vr,\temp)>0.
	\end{equation}
Additionally to (\ref{pp2}) we assume that
	\begin{equation}\label{pp3}
	0<\frac{ \frac{5}{3} P(Z) - P'(Z)Z }{ Z }< c \qquad\qquad\mbox{ for all }\; Z > 0\,.
	\end{equation}
The condition \eqref{pp3} means that the specific heat at constant volume is positive, namely 
	$\partial_\temp e(\vr,\temp)$ is positive and bounded, see below.
In view of  \eqref{pp3}, we have that $Z \mapsto P(Z) / Z^{5/3}$ is a decreasing function; additionally we assume
	\begin{equation}\label{pp4}
	\lim\limits_{Z\to +\infty} \frac{P(Z)}{Z^{5/3}} = P_\infty >0\,.
	\end{equation}

Accordingly to Gibbs' relation \eqref{gibbs}, the specific internal energy and the specific entropy can be written in the following forms:
$$
	e(\vr,\temp) = e_M(\vr,\temp) + a\frac{\temp^{4}}{\vr}\,, \quad\quad
	s(\vr,\temp)= S\left(\frac{\vr}{\temp^{3/2}}\right) + \frac{4}{3} a \frac{\temp^{3}}{\vr}\,,
$$
where we have set
	\begin{equation}\label{ss1s}
e_M(\vr,\temp)=\frac{3}{2} \frac{\temp^{5/2}}{\vr} P\left( \frac{\vr}{\temp^{3/2}} \right)\qquad\mbox{ and }\qquad
S'(Z) = -\frac{3}{2} \frac{ \frac{5}{3} P(Z) - Z P'(Z)}{Z^2}\quad \mbox{ for all }\, Z>0\,.
	\end{equation}

The diffusion coefficients $\mu$ (shear viscosity), $\eta$ (bulk viscosity) and $\kappa$ (heat conductivity)  are assumed to be continuously differentiable 
functions of the temperature  $\temp \in [0,\infty[\,$, satisfying the following  growth conditions for all $\temp\geq 0$:
	\begin{equation}\label{mu}
	0<\underline\mu(1+\temp) \leq \mu(\temp) \leq \overline\mu (1 + \temp),
\quad
	0 \leq \eta(\temp) \leq \overline\eta(1 + \temp), 
\quad
	0 <  \underline\kappa (1 + \temp^3) \leq \kappa(\temp) \leq \overline\kappa(1+\temp^3),
	\end{equation}
 where $\underline\mu$, $\overline\mu$, $\overline\eta$, $\underline\kappa$ and $\overline\kappa$ are positive constants. Let us remark that the above assumptions may be not optimal from the point of view of the existence theory.

\subsubsection{Analysis of the equilibrium states} \label{sss:equilibrium}

For each scaled (NSF)$_\ep$ system, the so-called \emph{equilibrium states}  consist of static density $\vret$ and constant temperature distribution  $\tems>0$
satisfying
	\begin{equation}\label{prF}
\nabla_x p(\vret,\tems) = \ep^{2(m-1)} \vret \nabla_x F + \ep^m  \vret  \nabla_x G \quad \mbox{ in } \Omega.
	\end{equation}	
We point out that, for later use, it is convenient to state \eqref{prF} on whole set $\Omega$. 	
We also notice that, \textsl{a priori}, it is not known that the target temperature has to be constant: this is a consequence that entropy production
rate $\sigma_\ep$ has to be  kept  small  and then $\nabla_x \temp_\ep$  needs to vanish as $\ep \to 0$ (see Section 4.2 of \cite{F-N} for more comments about this).

%
%

Equation \eqref{prF} identifies $\wtilde{\vrho}_\veps$ up to an additive constant: normalizing it to $0$, and taking the target density to be $1$, we get
\begin{equation} \label{eq:target-rho}
  \Pi(\vret)\,=\,\wtilde{F}_\veps\,:=\, \ep^{2(m-1)} F + \ep^m G\,,\qquad\qquad \mbox{ where }\qquad 
\Pi(\vrho) = \int_1^{\vrho} \frac{\partial_\varrho p(z,\oline{\vtheta})}{z} {\rm d}z\,.
\end{equation}

From this relation, we immediately get the following properties:
\begin{itemize}
 \item[(i)] when $m>1$, or $m=1$ and $F=0$, for any $x\in\Omega$ one has $\wtilde{\vrho}_\veps(x)\longrightarrow 1$ in the limit $\veps\ra0$;
 \item[(ii)] for $m=1$ and $F\neq0$, $\bigl(\wtilde{\vrho}_\veps\bigr)_\veps$ converges pointwise to $\wtilde{\vrho}$, where
$$
\wtilde\vrho\quad\mbox{ is a solution of the problem}\qquad \Pi\bigl(\wtilde{\vrho}(x)\bigr)\,=\,F(x)\,,\ \mbox{ with }\ x\in\Omega\,.
$$
In particular,
$\wtilde{\vrho}$ is non-constant, of class $C^2(\Omega)$ (keep in mind assumptions \eqref{pp2} and \eqref{p_comp} above) and it depends just on the horizontal variables due to \eqref{assFG}.
\end{itemize}

We are now going to study more in detail the equilibrium densities $\wtilde\vrho_\veps$.
In order to keep the discussion as general as possible, we are going to consider both cases (i) and (ii) listed above, even though our results will concern only case (i).

The first problem we have to face is that the right-hand side of \eqref{eq:target-rho} may be negative: this means that $\wtilde{\varrho}_\veps$ can go below the value $1$ in some regions of $\Omega$.
Nonetheless, the next statement excludes the presence of vacuum.
\begin{lemma} \label{l:target-rho_pos}
Let the  centrifugal force $F$ and the gravitational force $G$ be given by \eqref{assFG}.
Let $\bigl(\wtilde{\vrho}_\veps\bigr)_{0<\veps\leq1}$ be a family of static solutions to equation \eqref{eq:target-rho}
on $\Omega$.

Then, there exist an $\veps_0>0$ and a $0<\rho_*<1$ such that $\wtilde{\vrho}_\veps\geq\rho_*$ for all $\veps\in\,]0,\veps_0]$
and all $x\in\Omega$.
\end{lemma}

\begin{proof}
Let us consider the case $m>1$ (hence $F\neq0$) first. Suppose, by contradiction, that there exists a sequence $\bigl(\veps_n,x_n\bigr)_n$ such that
$0\,\leq\,\wtilde{\vrho}_{\veps_n}(x_n)\,\leq\,1/n$. We observe that the sequence $\bigl(x_n\bigr)_n$ cannot be bounded. Indeed, relation
\eqref{eq:target-rho}, computed on $\wtilde{\vrho}_{\veps_n}(x_n)$, would immediately imply that $\wtilde{\vrho}_{\veps_n}(x_n)$ should
rather converge to $1$.
In any case, since $1/n<1$ for $n\geq2$ and $x^3\in\; ]0,1[\, $, we deduce that
$$
-\,(\veps_n)^m\,\leq\,\wtilde{F}_{\veps_n}(x_n)\,=\,(\veps_n)^{2(m-1)}\,|\,(x_n)^h\,|^2\,-\,(\veps_n)^m\,(x_n)^3\,<\,0\,,
$$
which in particular implies that $\wtilde{F}_{\veps_n}(x_n)$ has to go to $0$ for $\veps\ra0$. As a consequence, since
$\Pi(1)=0$, by the mean value theorem and \eqref{eq:target-rho} we get
$$
\wtilde{F}_{\veps_n}(x_n)\,=\,\Pi\bigl(\wtilde{\vrho}_{\veps_n}(x_n)\bigr)\,=\,\Pi'(z_n)\,\bigl(\wtilde{\vrho}_{\veps_n}(x_n)-1\bigr)\,=\,
\frac{\d_\varrho p\bigl(z_n,\oline{\vtheta}\bigr)}{z_n}\,\bigl(\wtilde{\vrho}_{\veps_n}(x_n)-1\bigr)\,\longrightarrow\,0\,,
$$
for some $z_n\in\,]\wtilde{\vrho}_{\veps_n}(x_n),1[\,\subset\,]0,1[\,$, for all $n\in\N$. In turn, this relation,
combined with \eqref{p_comp}, implies that $\wtilde{\vrho}_{\veps_n}(x_n)\to1$,
which is in contradiction with the fact that it has to be $\leq1/n$ for any $n\in\N$.

The case $m=1$ and $F=0$ can be treated in a similar way. Let us now assume that $m=1$ and $F\neq0$: relation \eqref{eq:target-rho} in this case becomes
\begin{equation} \label{eq:target_m=1}
\Pi\bigl(\wtilde{\vrho}_\veps(x)\bigr)\,=\,|\,x^h\,|^2\,-\,\veps\,x^3\,.
\end{equation}
We observe that the right-hand side of this identity is negative on the set
$\left\{0\,\leq\,|\,x^h\,|^2\,\leq\,\veps\,x^3\right\}$. By definition \eqref{eq:target-rho}, this is equivalent to having $\wtilde{\vrho}_\veps(x)\leq1$.

In particular, the smallest value of $\wtilde{\vrho}_\veps(x)$
is attained for $x^h=0$, $x^3=1$, for which $\Pi\bigl(\wtilde{\vrho}_\veps(0,0,1)\bigr)=-\veps$.
On the other hand, fixed a $x^0_\veps$ such that $|\,(x^0_\veps)^h\,|^2=\veps$ and $(x^0_\veps)^3=1$, we have
$\Pi\bigl(\wtilde{\vrho}_\veps(x_\veps^0)\bigr)=0$, and then $\wtilde{\vrho}_\veps(x^0_\veps)=1$.
Therefore, by mean value theorem again we get
\begin{align*}
-\,\veps\;=\;\Pi\bigl(\wtilde{\vrho}_\veps(0,0,1)\bigr)\,-\,\Pi\bigl(\wtilde{\vrho}_\veps(x^0_\veps)\bigr)\,&=\,
\frac{\d_\varrho p\bigl(\wtilde{\vrho}_\veps(y_\veps),\oline{\vtheta}\bigr)}{\wtilde{\vrho}_\veps(y_\veps)}\,
\bigl(\wtilde{\vrho}_\veps(0,0,1)\,-\,\wtilde{\vrho}_\veps(x^0_\veps)\bigr) \\
 &=\,\frac{\d_\varrho p\bigl(\wtilde{\vrho}_\veps(y_\veps),\oline{\vtheta}\bigr)}{\wtilde{\vrho}_\veps(y_\veps)}\,
\bigl(\wtilde{\vrho}_\veps(0,0,1)\,-\,1\bigr)\,
\end{align*}
for some suitable point $y_\veps=\bigl((x_\veps)^h,1\bigr)$ lying on the line connecting $(0,0,1)$ with $x^0_\veps$.

From this equality and the structural hypothesis \eqref{p_comp}, since $\wtilde{\vrho}_\veps(0,0,1)\,-\,1<0$ (due to the fact that $\Pi\bigl(\wtilde{\vrho}_\veps(0,0,1)\bigr)<0$),
we deduce that $\wtilde{\vrho}_\veps(y_\veps)>0$ for all $\veps>0$. On the other hand, \eqref{eq:target_m=1} says that,
for $x^3$ fixed, the function $\Pi\circ\wtilde{\vrho}_\veps$ is radially increasing on $\R^2$: then, in particular
$\wtilde{\vrho}_\veps(y_\veps)\leq\wtilde{\vrho}_\veps(x^0_\veps)=1$.

Finally, thanks to these relations and the regularity properties \eqref{pp1} and \eqref{pp2}, we see that
$$
\wtilde{\vrho}_\veps(0,0,1)\,=\,1\,-\,\veps\,
\frac{\wtilde{\vrho}_\veps(y_\veps)}{\d_\varrho p\bigl(\wtilde{\vrho}_\veps(y_\veps),\oline{\vtheta}\bigr)}
$$
remains strictly positive, at least for $\veps$ small enough.
\qed
\end{proof}

\medbreak
For simplicity, and without loss of any generality, we assume from now on that $\veps_0=1$ in Lemma \ref{l:target-rho_pos}.

Next, denoted as above $B_l(0)$ the ball in the horizontal variables $x^h\in\R^2$ of center $0$ and radius $l>0$, we define the cylinder
$$
	\mbb B_{L} := \left\{ x\in \Omega \ : \ |x^h| < L  \right\}=B_L(0)\times\, ]0,1[\, .
$$

We can now state the next boundedness property for the family $\bigl(\wtilde{\vrho}_\veps\bigr)_\veps$.
\begin{lemma} \label{l:target-rho_bound}
Let $m\geq1$. Let $F$ and $G$ satisfy \eqref{assFG}. Then, for any $l>0$, there exists a constant $C(l)>1$ such that for all $\veps\in\,]0,1]$ one has
\begin{equation} \label{est:target-rho_in}
\wtilde{\vrho}_\veps\,\leq\,C(l)\qquad\qquad\mbox{ on }\qquad \oline{\mbb{B}}_{l}\,.
\end{equation}

If $F=0$, then there exists a constant $C>1$ such that, for all $\veps\in\,]0,1]$ and all $x\in\Omega$, one has $\left|\wtilde\vrho_\veps(x)\right|\leq C$.
\end{lemma}

\begin{proof}
Let us focus on the case $m>1$ and $F\neq 0$ for a while.
In order to see \eqref{est:target-rho_in}, we proceed in two steps. First of all, we fix $\veps$ and we show that $\wtilde{\varrho}_\veps$
is bounded in the previous set. Assume it is not: there exists a sequence $\bigl(x_n\bigr)_{n}\subset \oline{\mbb{B}}_{l}$
such that $\wtilde{\varrho}_\veps(x_n)\geq n$. But then, thanks to hypothesis \eqref{pp3}, we can write
$$
\Pi\bigl(\wtilde{\varrho}_\veps(x_n)\bigr)\,\geq\,\int^n_1\frac{\d_\varrho p(z,\oline{\vtheta})}{z}{\rm\,d}z\,\geq\,C(\oline{\vtheta})\,
\int^{n/\oline{\vtheta}^{3/2}}_{1/\oline{\vtheta}^{3/2}}\frac{P(Z)}{Z^2}{\rm\,d}Z\,,
$$
and, by use of \eqref{pp4}, it is easy to see that the last integral diverges to $+\infty$ for $n\ra+\infty$. On the other hand,
on the set $\oline{\mbb{B}}_{l}$, the function $\wtilde{F}_\veps$ is uniformly bounded by the constant
$l^2+1$, and, recalling formula \eqref{eq:target-rho}, these two facts are in contradiction one with other.

So, we have proved that, $\wtilde{\varrho}_\veps\,\leq\,C(\veps,l)$ on the set $\oline{\mbb{B}}_{l}$.
But, thanks to point (i) below \eqref{eq:target-rho}, the pointwise convergence of $\wtilde{\vrho}_\veps$ to $1$ becomes uniform in the
previous set, so that the constant $C(\veps,l)$ can be dominated by a new constant $C(l)$, just depending on the fixed $l$. 

Let us now take $m=1$ and $F\neq0$. 
We start by observing that, again, the following property holds true:
for any $\veps$ and any $l>0$ fixed, one has $\wtilde{\vrho}_\veps\,\leq\,C(\veps,l)$ in $\oline{\mbb{B}}_{l}$.
Furthermore, by point (ii) below \eqref{eq:target-rho} we have that $\wtilde\vrho\in C^2(\Omega)$, and then $\wtilde\vrho$ is locally bounded:
for any $l>0$ fixed, we have $\wtilde{\vrho}\,\leq\,C(l)$ on the set $\oline{\mbb{B}}_{l}$. On the other hand, the pointwise convergence of
$\bigl(\wtilde{\vrho}_\veps\bigr)_\veps$ towards $\wtilde\vrho$ becomes uniform on the compact set $\oline{\mbb{B}}_{l}$:
gluing these facts together, we infer that, in the previous bound for $\wtilde{\vrho}_\veps$, we can replace $C(\veps,l)$ by a constant $C(l)$ which is uniform in $\veps$.

Let us now consider the case $F=0$, for any value of the parameter $m\geq1$. In this case, relation \eqref{eq:target-rho} becomes
$$
\Pi\big(\wtilde\vrho_\veps\big)\,=\,\veps^m\,G\qquad\Longrightarrow\qquad \left|\Pi\big(\wtilde\vrho_\veps\big)\right|\,\leq\,C\quad\mbox{ in }\;\Omega\,.
$$
At this point, as a consequence of the structural assumptions \eqref{pp1}, \eqref{pp3} and \eqref{pp4}, we observe that
$\Pi(z)\longrightarrow+\infty$ for $z\ra+\infty$. Then, $\wtilde\vrho_\veps$ must be uniformly bounded in $\Omega$.

This completes the proof of the lemma. \qed
\end{proof}

\medbreak
We conclude this paragraph by showing some additional bounds, which will be relevant in the sequel.
\begin{proposition} \label{p:target-rho_bound}
Let $F\neq0$. For any $l>0$, on the cylinder $\oline{\mbb{B}}_{l}$ one has, for any $\veps\in\,]0,1]$:
\begin{enumerate}[(1)]
\item $
\left|\wtilde{\vrho}_\veps(x)\,-\,1\right|\,\leq\,C(l)\,\veps^m\, \text{ if }m\geq2$; 
\item $
\left|\wtilde{\vrho}_\veps(x)\,-\,1\right|\,\leq\,C(l)\,\veps^{2(m-1)}\, \text{ if }1<m<2
$;
\item $
\left|\wtilde{\vrho}_\veps(x)\,-\,\wtilde{\varrho}(x)\right|\,\leq\,C(l)\,\veps\, \text{ if }m=1$. 
\end{enumerate}

When $F=0$ and $m\geq1$, instead, one has $\left|\wtilde{\vrho}_\veps(x)\,-\,1\right|\,\leq\,C\,\veps^m$, for a constant $C>0$ which is uniform in $x\in\Omega$ and in $\veps\in\,]0,1]$.
\end{proposition}

\begin{proof}
Assume $F\neq0$ for a while. Let $m\geq2$. Thanks to the Lemma \ref{l:target-rho_bound}, the estimate on $\left|\wtilde{\vrho}_\veps(x)\,-\,1\right|$ easily
follows applying the mean value theorem to equation \eqref{eq:target-rho}, and noticing that
$$
\sup_{z\in[\rho_*,C(l)]}\left|\frac{z}{\d_\varrho p(z,\oline{\vtheta})}\right|\,<\,+\infty\,,
$$ 
on $\oline{\mbb{B}}_l$ for any fixed $l>0$.
According to the hypothesis $m\geq2$, we have $2(m-1)\geq m$. The claimed bound then follows.
The proof of the inequality for $1<m<2$ is analogous, using this time that $2(m-1)\leq m$.

In order to prove the inequality for $m=1$, we consider the equations satisfied by $\wtilde{\vrho}_\veps$ and $\wtilde{\vrho}$: we have
$$
\Pi\bigl(\wtilde{\vrho}_\veps(x)\bigr)\,=\,|\,x^h\,|^2\,-\,\veps\,x^3\qquad\qquad\mbox{ and }\qquad\qquad
\Pi\bigl(\wtilde{\vrho}(x)\bigr)\,=\,|\,x^h\,|^2\,.
$$
Now, we take the  difference and we apply the mean value theorem, finding
$$
\Pi'\big(z_\veps(x)\big)\,\big(\wtilde{\vrho}_\veps(x)\,-\,\wtilde{\varrho}(x)\big)\,=\,-\veps\,x^3\,,
$$
for some $z_\veps(x)\in\,]\wtilde{\vrho}_\veps(x),\wtilde{\varrho}(x)[\,$ (or with exchanged extreme points, depending on $x$). By Lemma \ref{l:target-rho_bound}
we have uniform (in $\veps$) bounds on the set $\oline{\mbb{B}}_{l}$, depending on $l$, for $\wtilde{\vrho}_\veps(x)$ and
$\wtilde{\varrho}(x)$: then, from the previous identity, on this cylinder we find
$$
\left|\wtilde{\vrho}_\veps(x)\,-\,\wtilde{\varrho}(x)\right|\,\leq\,C(l)\,\veps\,.
$$

The bounds in the case $F=0$ can be shown in an analogous way. We omit the details here.
The proposition is now  completely proved.
\qed
\end{proof}

\medbreak

%

From now on, we will focus on the following cases:
\begin{equation} \label{eq:choice-m}
\mbox{ either }\quad m\geq2\,,\qquad\quad \mbox{ or }\quad\qquad m\geq1\quad\mbox{ and }\quad F=0\,.
\end{equation}
Notice that in all those cases, the target density profile $\wtilde\vrho$ is constant, namely $\wtilde\vrho\equiv1$.


\subsubsection{Initial data and finite energy weak solutions} \label{sss:data-weak}

We address the singular perturbation problem described in Paragraph \ref{sss:primsys} for general \emph{ill prepared initial data}, in the framework of \emph{finite energy weak solutions},
whose theory was developed in \cite{F-N} with the suitable extension of \cite{Poul} in order to treat Lipschitz domains.
Since we work with weak solutions based on dissipation estimates and control of entropy production rate, we need to assume that the initial data are close to the equilibrium states $(\vret,\tems)$ that we have just identified.
Namely, we consider initial densities and temperatures of the following form:
	\begin{equation}\label{in_vr}
	\vrez = \vret + \ep^m \vrez^{(1)} 
	\qquad\qquad\mbox{ and }\qquad\qquad
	\temz = \tems + \ep^m \Theta_{0,\veps}\,.
	\end{equation}
For later use, let us introduce also the following decomposition of the initial densities: after setting $\wtilde r_\veps\,:=\,\bigl(\wtilde\vrho_\veps-1\bigr)/\veps^m$, we write
\begin{equation} \label{eq:in-dens_dec}
\vrho_{0,\veps}\,=\,1\,+\,\veps^m\,R_{0,\veps}\qquad\qquad\mbox{ with }\qquad
R_{0,\veps}\,=\,\vrho_{0,\veps}^{(1)}\,+\,\wtilde r_\veps\,.
\end{equation}
Notice that the $\wtilde r_\veps$'s are in fact data of the system, since they only depend on $p$, $F$ and $G$.

We suppose $\vrez^{(1)}$ and $\Theta_{0,\veps}$ to be bounded measurable functions satisfying the controls
	\begin{align}
\sup_{\veps\in\,]0,1]}\left\|  \vrez^{(1)} \right\|_{(L^2\cap L^\infty)(\Omega_\veps)}\,\leq \,c\,,\qquad 
\sup_{\veps\in\,]0,1]}\left(\left\|\Theta_{0,\veps}\right\|_{L^\infty(\Omega_\veps)}\,+\,\left\| \sqrt{\wtilde\vrho_\veps}\,\Theta_{0,\veps}\right\|_{L^2(\Omega_\veps)}\right)\,\leq\, c\,,\label{hyp:ill_data}
	\end{align}
together with the conditions
$$
\int_{\Omega_\veps}  \vrez^{(1)} \dx = 0\qquad\qquad\mbox{ and }\qquad\qquad \int_{\Omega_\veps}\Theta_{0,\veps} \dx = 0\,.
$$
As for the initial velocity fields, we will assume instead the following uniform bounds:
\begin{equation} \label{hyp:ill-vel}
 	\sup_{\veps\in\,]0,1]}\left(\left\| \sqrt{\wtilde\vrho_\veps} \vec{u}_{0,\ep} \right\|_{L^2(\Omega_\veps)}\,+\, \left\| \vec{u}_{0,\ep}  \right\|_{L^\infty(\Omega_\veps)}\right)\,  \leq\, c\,.
\end{equation}

\begin{remark} \label{r:ill_data}
In view of Lemma \ref{l:target-rho_pos}, the conditions in \eqref{hyp:ill_data} and \eqref{hyp:ill-vel} imply in particular that
$$
\sup_{\veps\in\,]0,1]}\left(\left\| \Theta_{0,\veps}\right\|_{L^2(\Omega_\veps)}\,+\,\left\| \vec{u}_{0,\ep}  \right\|_{L^2(\Omega_\veps)}\right)\,\leq\,c\,.
$$
\end{remark}


Thanks to the previous uniform estimates, up to extraction, we can assume that
\begin{equation} \label{conv:in_data}
\vrho^{(1)}_0\,:=\,\lim_{\veps\ra0}\vrho^{(1)}_{0,\veps}\;,\qquad
R_0\,:=\,\lim_{\veps\ra0}R_{0,\veps}\;,\qquad
\Theta_0\,:=\,\lim_{\veps\ra0}\Theta_{0,\veps}\;,\qquad
\vec{u}_0\,:=\,\lim_{\veps\ra0}\vec{u}_{0,\veps}\,,
\end{equation}
where we agree that the previous limits are taken in the weak-$*$ topology of $L_{\rm loc}^\infty(\Omega)\,\cap\,L_{\rm loc}^2(\Omega)$.

\medbreak


Let us specify better what we mean for \emph{finite energy weak solution} (see \cite{F-N} and \cite{Poul} for details). First of all, the equations have to be satisfied in a distributional sense:
	\begin{equation}\label{weak-con}
	-\int_0^T\int_{\Omega_\veps} \left( \vre \partial_t \varphi  + \vre\ue \cdot \nabla_x \varphi \right) \dxdt = 
	\int_{\Omega_\veps} \vrez \varphi(0,\cdot) \dx
	\end{equation}
for any $\varphi\in C^\infty_c([0,T[\,\times \overline\Omega_\veps)$;
	\begin{align}
	&\int_0^T\!\!\!\int_{\Omega_\veps}  
	\left( - \vre \ue \cdot \partial_t \vec\psi - \vre [\ue\otimes\ue]  : \nabla_x \vec\psi 
	+ \frac{1}{\ep} \, \e_3 \times (\vre \ue ) \cdot \vec\psi  - \frac{1}{\ep^{2m}} p(\vre,\tem) \div \vec\psi  \right) \dxdt \label{weak-mom} \\
	& =\int_0^T\!\!\!\int_{\Omega_\veps} 
	\left(- \mbb{S}(\vtheta_\veps,\nabla_x\vec u_\veps)  : \nabla_x \vec\psi +  \left(\frac{1}{\ep^2} \vre \nabla_x F +  \frac{1}{\ep^m} \vre \nabla_x G \right)\cdot \vec\psi \right) \dxdt 
	+ \int_{\Omega_\veps}\vrez \uez  \cdot \vec\psi (0,\cdot) \dx \nonumber
	\end{align}
for any test function $\vec\psi\in C^\infty_c([0,T[\,\times \overline\Omega_\veps; \R^3)$ such that $\big(\vec\psi \cdot \n_\veps\big)_{|\partial {\Omega_\veps}} = 0$;
	\begin{align}
	\int_0^T\!\!\!\int_{\Omega_\veps} 
	& \Bigl( - \vre s(\vre,\tem) \partial_t \varphi  -  \vre s(\vre,\tem) \ue \cdot \nabla_x \varphi \Bigr) \dxdt \label{weak-ent} \\
	& - \int_0^T\int_{\Omega_\veps}  \frac{\q(\vtheta_\veps,\nabla_x\vtheta_\veps)}{\tem} \cdot \nabla_x \varphi \dxdt - \langle \sigma_\ep; \varphi  \rangle_{ [{\cal{M}}; C]([0,T]\times \overline\Omega_\veps)}
	= \int_{\Omega_\veps} \vrez s(\vrez,\temz) \varphi (0,\cdot) \dx \nonumber
	\end{align}
for any $\varphi\in C^\infty_c([0,T[\,\times \overline\Omega_\veps)$, with $\sigma_\ep \in {\cal{M}}^+ ([0,T]\times \overline\Omega_\veps)$. 
In addition, we require that the energy identity 
\begin{align}
	\int_{\Omega_\veps}
	& \left( \frac{1}{2} \vre |\ue|^2  +  \frac{1}{\ep^{2m}} \vre e(\vre,\tem)  -  \frac{1}{\ep^2} \vre F - \frac{1}{\ep^m} \vre G \right) (t) \dx  \label{weak-eng} \\
	&=  \int_{\Omega_\veps} \left( \frac{1}{2} \vrez |\uez|^2  +  \frac{1}{\ep^{2m}} \vrez e(\vrez,\temz) -  \frac{1}{\ep^2} \varrho_{0,\ep} F  -\frac{1}{\ep^m} \varrho_{0,\ep} G  \right) \dx \nonumber
	\end{align}
holds true for almost every $t\in\,]0,T[\,$. Notice that this is the integrated version of \eqref{eeq}.

Under the previous assumptions, collected in Paragraphs \ref{sss:primsys} and \ref{sss:structural} and here above, at any \emph{fixed} value of the parameter $\veps\in\,]0,1]$,
the existence of a global in time finite energy weak solution $(\vrho_\veps,\vec u_\veps,\vtheta_\veps)$ to system (NSF)$_\veps$, related to the initial datum
$(\vrho_{0,\veps},\vec u_{0,\veps},\vtheta_{0,\veps})$, has been proved in e.g. \cite{F-N} (see Theorems 3.1 and 3.2 therein for the case
of smooth domains, see \cite{Poul} for the extension to Lipschitz domains).
Moreover, 
the following regularity of solutions $( \vre, \ue, \tem )$  can be obtained (at any fixed value of $\veps$), 
which justifies all the integrals appearing in \eqref{weak-con} to \eqref{weak-eng}: for any $T>0$ fixed, one has
	\begin{equation*}
	\vre \in C_{{\rm weak}}\big([0,T];L^{5/3}(\Omega_\ep)\big),\quad 
	\vre \in L^q\big((0,T)\times \Omega_\ep\big) \ \mbox{ for some } q>\frac{5}{3}\,,
\qquad
	\ue \in L^2\big([0,T]; W^{1,2}(\Omega_\ep;\R^3)\big)\,.
	\end{equation*}
In addition, also the mapping $t \mapsto (\vre\ue)(t,\cdot)$ is weakly continuous, and one has $(\vre)_{|t=0} = \vrez$ and $(\vre\ue)_{|t=0}= \vrez\uez$. Finally,
the absolute temperature $\tem$ is a measurable function, $\tem>0$ a.e. in $\R_+ \times \Omega_\ep$, and given any $T>0$, one has
	\begin{equation*}
	\tem \in L^2\big([0,T]; W^{1,2}(\Omega_\ep)\big)\cap L^{\infty}\big([0,T]; L^4 (\Omega_\ep)\big), \quad 
	\log \tem \in L^2\big([0,T]; W^{1,2}(\Omega_\ep)\big)\,.
	\end{equation*}

Notice that, in view of \eqref{ceq}, the total mass is conserved in time.
Since the measure of the domain $\Omega_\veps$ tends to $+\infty$, the mass conservation has to be understood in the following sense: for almost every $t\in[0,+\infty[\,$,
one has
\begin{equation} \label{eq:mass_conserv}
\int_{\Omega_\veps}\bigl(\vre(t)\,-\,\vret\bigr)\,\dx\,=\,0\,.
\end{equation}

Let us now introduce the ballistic free energy function
$$
	H_{\tems}(\vr,\temp)\,:=\,\vr \bigl( e(\vr,\temp) - \tems s(\vr,\temp)  \bigr)\,,
$$
and define the \emph{relative entropy functional} (for details, see in particular Chapters 1, 2 and 4 of \cite{F-N})
$$
\mc E\left(\rho,\theta\;|\;\wtilde\vrho_\veps,\oline\vtheta\right)\,:=\,H_{\tems}(\rho,\theta) - (\rho - \vret)\,\partial_\vrho H_{\tems}(\vret,\tems)
- H_{\tems}(\vret,\tems)\,.
$$
Combining the total energy balance \eqref{weak-eng}, the entropy equation \eqref{weak-ent} and the mass conservation \eqref{eq:mass_conserv}, we obtain 
the following total dissipation balance, for any $\veps>0$ fixed:
\begin{align}
&\hspace{-0.7cm} \int_{\Omega_\veps}\frac{1}{2}\vre|\ue|^2(t) \dx\,+\,\frac{1}{\ep^{2m}}\int_{\Omega_\veps}\mc E\left(\vrho_\veps,\vtheta_\veps\;|\;\wtilde\vrho_\veps,\oline\vtheta\right) \dx
\label{est:dissip} \\
&+\int^t_0\int_{\Omega_\veps}\left(\mbb{S}(\vtheta_\veps,\nabla_x\vec u_\veps)  : \nabla_x\vec u_\veps\,-\,\frac{\q(\vtheta_\veps,\nabla_x\vtheta_\veps)}{\tem} \cdot \nabla_x\vtheta_\veps\right)\dx\,\dt
+  \frac{\tems}{\ep^{2m}}\sigma_\ep \left[[0,t]\times \overline\Omega_\veps\right] \nonumber \\
&\qquad\qquad\qquad\qquad\qquad\qquad\qquad\qquad
\,\leq\,
\int_{\Omega_\veps}\frac{1}{2}\vrez|\uez|^2 \dx\,+\,
\frac{1}{\ep^{2m}}\int_{\Omega_\veps}\mc E\left(\vrho_{0,\veps},\vtheta_{0,\veps}\;|\;\wtilde\vrho_\veps,\oline\vtheta\right) \dx\,.
\nonumber
\end{align}

Inequality \eqref{est:dissip} will be the only tool to derive uniform estimates for the family of weak solutions we consider. As a matter of fact,
we will establish in Lemma \ref{l:initial-bound} below that, under the previous assumptions on the initial data, the quantity on the right-hand side of \eqref{est:dissip}
is uniformly bounded for any $\veps\in\,]0,1]$.

\medbreak 
To conclude this part, let us introduce an additional quantity. Since the entropy production rate is a non-negative measure, and in particular it may possess jumps, the total entropy 
$\vre s(\vre,\tem)$ may not be weakly continuous in time. To avoid 
this problem, following \cite{F-N}, we introduce a time lifting $\Sigma_\ep$ of the measure $\sigma_\ep$ by the following formula:
	\begin{equation}\label{lift0}
	\langle \Sigma_\ep , \varphi \rangle = \langle \sigma_\ep , I[\varphi] \rangle,
\quad \mbox{ where }\quad
	I [\varphi] (t,x) = \int _0^t \varphi (\tau,x) {\rm\, d}\tau\quad\mbox{for any } \varphi \in L^1(0,T; C(\overline\Omega_\veps)).
	\end{equation}
The time lifting $\Sigma_\ep$ can be identified with an abstract function 
$\Sigma_\ep \in L^{\infty}_{\rm{weak}}(0,T; {\cal{M}}^+(\overline{\Omega}_\veps))$, where the notation stands for ``weakly measurable'',
and $\Sigma_\veps$ is defined by the relation
	\begin{equation*}
	\langle \Sigma_\ep(\tau),\varphi \rangle = \lim\limits_{\delta \to 0^+} \langle \sigma_\ep , \psi_\delta\,\varphi \rangle,
\quad \mbox{ with } \quad
	\psi_\delta(t) =
	\left\{\begin{array}{cc}
	0 					& {\mbox{for }} t\in [0,\tau), \\
	\frac{1}{\delta}(t - \tau) 	& {\mbox{for }} t\in (\tau, \tau + \delta), \\
	1					& {\mbox{for }} t \geq \tau +\delta.
	\end{array}\right.
	\end{equation*}
In particular, the measure $\Sigma_\ep$ is well-defined for any $\tau\in[0,T]$, and mapping 
$\tau \to \Sigma_\ep(\tau)$ is non-increasing in the sense of measures.

Then, the weak formulation of the entropy balance can be equivalently rewritten as
	\begin{equation*} 
	\begin{split}
	& \int_{\Omega_\veps} 
	 \left[  \vre s(\vre,\tem)(\tau)\varphi(\tau)
	-   \vrez s(\vrez,\temz)\varphi(0)  \right] \dx 
	 + \langle \Sigma_\ep(\tau),\varphi(\tau) \rangle  -  \langle \Sigma_\ep(0),\varphi(0) \rangle\\
	& = \int_0^\tau  \langle \Sigma_\ep,\partial_t \varphi \rangle \dt
	+ \int_0^\tau \int_{\Omega_\veps} 
	\left( 
	 \vre s(\vre,\tem)   \partial_t \varphi  +  \vre s(\vre,\tem)\ue \cdot \nabla_x \varphi 
	 +  \frac{\q(\vtheta_\veps,\nabla_x\vtheta_\veps)}{\tem} \cdot \nabla_x \varphi 
	\right) \dxdt
	\end{split}
	\end{equation*}
for any $\varphi\in C^\infty_c([0,T]\times \overline\Omega_\veps)$, and  the mapping 
	$
	t \to \vre s(\vre,\tem)(t,\cdot) + \Sigma_\ep(t) \ 
	$
 is continuous with values in ${\cal{M}}^+(\overline{\Omega}_\veps)$, provided that 
${\cal{M}}^+$ is endowed with the weak-$*$ topology.
 

\subsection{Main results}\label{ss:results}

\medbreak
We can now state our main results. The first statement concerns the case when low Mach number effects are predominant with respect to fast rotation, i.e. $m>1$. 
For technical reasons which will appear clear in the course of the proof, when $F\neq0$ we need to take $m\geq2$. 

We also underline that the limit dynamics of $\vec{U}$ is purely horizontal (see \eqref{eq_lim_m:momentum} below) on the plane $\R^2\times \{0\}$ accordingly to the celebrated Taylor-Proudman theorem.
Nonetheless the equations that involve $R$ and $\Theta$ (see \eqref{eq_lim_m:temp} and \eqref{eq_lim_m:Boussinesq} below) depend also on the vertical variable. 
\begin{theorem}\label{th:m-geq-2}
For any $\veps\in\,]0,1]$, let $\Omega_\ep$ be the domain defined by \eqref{dom} and $\Omega = \R^2 \times\,]0,1[\,$. 
Let  $p$, $e$, $s$  satisfy  Gibbs' relation \eqref{gibbs} and structural hypothesis from \eqref{pp1} to \eqref{ss1s}, and that the diffusion coefficients $\mu$, $\eta$, $\kappa$
enjoy growth conditions  \eqref{mu}. Let $G\in W^{1,\infty}(\Omega)$ be given as in \eqref{assFG}.
Take either $m\geq 2$ and $F\in W_{loc}^{1,\infty}(\Omega)$ as in \eqref{assFG}, or $m>1$ and $F=0$. \\
For any fixed value of $\veps \in \; ]0,1]$, let initial data $\left(\vrho_{0,\veps},\vec u_{0,\veps},\vtheta_{0,\veps}\right)$ verify the hypotheses fixed in Paragraph \ref{sss:data-weak}, and let
$\left( \vre, \ue, \tem\right)$ be a corresponding weak solution to system (NSF)$_\veps$, supplemented with structural hypotheses from \eqref{S} to \eqref{q} and with boundary conditions \eqref{bc1-2} and \eqref{bc3}.
Assume that the total dissipation balance \eqref{est:dissip} is satisfied.
Let $\left(R_0,\vec u_0,\Theta_0\right)$ be defined as in \eqref{conv:in_data}.

Then one has the following convergence properties:
	\begin{align*}
	\varrho_\ep \rightarrow 1 \qquad\qquad \mbox{ in } \qquad L^{\infty}\big([0,T]; L_{\rm loc}^{2}+L_{\rm loc}^{5/3}(\Omega )\big)\,, \\
	R_\veps:=\frac{\varrho_\ep - 1}{\ep^m}  \weakstar R \qquad\qquad \mbox{ weakly-$*$ in }\qquad L^\infty\bigl([0,T]; L^{5/3}_{\rm loc}(\Omega)\bigr)\,, \\
	\vec{u}_\ep \weak \vec{U} \qquad\qquad \mbox{ weakly in }\qquad L^2\big([0,T];W_{\rm loc}^{1,2}(\Omega)\big)\,, \\
	\Theta_\veps:=\frac{\vartheta_\ep - \bar{\vartheta}}{\ep^m}  \weak \Theta \qquad\qquad \mbox{ weakly in }\qquad L^2\big([0,T];W_{\rm loc}^{1,2}(\Omega)\big)\,,
	\end{align*}	
where $\vec{U} = (\vec U^h,0)$, with $\vec U^h=\vec U^h(t,x^h)$ such that $\divh\vec U^h=0$. In addition, the triplet $\Big(\vec{U}^h,\, \, R ,\, \, \Theta \Big)$ is a weak solution
to the incompressible Oberbeck-Boussinesq system  in $\R_+ \times \Omega$:
\begin{align}
& \divh\vec U^h\,=\,0\,, \label{eq_lim_m:div} \\
& \d_t \vec U^{h}+\divh\left(\vec{U}^{h}\otimes\vec{U}^{h}\right)+\nabla_h\Gamma-\mu (\oline\vtheta )\Delta_{h}\vec{U}^{h}=\delta_2(m)\lan R\ran\nabla_{h}F\,, \label{eq_lim_m:momentum} \\
& c_p(1,\oline\vtheta)\,\Bigl(\d_t\Theta\,+\,\divh(\Theta\,\vec U^h)\Bigr)\,-\,\kappa(\oline\vtheta)\,\Delta\Theta\,=\,\oline\vtheta\,\alpha(1,\oline\vtheta)\,\vec{U}^h\cdot\nabla_h\mc G\,,
\label{eq_lim_m:temp} \\
& \nabla_{x}\Big( \d_\varrho p(1,\oline{\vtheta})\,R\,+\,\d_\vtheta p(1,\oline{\vtheta})\,\Theta \Big)\,=\,\nabla_{x}G\,+\,\delta_2(m)\,\nabla_{x}F\,, \label{eq_lim_m:Boussinesq}
\end{align}
supplemented with the initial conditions
$$
\vec{U}_{|t=0}=\h_h\left(\lan\vec{u}^h_{0}\ran\right)\qquad \text{ and }\qquad
\Theta_{|t=0}\,=\,\frac{\oline\vtheta}{c_p(1,\oline\vtheta)}\,\Big(\d_\vrho s(1,\oline\vtheta)\,R_0\,+\,\d_\vtheta s(1,\oline\vtheta)\,\Theta_0\,+\,
\alpha(1,\oline\vtheta)\,\mc G\Big)
$$
and the boundary condition $\nabla_x \Theta \cdot\vec{n}_{|\d\Omega}\,=\,0$,
where $\vec{n}$ is the outer normal to $\d\Omega\,=\,\{x_3=0\}\cup\{x_3=1\}$. \\
In \eqref{eq_lim_m:momentum}, $\Gamma$ is a distribution in $\mc D'(\R_+\times\R^2)$ and we have set $\delta_2(m)=1$ if $m=2$, $\delta_2(m)=0$ otherwise.
In \eqref{eq_lim_m:temp}, we have defined 
$$
\mc G\,:=\,G\,+\,\delta_2(m)F\,,\qquad
c_p(\vrho,\vtheta)\,:=\,\d_\vtheta e(\vrho,\vtheta)\,+\,\alpha(\vrho,\vtheta)\,\frac{\vtheta}{\vrho}\,\d_\vtheta p(\vrho,\vtheta)\,,\qquad 
\alpha(\vrho,\vtheta)\,:=\,\frac{1}{\vrho}\,\frac{\d_\vtheta p(\vrho,\vtheta)}{\d_\vrho p(\vrho,\vtheta)}\,.
$$
\end{theorem}


\begin{remark}\label{r:lim delta theta} 
%
We notice that, after defining
$$
\Upsilon := \d_\vrho s(1,\oline{\vtheta})R + \d_\vtheta s(1,\oline{\vtheta})\,\Theta\qquad\mbox{ and }\qquad
\Upsilon_0\,:=\,\d_\vrho s(1,\oline\vtheta)\,R_0\,+\,\d_\vtheta s(1,\oline\vtheta)\,\Theta_0\,,
$$
from equation \eqref{eiq} one would get, in the limit $\veps\ra0$, the equation
\begin{equation} \label{eq:Upsilon}
\d_{t} \Upsilon +\divh\left( \Upsilon \vec{U}^{h}\right) -\frac{\kappa(\oline\vtheta)}{\oline\vtheta} \Delta \Theta =0\,,
\qquad\qquad \Upsilon_{|t=0}\,=\,\Upsilon_0\,,
\end{equation}
which is closer to the formulation of the target system given in \cite{K-M-N} and \cite{K-N}.
From \eqref{eq:Upsilon} one easily recovers \eqref{eq_lim_m:temp} by using \eqref{eq_lim_m:Boussinesq}. Formulation \eqref{eq_lim_m:temp} is in the spirit of Chapters 4 and 5 of \cite{F-N}.
\end{remark}

The case $m=1$ realizes the \emph{quasi-geostrophic balance} in the limit. Namely the Mach and Rossby numbers have the same order of magnitude, and they keep in balance in the whole asymptotic process.
The next statement is devoted to this case. Nonetheless, due to technical reasons, in this instance we have to assume $F=0$.
\begin{theorem} \label{th:m=1_F=0}
For any $\veps\in\,]0,1]$, let $\Omega_\ep$ be the domain defined by \eqref{dom} and $\Omega = \R^2 \times\,]0,1[\,$.
Let  $p$, $e$, $s$  satisfy  Gibbs' relation \eqref{gibbs} and structural hypothesis from \eqref{pp1} to \eqref{ss1s}, and that the diffusion coefficients $\mu$, $\eta$, $\kappa$
enjoy growth conditions  \eqref{mu}. Let $F=0$ and $G\in W^{1,\infty}(\Omega)$ be given as in relation \eqref{assFG}. Take $m=1$. \\
For any fixed value of $\veps$, let initial data $\left(\vrho_{0,\veps},\vec u_{0,\veps},\vtheta_{0,\veps}\right)$ verify the hypotheses fixed in Paragraph \ref{sss:data-weak}, and let
$\left( \vre, \ue, \tem\right)$ be a corresponding weak solution to system (NSF)$_\veps$, supplemented with structural hypotheses from \eqref{S} to \eqref{q} and with boundary conditions \eqref{bc1-2} and \eqref{bc3}.
Assume that the total dissipation balance \eqref{est:dissip} is satisfied.
Let $\left(R_0,\vec u_0,\Theta_0\right)$ be defined as in \eqref{conv:in_data}.

Then the convergence properties stated in the previous theorem still hold true: namely, one has
	\begin{align*}
	\varrho_\ep \rightarrow 1 \qquad\qquad \mbox{ in } \qquad L^{\infty}\big([0,T]; L_{\rm loc}^{2}+L_{\rm loc}^{5/3}(\Omega )\big)\,, \\
	R_\veps:=\frac{\varrho_\ep - 1}{\ep}  \weakstar R \qquad\qquad \mbox{ weakly-$*$ in }\qquad L^\infty\bigl([0,T]; L_{\rm loc}^{5/3}(\Omega)\bigr)\,, \\
	\vec{u}_\ep \weak \vec{U} \qquad\qquad \mbox{ weakly in }\qquad L^2\big([0,T];W_{\rm loc}^{1,2}(\Omega)\big)\,, \\
	\Theta_\veps:=\frac{\vartheta_\ep - \bar{\vartheta}}{\ep}  \weak \Theta \qquad\qquad \mbox{ weakly in }\qquad L^2\big([0,T];W_{\rm loc}^{1,2}(\Omega)\big)\,,
	\end{align*}	
where $\vec{U} = (\vec U^h,0)$, with $\vec U^h=\vec U^h(t,x^h)$ such that $\divh\vec U^h=0$. Moreover, let us introduce the real number $\mc A>0$ by the formula
\begin{equation} \label{def:A}
\mc{A}\,=\,\d_\vrho p(1,\oline{\vtheta})\,+\,\frac{\left|\d_\vtheta p(1,\oline{\vtheta})\right|^2}{\d_\vtheta s(1,\oline{\vtheta})}\,,
\end{equation}
and define
$$
\Upsilon\, := \,\d_\vrho s(1,\oline{\vtheta}) R + \d_\vtheta s(1,\oline{\vtheta})\,\Theta\qquad\mbox{ and }\qquad
q\,:=\, \d_\varrho p(1,\oline{\vtheta})R +\d_\vtheta p(1,\oline{\vtheta})\Theta -G-1/2\,.
$$
Then we have
$$
q\,=\,q(t,x^h)\,=\,\d_\varrho p(1,\oline{\vtheta})\lan R\ran\,+\,\d_\vtheta p(1,\oline{\vtheta})\lan\Theta\ran\qquad\mbox{ and }\qquad
\vec{U}^{h}=\nabla_h^{\perp} q\,.
$$
Moreover, the couple $\Big(q,\Upsilon \Big)$ satisfies (in the weak sense) the  quasi-geostrophic type system
\begin{align}
& \d_{t}\left(\frac{1}{\mc{A}}q-\Delta_{h}q\right) -\nabla_{h}^{\perp}q\cdot
\nabla_{h}\left( \Delta_{h}q\right) +\mu (\oline{\vtheta})
\Delta_{h}^{2}q\,=\,0\,, \label{eq_lim:QG}  \\
& c_p(1,\oline\vtheta)\Big(\d_{t} \Upsilon +\nabla_h^\perp q\cdot\nabla_h\Upsilon \Big)-\kappa(\oline\vtheta) \Delta \Upsilon\,=\,
\kappa(\oline\vtheta)\,\alpha(1,\oline\vtheta)\,\Delta_hq\,, \label{eq_lim:transport}
\end{align}
supplemented with the initial conditions
$$
q_{|t=0}=\curlh\lan\vec u^h_{0}\ran-\left(\d_\vrho p(1,\oline\vtheta)\lan R_0\ran+\d_\vtheta p(1,\oline\vtheta)\lan\Theta_0\ran\right)\,,\qquad
\Upsilon_{|t=0}=\d_\vrho s(1,\oline\vtheta)R_0+\d_\vtheta s(1,\oline\vtheta)\Theta_0
$$
and the boundary condition
\begin{equation} \label{eq:bc_limit_2}
\nabla_x\big(\Upsilon\,+\,\alpha(1,\oline\vtheta)\,G\big) \cdot\vec{n}_{|\d\Omega}\,=\,0\,,
\end{equation}
where $\vec{n}$ is the outer normal to the boundary $\d\Omega\,=\,\{x_3=0\}\cup\{x_3=1\}$.
\end{theorem}

\begin{remark} \label{r:limit_1}
Observe that $q$ and $\Upsilon$  
can be equivalently chosen for describing the target problem. Indeed, straightforward computations show that
\begin{align*}
R\,&=\,-\,\frac{1}{\beta}\,\big(\d_\vtheta p(1,\oline\vtheta)\,\Upsilon\,-\,\d_\vtheta s(1,\oline\vtheta)\,q\,-\,\d_\vtheta s(1,\oline\vtheta)\,G\big) \\
\Theta\,&=\,\frac{1}{\beta}\,\big(\d_\vrho p(1,\oline\vtheta)\,\Upsilon\,-\,\d_\vrho s(1,\oline\vtheta)\,q\,-\,\d_\vrho s(1,\oline\vtheta)\,G\big)\,,
\end{align*}
where we have set $\beta\,=\,\d_\vrho p(1,\oline\vtheta)\,\d_\vtheta s(1,\oline\vtheta)\,-\,\d_\vtheta p(1,\oline\vtheta)\,\d_\vrho s(1,\oline\vtheta)$. In particular, equation \eqref{eq_lim:transport}
can be deduced from \eqref{eq:Upsilon}, which is valid also when $m=1$, using that expression of $\Theta$ and the fact that
$$
\beta\,=\,c_p(1,\oline\vtheta)\,\frac{\d_\vrho p(1,\oline\vtheta)}{\oline\vtheta}\,.
$$
Here we have chosen to formulate the target entropy balance equation in terms of $\Upsilon$ (as in \cite{K-N}) rather than $\Theta$ (as in Theorem \ref{th:m-geq-2} above),
because the equation for $\Upsilon$ looks simpler (indeed, the equation for $\Theta$ would make a term in $\d_tq$ appear). The price to pay is the non-homogeneous boundary
condition \eqref{eq:bc_limit_2}, which may look a bit unpleasant.
\end{remark}

As pointed out for Theorem \ref{th:m-geq-2}, we notice that, despite the function $q$ is defined in terms of $G$, the dynamics described by \eqref{eq_lim:QG} is purely horizontal.
On the contrary, dependence on $x^3$ and vertical derivatives do appear in \eqref{eq_lim:transport}.


\section{Analysis of the singular perturbation} \label{s:sing-pert}

The purpose of this section is twofold. First of all, in Subsection \ref{ss:unif-est} we establish uniform bounds and further properties for our family
of weak solutions. Then, we study the singular operator underlying to the primitive equations (NSF)$_\veps$, and determine constraints that the limit points of our family of weak solutions
have to satisfy (see Subsection \ref{ss:ctl1}).

\subsection{Uniform bounds}\label{ss:unif-est}

This section is devoted to establish uniform bounds on our family of weak solutions $\bigl(\vrho_\veps,\vec u_\veps,\vtheta_\veps\bigr)_\veps$. First of all, let us introduce some
preliminary material.

\subsubsection{Preliminaries} \label{sss:unif-prelim}
Let us recall here some basic notations and results, which we need in proving our convergence results. We refer to Sections 4, 5 and 6 of \cite{F-N} for more details.

First of all, let us introduce the Helmholtz projection $\mbb{H}_\veps[\vec{v}]$ of a vector field $\vec{v}\in L^p(\oline\Omega_\veps; \R^3)$ on the subspace of divergence-free vector fields. It is defined by 
the decomposition
$$
	\vec{v} = \mbb{H}_\veps[\vec{v}] + \nabla_x \Psi
$$
where $\Psi \in \mc D^{1,p}(\oline\Omega_\veps)$ is the unique solution of 
	$$\int_{\Omega_\veps} \nabla_x \Psi \cdot \nabla_{x} \varphi \dx = \int_{\Omega_\veps} \vec{v} \cdot \nabla_x \varphi  \dx
	\quad\mbox{for all } \varphi \in C^\infty_c (\overline{\Omega}_\veps),$$
which formally means: $\Delta \Psi = \div \vec{v}$ and 
$\big(\vec{v}  \cdot \n_\veps\big)_{|\partial \Omega_\veps } =  0$.

\medbreak
Next, let us introduce the so-called ``essential'' and ``residual'' sets. Recall that the positive constant $\rho_*$ has been defined in Lemma \ref{l:target-rho_pos}.
Following the approach of \cite{F-N}, we define
$$
	 {\cal{O}}_{\ess}\,: = \, \left[2\rho_*/3\, ,\, 2 \right]\, \times\, \left[\tems/2\,,\, 2 \tems\right]\,,\qquad  
	 {\cal{O}}_{\res}\,: =\, \,]0,+\infty[\,^2\setminus {\cal{O}}_{\ess}\,.
$$
Then, we fix a smooth function
$\mf{b} \in C^\infty_c \bigl( \,]0,+\infty[\,\times\,]0,+\infty[\, \bigr)$ such that $0\leq \mf b\leq 1, \ \mf b\equiv1$  on the set $ {\cal{O}}_{\ess}$, and
we introduce the decomposition on essential and residual part of a measurable function $h$ as follows:
	\begin{equation*}\label{ess-def}
	h = [h]_{\ess} + [h]_{\res},\qquad\mbox{ with }\quad  [ h]_{\ess} := \mf b(\vre,\tem) h\,,\quad \  [h]_{\res} = \bigl(1-\mf b(\vre,\tem)\bigr)h\,.
	\end{equation*}
We also introduce the sets $\mc{M}^\veps_{\ess}$ and $\mc{M}^\veps_{\res}$, defined as 
	$$\mc{M}^\veps_{\ess}  := \left\{ (t,x) \in\, ]0,T[\, \times\, \Omega_\veps \ : \ \bigl(\varrho_\ep(t,x),\vartheta_\ep(t,x)\bigr) \in  {\cal{O}}_{\ess} \right\}\qquad\mbox{ and }\qquad
	\mc{M}^\veps_{\res}  := \big(\,]0,T[\,\times\,\Omega_\veps\,\big) \setminus \mc{M}^\veps_{\ess}\,,$$
and their version at fixed time $t\geq0$, i.e. 
	$$\mc{M}^\veps_{\ess} [t] := \{ x \in \Omega_\veps \  : (t,x) \in \mc{M}^\veps_{\ess} \}\qquad \mbox{ and }\qquad
	\mc{M}^\veps_{\res}[t]  := \Omega_\veps \setminus \mc{M}^\veps_{\ess}[t]\,.$$

The next result, which will be useful in the next subsection, is the analogous of Lemma 5.1 in \cite{F-N} in our context. Here we need to pay attention to the fact that,
in the case $F\neq0$, our estimates for the equilibrium states (see especially Proposition \ref{p:target-rho_bound}) are not uniform on the whole $\Omega_\veps$.
\begin{lemma}\label{l:H}
Fix $m\geq1$ and let $\vret$ and $\tems$ be the static states identified in Paragraph \ref{sss:equilibrium}. Under the previous assumptions, and with the notations introduced above,
we have the following properties.

Let $F\neq0$. For all $l>0$, there exist $\veps(l)$ and positive constants $c_j\,=\,c_j(\rho_*,\oline\vtheta,l)$, with $1\leq j\leq3$, such that, for all $0<\veps\leq\veps(l)$,
the next properties hold true, for all $x\in\oline{\mbb B}_{l}$:
\begin{enumerate}[(a)]
 \item for all $(\rho,\theta)\,\in\,\mc O_{\ess}$, one has
$$
c_1\,\left(\left|\rho-\wtilde\vrho_\veps(x)\right|^2\,+\,\left|\theta-\oline\vtheta\right|^2\right)\,\leq\,\mc E\left(\rho,\theta\;|\;\wtilde\vrho_\veps(x),\oline\vtheta\right)\,\leq\,
c_2\,\left(\left|\rho-\wtilde\vrho_\veps(x)\right|^2\,+\,\left|\theta-\oline\vtheta\right|^2\right)\,;
$$
\item for all $(\rho,\theta)\,\in\,\mc O_{\res}$, one has
$$
\mc E\left(\rho,\theta\;|\;\wtilde\vrho_\veps(x),\oline\vtheta\right)\,\geq\,c_3\,.
$$
\end{enumerate}

When $F=0$, the previous constants $\big(c_j\big)_{1\leq j\leq3}$ can be chosen to be independent of $l>0$.
\end{lemma}
\begin{proof}
Let us start by considering the case $F\neq0$. Fix $m\geq1$. In view of Lemma \ref{l:target-rho_pos} and Proposition \ref{p:target-rho_bound}, for all $l>0$ fixed, there exists $\veps(l)$ such that,
for all $\veps\leq\veps(l)$, we have $\wtilde\vrho_\veps(x)\,\in\,[\rho_*,3/2]\,\subset\,\mc O_{\ess}$ for all $x\in \oline{\mbb B}_{l}$.
With this inclusion at hand, the first inequality is an immediate consequence of the decomposition
\begin{align*}
\mc E\left(\rho,\theta\;|\;\wtilde\vrho_\veps,\oline\vtheta\right)\,&=\,\Bigl(H_{\tems}(\rho,\theta)-H_{\tems}(\rho,\oline\vtheta)\Bigr)\,+\,
\Bigl(H_{\tems}(\rho,\oline\vtheta) - H_{\tems}(\wtilde\vrho_\veps,\oline\vtheta) - (\rho - \vret)\,\partial_\vrho H_{\tems}(\vret,\tems)\Bigr) \\
&=\,\d_\vtheta H_{\oline\vtheta}(\rho,\eta)\,\bigl(\vtheta-\oline\vtheta\bigr)\,+\, \frac{1}{2}\d^2_{\vrho\vrho}H_{\oline\vtheta}(z_\veps,\oline\vtheta)\,\bigl(\rho-\wtilde\vrho_\veps\bigr)^2\,,
\end{align*}
for some suitable $\eta$ belonging to the interval connecting $\theta$ and $\oline\vtheta$, and $z_\veps$ belonging to the interval connecting $\rho$ and $\wtilde\vrho_\veps$. Indeed,
it is enough to use formulas (2.49) and (2.50) of \cite{F-N}, together with the fact that we are in the essential set.

Next, thanks again to the property $\wtilde\vrho_\veps(x)\,\in\,[\rho_*,3/2]\,\subset\,\mc O_{\ess}$, we can conclude, exactly as in relation (6.69) of \cite{F-N}, that
$$
\inf_{(\rho,\theta)\in\mc O_\res}\mc E\left(\rho,\theta\;|\;\wtilde\vrho_\veps,\oline\vtheta\right)\,\geq\,
\inf_{(\rho,\theta)\in\d\mc O_\ess}\mc E\left(\rho,\theta\;|\;\wtilde\vrho_\veps,\oline\vtheta\right)\,\geq\,c\,>\,0\,.
$$

The case $F=0$ follows by similar arguments, using that the various constants in Lemma \ref{l:target-rho_bound} and Proposition \ref{p:target-rho_bound} are uniform in $\Omega$.
This completes the proof of the lemma.
\qed
\end{proof}

\subsubsection{Uniform estimates for the family of weak solutions} \label{sss:uniform}

With the total dissipation balance \eqref{est:dissip} and Lemma \ref{l:H} at hand, we can derive uniform bounds for our family of weak solutions.
Since this derivation is somehow classical, we limit ourselves to recall the main inequalities and sketch the proofs; we refer the reader to Chapters 5, 6 and 8 of \cite{F-N} for details.

To begin with, we remark that, owing to the assumptions fixed in Paragraph \ref{sss:data-weak} on the initial data and to the structural hypotheses of Paragraphs \ref{sss:primsys} and \ref{sss:structural},
the right-hand side of \eqref{est:dissip} is \emph{uniformly bounded} for all $\veps\in\,]0,1]$.
\begin{lemma} \label{l:initial-bound}
Under the assumptions fixed in Paragraphs \ref{sss:primsys}, \ref{sss:structural} and \ref{sss:data-weak}, there exists an absolute constant $C>0$ such that, for all $\veps\in\,]0,1]$,
one has
$$
\int_{\Omega_\veps} \frac{1}{2}\vrez|\uez|^2\,\dx + \frac{1}{\ep^{2m}}\int_{\Omega_\veps}\mc E\left(\vrho_{0,\veps},\vtheta_{0,\veps}\;|\;\wtilde\vrho_\veps,\oline\vtheta\right)\,\dx\,\leq\,C\,.
$$
\end{lemma}

\begin{proof}
The boundedness of the first term in the left-hand side is an obvious consequence of \eqref{hyp:ill-vel} and \eqref{hyp:ill_data} for the density. So, let us show how to control the term containing
$\mc E\left(\vrho_{0,\veps},\vtheta_{0,\veps}\;|\;\wtilde\vrho_\veps,\oline\vtheta\right)$. Owing to Taylor formula, one has
\begin{align*}
\mc E\left(\vrho_{0,\veps},\vtheta_{0,\veps}\;|\;\wtilde\vrho_\veps,\oline\vtheta\right)\,&=\,
\d_\vtheta H_{\oline\vtheta}(\vrho_{0,\veps},\eta_{0,\veps})\,\bigl(\vtheta_{0,\veps}-\oline\vtheta\bigr)\,+\,\frac{1}{2}\,
\d^2_{\vrho\vrho}H_{\oline\vtheta}(z_{0,\veps},\oline\vtheta)\,\bigl(\vrho_{0,\veps}-\wtilde\vrho_\veps\bigr)^2\,,
\end{align*}
where we can write $\eta_{0,\veps}(x)\,=\,\oline\vtheta\,+\,\veps^m\,\lambda_\veps(x)\,\Theta_{0,\veps}$ and $z_{0,\veps}\,=\,\wtilde\vrho_\veps\,+\,\veps^m\,\z_\veps(x)\,\vrho^{(1)}_{0,\veps}$,
with both the families $\bigl(\lambda_\veps\bigr)_\veps$ and $\bigl(\z_\veps\bigr)_\veps$ belonging to $L^\infty(\Omega_\veps)$, uniformly in $\veps$ (in fact,
$\lambda_\veps(x)$ and $\z_{\veps}(x)$ belong to the interval $\,]0,1[\,$ for all $x\in\Omega_\veps$).
Notice that $\eta_{0,\veps}\,\in\,L^\infty(\Omega_\veps)$ uniformly in $\veps$, and that $\eta_{0,\veps}\geq c_1>0$ and $z_{0,\veps}\geq c_2>0$ (at least for $\veps$ small enough).

By the structural hypotheses fixed in Paragraph \ref{sss:structural} (and in particular Gibbs' law), we get (see also formula (2.50) in \cite{F-N})
\begin{align} \label{eq:d_th-H_th}
\d_\vtheta H_{\oline\vtheta}(\vrho_{0,\veps},\eta_{0,\veps})\,&=\,4\,a\,\eta_{0,\veps}^2\,\bigl(\eta_{0,\veps}-\oline\vtheta\bigr)\,+\,
\frac{\vrho_{0,\veps}}{\eta_{0,\veps}}\,\bigl(\eta_{0,\veps}-\oline\vtheta\bigr)\,\d_\vtheta e_M(\rho_{0,\veps},\eta_{0,\veps})\,.
\end{align}
In view of condition \eqref{pp3}, we gather that $\left|\d_\vtheta e_M\right|\,\leq\,c$; therefore, from hypotheses \eqref{hyp:ill_data} and Remark \ref{r:ill_data}
it is easy to deduce that
$$
\frac{1}{\veps^{2m}}\int_{\Omega_\veps}\d_\vtheta H_{\oline\vtheta}(\vrho_{0,\veps},\eta_{0,\veps})\,\bigl(\vtheta_{0,\veps}-\oline\vtheta\bigr)\,dx\,\leq\,C\,.
$$

Moreover, by \eqref{pp1} we get (keep in mind formula (2.49) of \cite{F-N})
$$
\d^2_{\vrho\vrho}H_{\oline\vtheta}(z_{0,\veps},\oline\vtheta)\,=\,\frac{1}{z_{0,\veps}}\,\d_\vrho p_M(z_{0,\veps},\oline\vtheta)\,=\,\frac{1}{\sqrt{\oline\vtheta}}\,\frac{1}{Z_{0,\veps}}\,P'(Z_{0,\veps})\,,
$$
where we have set $Z_{0,\veps}\,=\,z_{0,\veps}\,\oline\vtheta^{-3/2}$. Now, thanks to \eqref{pp3} again and to the fact that $z_{0,\veps}$ is strictly positive, we can estimate,
for some positive constants which depend also on $\oline\vtheta$,
$$
\frac{1}{Z_{0,\veps}}\,P'(Z_{0,\veps})\,\leq\,C\,\frac{P(Z_{0,\veps})}{Z_{0,\veps}^2}\,\leq\,C\left(\frac{P(Z_{0,\veps})}{Z_{0,\veps}^2}\,\bbbone_{\{0\leq Z_{0,\veps}\leq1\}}\,+\,
\frac{P(Z_{0,\veps})}{Z_{0,\veps}^{5/3}}\,\bbbone_{\{Z_{0,\veps}\geq1\}}\right)\,\leq\,C\,,
$$
where we have used also \eqref{pp4}. Hence, it is now easy to check that
$$
\frac{1}{2\veps^{2m}}\int_{\Omega_\veps}\d^2_{\vrho\vrho}H_{\oline\vtheta}(z_{0,\veps},\oline\vtheta)\,\bigl(\vrho_{0,\veps}-\wtilde\vrho_\veps\bigr)^2\,dx\,\leq\,C\,.
$$
This inequality completes the proof of the lemma.
\qed
\end{proof}

\medbreak

Owing to the previous lemma, from \eqref{est:dissip} we gather, for any $T>0$, the estimates
\begin{align}
	\sup_{t\in[0,T]} \| \sqrt{\vre}\ue\|_{L^2(\Omega_\veps;\R^3)}\, &\leq\,c \label{est:momentum} \\
	\| \sigma_\ep\|_{{\mathcal{M}}^+ ([0,T]\times\oline\Omega_\veps )}\, &\leq \,\ep^{2m}\, c\,. \label{est:sigma}
\end{align}
Fix now any $l>0$. Employing Lemma~\ref{l:H} (and keeping track of the dependence of constants only on $l$), we deduce
	\begin{align}
	\sup_{t\in[0,T]} \left\| \left[ \dfrac{\vre - \vret}{\ep^m}\right]_\ess (t) \right\|_{L^2(\mbb B_{l})}\,&\leq\, c(l) \label{est:rho_ess} \\
	\sup_{t\in[0,T]} \left\| \left[ \dfrac{\tem - \tems}{\ep^m}\right]_\ess (t) \right\|_{L^2(\mbb B_{l})}\, &\leq\, c(l)\,. \label{est:theta-ess} 
	\end{align}
In addition, we infer also that the measure of the ``residual set'' is small: more precisely, we have
\begin{equation}\label{est:M_res-measure}
	\sup_{t\in[0,T]} \int_{\mbb B_{l}}	\bbbone_{\mc{M}^\veps_\res[t]} \,dx\,\leq \,\ep^{2m}\, c(l)\,.
	\end{equation}

{\begin{remark}\label{rmk:cut-off}
When $F=0$, thanks to Lemma \ref{l:target-rho_bound} and Proposition \ref{p:target-rho_bound},
one can see that estimates \eqref{est:rho_ess}, \eqref{est:theta-ess} and \eqref{est:M_res-measure} hold on the whole $\Omega_\veps$, without any need of taking the localisation on the cylinders $\B_l$.
From this observation, it is easy to see that, when $F=0$, we can replace $\mathbb{B}_l$ with the whole $\Omega_\veps$ in all the following estimates.
\end{remark}}
		
Now, we fix $l>0$. We estimate
$$
\int_{\B_{l}}\left|\left[\vrho_\veps\,\log\vrho_\veps\right]_\res\right|\,dx\,=\,
\int_{\B_{l}}\left|\vrho_\veps\,\log\vrho_\veps\right|\,\bbbone_{\{0\leq\vrho_\veps\leq2\rho_*/3\}}\,dx\,+\,
\int_{\B_{l}}\left|\vrho_\veps\,\log\vrho_\veps\right|\,\bbbone_{\{\vrho_\veps\geq2\}}\,dx\,.
$$
Thanks to \eqref{est:M_res-measure}, the former term in the right-hand side is easily controlled by $\veps^{2m}$, up to a suitable multiplicative constant also depending on $l$.
As for the latter term, we have to argue in a different way. Owing to inequalities from \eqref{pp2} 
to \eqref{pp4}, we get that $\d^2_\vrho H_{\oline\vtheta}(\vrho,\oline\vtheta)\geq C/\vrho$; therefore, by direct integration we find
\begin{align*}
C\,\vrho_\veps\,\log\vrho_\veps\,-\,C\left(\vrho_\veps-1\right)\,&\leq\,H_{\oline\vtheta}(\vrho_\veps,\oline\vtheta)\,-\,H_{\oline\vtheta}(1,\oline\vtheta)\,-\,
\d_\vrho H_{\oline\vtheta}(1,\oline\vtheta)(\vrho_\ep-1) \\
&\leq\,\mc E\left(\vrho_\veps,\vtheta_\veps\;|\;\wtilde\vrho_\veps,\oline\vtheta\right)\,+\,\mc E\left(\wtilde\vrho_\veps,\oline\vtheta\;|\;1,\oline\vtheta\right)\,+\,
\Big( \d_\vrho H(\wtilde\vrho_\veps,\oline\theta) - \d_\vrho H(1,\oline\theta)\Big) \big( \vrho_\veps - \wtilde\vrho_\veps \big)\,,
\end{align*}
because an expansion analogous to \eqref{eq:d_th-H_th} allows to gather that $H_{\oline\vtheta}(\vrho_\veps,\oline\vtheta)\,-\,H_{\oline\vtheta}(\vrho_\veps,\vtheta_\veps)\,\leq\,0$.
On the one hand, using \eqref{est:dissip}, Proposition \ref{p:target-rho_bound} and \eqref{est:M_res-measure} one deduces
$$
\left|\int_{\B_{l}\cap\mc O_\res}\left(\mc E\left(\vrho_\veps,\vtheta_\veps\;|\;\wtilde\vrho_\veps,\oline\vtheta\right)\,+\,
\mc E\left(\wtilde\vrho_\veps,\oline\vtheta\;|\;1,\oline\vtheta\right)\,+\,\Big( \d_\vrho H(\wtilde\vrho_\veps,\oline\theta) - \d_\vrho H(1,\oline\theta)\Big)
\big( \vrho_\veps - \wtilde\vrho_\veps \big)\,\right)\right|\,\leq\,C\,\veps^{2m}\,.
$$
On the other hand, $\vrho_\veps\log\vrho_\veps-\left(\vrho_\veps-1\right)\,\geq\,\vrho_\veps\left(\log\vrho_\veps-1\right)\,\geq\,(1/2)\,\vrho_\veps\,\log\vrho_\veps$ whenever
$\vrho_\veps\geq e^2$. Hence,  since we have
$$
\int_{\B_{l}}\left|\vrho_\veps\,\log\vrho_\veps\right|\,\bbbone_{\{2\leq\vrho_\veps\leq e^2\}}\,dx\,\leq\,C\,\veps^{2m}
$$
owing to \eqref{est:M_res-measure} again, we finally infer that
\begin{equation} \label{est:rho*log-rho}
\sup_{t\in[0,T]}\int_{\B_{l}}\left|\left[\vrho_\veps\,\log\vrho_\veps\right]_\res(t)\right|\,dx\,\leq\,c(l)\,\veps^{2m}\,.
\end{equation}


Owing to inequality \eqref{est:rho*log-rho}, we deduce (exactly as in \cite{F-N}, see estimates (6.72) and (6.73) therein) that
	\begin{align}
	\sup_{t\in [0,T]} \int_{\B_{l}}
	\bigl(
	\left| [ \vre e(\vre,\tem)]_\res\right| +  \left| [ \vre s(\vre,\tem)]_\res\right|\bigr)\,\dx\,&\leq\,\ep^{2m}\, c (l)\,, \label{est:e-s_res}
	\end{align}
which in particular implies (again, we refer to Section 6.4.1 of \cite{F-N} for details) the following bounds:
\begin{align}
	\sup_{t\in [0,T]} \int_{\B_{l}} [ \vre]^{5/3}_\res (t)\, \dx \,&\leq\,\ep^{2m}\, c (l) \label{est:rho_res} \\
	\sup_{t\in [0,T]} \int_{\B_{l}} [ \tem]^{4}_\res (t)\, \dx\, &\leq\,\ep^{2m}\, c(l)\,. \label{est:theta_res}
\end{align}

Let us move further. In view of \eqref{S}, \eqref{ss}, \eqref{q} and \eqref{mu}, \eqref{est:sigma} implies
	\begin{align}
	\int_0^T \left\| \nabla_x \ue + \nabla_x^T \ue  - \frac{2}{3} \div \ue \tens{Id} \right\|^2_{L^2(\Omega_\veps;\R^{3\times3})}\, \dt\,
	&\leq\, c \label{est:Du}  \\
	\int_0^T \left\| \nabla_x \left(\frac{\tem - \tems}{\ep^m}\right) \right\|^2_{L^2(\Omega_\veps;\R^3)}\, \dt\, +\,
	\int_0^T \left\| \nabla_x \left(\frac{\log(\tem) - \log(\tems)}{\ep^m}\right) \right\|^2_{L^2(\Omega_\veps;\R^3)} \,\dt\,
	&\leq\, c\,. \label{est:D-theta}
	\end{align}

{Thanks to the previous inequalities and \eqref{est:M_res-measure}, we can argue as in Subsection 8.2 of \cite{F-N}:
by generalizations of respectively Poincar\'e and Korn inequalities, for all $l>0$ we gather also}
\begin{align}
\int_0^T \left\| \frac{\tem - \tems}{\ep^m} \right\|^2_{W^{1,2}({\B_l};\R^3)}\, \dt \,+\,
\int_0^T \left\| \frac{\log(\tem) - \log(\tems)}{\ep^m} \right\|^2_{W^{1,2}({\B_l};\R^3)}\, \dt\,&\leq\,c(l) \label{est:theta-Sob} \\
\int_0^T \left\| \ue \right\|^2_{W^{1,2}(\B_l; \R^3)} \dt\,&\leq\,c(l)\,. \label{est:u-H^1}
\end{align}

Finally, we discover that
\begin{align}
\int^T_0\left\|\left[\frac{\vrho_\veps\,s(\vrho_\veps,\vtheta_\veps)}{\veps^m}\right]_{\res}\right\|^{2}_{L^{30/23}(\B_{l})}\,dt\,+\,
\int^T_0\left\|\left[\frac{\vrho_\veps\,s(\vrho_\veps,\vtheta_\veps)}{\veps^m}\right]_{\res}\,
\vec{u}_\veps\right\|^{2}_{L^{30/29}(\B_{l})}\,dt\,&\leq\,c(l)  \label{est:rho-s_res} \\
\int^T_0\left\|\frac{1}{\veps^m}\,\left[\frac{\kappa(\vtheta_\veps)}{\vtheta_\veps}\right]_{\res}\,
\nabla_{x}\vtheta_\veps(t)\right\|^{2}_{L^{1}(\B_l)}\,dt\,&\leq\,c(l)\,.  \label{est:Dtheta_res}
\end{align}
Indeed, arguing similarly as in the proof of Proposition 5.1 of \cite{F-N}, we have that
\begin{equation} \label{5.58_book}
\left[\vrho_\veps\,s(\vrho_\veps,\vtheta_\veps)\right]_{\res}\leq C\, \left[ \vrho_\veps +\vrho_\veps \, |\log \vrho_\veps |\, +\vrho_\veps \, |\log \vtheta_\veps -\log \oline \vtheta|+\vtheta_{\veps}^{3}\, \right]_{\res} 
\end{equation}
and thanks to the previous uniform bounds \eqref{est:rho_res}, \eqref{est:theta_res} and \eqref{est:theta-Sob}, one has that
$\big(\left[\vrho_\veps \right]_{\res}\big)_\veps\subset L_{T}^{\infty}( L_{\rm loc}^{5/3})$,
$\big(\left[\vrho_\veps \, |\log \vrho_\veps |\,\right]_{\res}\big)_\veps\subset L_{T}^{\infty}( L_{\rm loc}^{q})$ for all $1\leq q< 5/3$ (see relation (5.60) in \cite{F-N}),
$\big(\left[\vrho_\veps \, |\log \vtheta_\veps -\log \oline \vtheta|\, \right]_{\res}\big)_\veps\subset L_{T}^{2}( L_{\rm loc}^{30/23})$ and finally
$\big(\left[\vtheta_{\veps}^{3}\, \right]_{\res}\big)_\veps\subset L_{T}^{\infty}( L_{\rm loc}^{4/3})$.
Let us recall that, as stipulated at the end of the introduction, the inclusion symbol means that the sequences are uniformly boundeed in the respective spaces.
Then, it follows that the first term in \eqref{est:rho-s_res}  is in $L_T^{2}(L_{\rm loc}^{30/23})$. 
Next, taking \eqref{5.58_book} we obtain
\begin{equation*}
\left[\vrho_\veps\,s(\vrho_\veps,\vtheta_\veps)\ue\right]_{\res}\leq C\, \left[ \vrho_\veps\ue +\vrho_\veps \, |\log \vrho_\veps |\, \ue\, +\vrho_\veps \, |\log \vtheta_\veps -\log \oline \vtheta|\,\ue +\vtheta_{\veps}^{3}\ue \, \right]_{\res} 
\end{equation*}
and using the uniform bounds \eqref{est:rho_res} and \eqref{est:u-H^1}, we have that $\big(\left[\vrho_\veps\ue\right]_{\res}\big)_\veps\subset L_T^{2}(L_{\rm loc}^{30/23})$. 
Now, we look at the second term. We know that $\big(\left[\vrho_\veps \, |\log \vrho_\veps |\, \right]_{\res}\big)_\veps\subset L_{T}^{\infty}( L_{\rm loc}^{q})$ for all $1\leq q< 5/3$ and
$\ue \in L_T^{2}(L_{\rm loc}^{6})$ (thanks to Sobolev embeddings). Then, we take $q$ such that $1/p:=1/q+1/6<1$ and so
$$ \big(\left[\vrho_\veps \, |\log \vrho_\veps |\, \ue\,\right]_{\res}\big)_\veps\subset L_T^{2}(L_{\rm loc}^{p})\, . $$
Keeping \eqref{est:rho_res}, \eqref{est:theta-Sob} and \eqref{est:momentum} in mind and using that
$$\left[\vrho_\veps \, |\log \vtheta_\veps -\log \oline \vtheta|\, \vec{u}_\veps\, \right]_{\res}=
\left[\sqrt{\vrho_\veps}\, |\log \vtheta_\veps -\log \oline \vtheta|\, \sqrt{\vrho_\veps}\, \vec{u}_{\veps}\, \right]_{\res}\,,$$
we obtain that the third term is uniformly bounded in $L_{T}^{2}( L_{\rm loc}^{30/29})$.
Using again the uniform bounds, we see that the last term is in $L_T^{\infty}(L_{\rm loc}^{12/11})$.
Thus, we obtain \eqref{est:rho-s_res}.  \\
To get \eqref{est:Dtheta_res}, we use instead the following estimate (see Proposition 5.1 of \cite{F-N}):
\begin{equation*}
\left[\frac{k(\vtheta_\veps )}{\vtheta_\veps}\right]_{\res}\left|\frac{\nabla_{x}\vtheta_\veps}{\veps^{m}}\right| \leq C
\left(\left|\frac{\nabla_{x}(\log \vtheta_\veps )}{\veps^m}\right|+\left[\vtheta_{\veps}^{2}\right]_{\res}\left|\frac{\nabla_{x}\vtheta_{\veps}}{\veps^m}\right|\right)\,.
\end{equation*} 
Owing to the previous uniform bounds, the former term is uniformly bounded in $L_{T}^{2}( L_{\rm loc}^{2})$ and the latter one is uniformly bounded in $L_{T}^{2}( L_{\rm loc}^{1})$.
So, we obtain the estimate \eqref{est:Dtheta_res}.

\begin{remark} \label{r:bounds}
We underline that, differently from \cite{F-N}, here we have made the integrability indices in \eqref{est:rho-s_res} and \eqref{est:Dtheta_res} explicit.
In particular, having the $L^2$ norm in time will reveal to be fundamental for the compensated compactness argument, see Lemma \ref{l:source_bounds} below.
\end{remark}

\subsection{Constraints on the limit}\label{ss:ctl1}

In this section, we establish some properties that the limit points of the fixed family of weak solutions $\bigl(\vrho_\veps,\vec u_\veps,\vtheta_\veps\bigr)_\veps$ have to satisfy.
These are static relations, which allow us to identify the limit density, velocity and temperature profiles; they do not characterise the limit dynamics yet.

\subsubsection{Preliminary considerations} \label{sss:constr_prelim}
To begin with, let us propose an extension of Proposition 5.2 of \cite{F-N}, which will be heavily used in the sequel. Two are the novelties here: firstly,
for the sake of generality we will consider a non-constant density profile $\wtilde\vrho$ in the limit (although this property is not used in our analysis); in addition, due to the
centrifugal force, when $F\neq0$ our result needs a localization on compact sets. 

\begin{proposition} \label{p:prop_5.2}
Let $m\geq1$ be fixed. Let $\wtilde\vrho_\veps$ and $\oline\vtheta$ be the static solutions identified and studied in Paragraph \ref{sss:equilibrium}, and take $\wtilde\vrho$
to be the pointwise limit of the family $\left(\wtilde\vrho_\veps\right)_\veps$ (in particular, $\wtilde\vrho\equiv1$ if $m>1$ or $m=1$ and $F=0$).
Let $(\varrho_\ep)_\ep$ and $(\vartheta_\ep)_\ep$ be sequences of non-negative measurable functions, and define
$$
R_\veps\,:=\,\frac{\varrho_\ep - \wtilde\vrho}{\ep^m}\qquad\mbox{ and }\qquad
\Theta_\veps\,:=\,\frac{\vartheta_\ep - \oline\vartheta}{\ep^m}\,.
$$
Suppose that, in the limit $\veps\ra0$, one has the convergence properties
\begin{equation}\label{hyp_p5.2:1}
\left[R_\veps\right]_{\rm ess}\, \weakstar\, R \quad\mbox{ and }\quad
\left[\Theta_\veps\right]_{\rm ess}\, \weakstar\, \Theta\qquad\quad \mbox{ in the weak-$*$ topology of} \ L^\infty\bigl([0,T];L^2(K)\bigr)\,,
\end{equation}
for any compact $K\subset \Omega$, and that, for any $L>0$, one has
\begin{equation}\label{hyp_p5.2:2}
\sup_{t\in[0,T]} \int_{\mbb B_{L}} \bbbone_{\mathcal{M}^\ep_{\rm res} [t] }\,dx\, \leq \,c(L)\,\ep^{2m}\,.
\end{equation}

Then, for any given function $G \in C^1(\overline{\mathcal{O}}_{\rm ess})$,  one has the convergence
$$
	\frac{[G(\varrho_\ep,\vartheta_\ep)]_{\ess} - G(\wtilde\vrho,\oline\vartheta)}{\ep^m}\,
	\weakstar\,\partial_\vrho G(\wtilde\vrho,\oline\vartheta)\,R\, +\,\partial_\vtheta G(\wtilde\vrho,\oline\vartheta)\,\Theta
	\qquad \mbox{ in the weak-$*$ topology of} \ L^\infty\bigl([0,T];L^2(K)\bigr)\,,
$$
for any compact $K\subset \Omega$.
\end{proposition}

\begin{proof}
The case $\wtilde\vrho\equiv1$ follows by a straightforward adaptation of the proof of Proposition 5.2 of \cite{F-N}. So, let us immediately focus on the case $m=1$ and $F\neq0$,
so that the target profile $\wtilde\vrho$ is non-constant.

We start by observing that, by virtue of \eqref{hyp_p5.2:2} and Lemma \ref{l:target-rho_bound}, the estimates
$$
\frac{1}{\veps}\,\left\|\left[G(\wtilde\vrho,\oline\vtheta)\right]_{\rm res}\right\|_{L^1(\mbb B_L)}\,\leq\,C(L)\,\veps\qquad\mbox{ and }\qquad
\frac{1}{\veps}\,\left\|\left[G(\wtilde\vrho,\oline\vtheta)\right]_{\rm res}\right\|_{L^2(\mbb B_L)}\,\leq\,C(L)
$$
hold true, for any $L>0$ fixed. Combining those bounds with hypothesis \eqref{hyp_p5.2:1}, after taking $L>0$ so large that $K\subset\mbb B_L$, we see that it is enough to prove the convergence
\begin{equation} \label{conv:to-prove}
\int_{K}\left[\frac{G(\varrho_\ep,\vartheta_\ep)- G(\wtilde\vrho,\oline\vartheta)}{\ep}\,-\,
\partial_\vrho G(\wtilde\vrho,\oline\vartheta)\,R_\veps\,-\,\partial_\vtheta G(\wtilde\vrho,\oline\vartheta)\,\Theta_\veps\right]_{\ess}\,\psi\,\dx\,\longrightarrow\,0
\end{equation}
for any compact $K$ fixed and any $\psi\in L^1\bigl([0,T];L^2(K)\bigr)$.

Next, we remark that, whenever $G\in C^2(\oline{\mc O}_{\rm ess})$, we have
\begin{align}
&\left|\left[\frac{G(\varrho_\ep,\vartheta_\ep)- G(\wtilde\vrho,\oline\vartheta)}{\ep}\,-\,
\partial_\vrho G(\wtilde\vrho,\oline\vartheta)\,R_\veps-\partial_\vtheta G(\wtilde\vrho,\oline\vartheta)\,\Theta_\veps\right]_{\ess}\right|\,\leq \label{est:prelim_5.2} \\
&\qquad\qquad\qquad\qquad\qquad\qquad\qquad\quad
\leq\,C\,\veps\,\left\|{\rm Hess}(G)\right\|_{L^\infty(\oline{\mc O}_{\rm ess})}\left(\left[R_\veps\right]_{\rm ess}^2+\left[\Theta_\veps\right]_{\rm ess}^2\right), \nonumber
\end{align}
where we have denoted by ${\rm Hess}(G)$ the Hessian matrix of the function $G$ with respect to its variables $(\vrho,\vtheta)$.
In particular, \eqref{est:prelim_5.2} implies the estimate
\begin{align}
\left\|\left[\frac{G(\varrho_\ep,\vartheta_\ep)- G(\wtilde\vrho,\oline\vartheta)}{\ep}\,-\,
\partial_\vrho G(\wtilde\vrho,\oline\vartheta)\,R_\veps\,-\,\partial_\vtheta G(\wtilde\vrho,\oline\vartheta)\,\Theta_\veps\right]_{\ess}\right\|_{L^\infty_T(L^1(K))}\,
\leq\,C\,\veps\,. \label{est:prop_5.2}
\end{align}
Property \eqref{conv:to-prove} then follows from \eqref{est:prop_5.2}, after noticing that both the terms $\left[G(\varrho_\ep,\vartheta_\ep)- G(\wtilde\vrho,\oline\vartheta)\right]_\ess/\veps$
and $\left[\partial_\vrho G(\wtilde\vrho,\oline\vartheta)\,R_\veps+\partial_\vtheta G(\wtilde\vrho,\oline\vartheta)\,\Theta_\veps\right]_{\ess}$
are uniformly bounded in $L_T^\infty\bigl(L^2(K)\bigr)$.

Finally, when $G$ is just $C^1(\oline{\mc O}_{\rm ess})$, we approximate it by a family of smooth functions $\bigl(G_n\bigr)_{n\in\N}$, uniformly in $C^1(\oline{\mc O}_{\rm ess})$.
Obviously, for each $n$, convergence \eqref{conv:to-prove} holds true for $G_n$. Moreover, we have
$$
\left|\left[\frac{G(\varrho_\ep,\vartheta_\ep)- G(\wtilde\vrho,\oline\vartheta)}{\ep}\right]_{\rm ess}-
\left[\frac{G_n(\varrho_\ep,\vartheta_\ep)- G_n(\wtilde\vrho,\oline\vartheta)}{\ep}\right]_{\rm ess}\right|\,\leq\,C\,
\left\|G\,-\,G_n\right\|_{C^1(\oline{\mc O}_{\rm ess})}\left(\left[R_\veps\right]_{\rm ess}+\left[\Theta_\veps\right]_{\rm ess}\right)\,,
$$
and a similar bound holds for the terms presenting partial derivatives of $G$. In particular, these controls entail that the remainders, created replacing $G$ by $G_n$ in
\eqref{conv:to-prove}, are uniformly small in $\veps$, whenever $n$ is sufficiently large.
This completes the proof of the proposition.
\qed
\end{proof}

\medbreak
From now on, we will focus on the two cases \eqref{eq:choice-m}: either $m\geq2$ and possibly $F\neq0$, or $m\geq1$ and $F=0$.
Indeed, if $1<m<2$ and $F\neq 0$, the convergence of $\wtilde{\varrho}_\veps$ to $1$ is too ``slow'' (keep in mind Proposition \ref{p:target-rho_bound})
and it is not possible to pass to the limit in Proposition \ref{p:prop_5.2} above with order $\veps^m$ (in this respect, see also Remark \ref{r:F-G} below).



Recall that, in both cases \eqref{eq:choice-m}, the limit density profile is always constant, say $\wtilde\vrho\equiv1$. Let us fix an arbitrary positive time
$T>0$, which we keep fixed until the end of this paragraph.
Thanks to \eqref{est:rho_ess}, \eqref{est:rho_res} and Proposition \ref{p:target-rho_bound}, we get
	\begin{equation}\label{rr1}
\| \vre - 1 \|_{L^\infty_T(L^2 + L^{5/3}(K))}\,\leq\, \ep^m\,c(K) \qquad \mbox{ for all }\;
K \subset \Omega\quad\mbox{ compact.}
	\end{equation}
In particular, keeping in mind the notations introduced in \eqref{in_vr} and \eqref{eq:in-dens_dec}, we can define
\begin{equation} \label{def_deltarho}
R_\veps\,:= \frac{\varrho_\ep -1}{\ep^m} = \,\vrho_\veps^{(1)}\,+\,\wtilde{r}_\veps\;,\qquad\quad\mbox{ where }\quad
\vrho_\veps^{(1)}(t,x)\,:=\,\frac{\vre-\wtilde{\vrho}_\veps}{\ep^m}\quad\mbox{ and }\quad
\wtilde{r}_\veps(x)\,:=\,\frac{\wtilde{\vrho}_\veps-1}{\ep^m}\,.
\end{equation}
Thanks to \eqref{est:rho_ess}, \eqref{est:rho_res} and Proposition \ref{p:target-rho_bound}, the previous quantities verify the following bounds:
\begin{equation}\label{uni_varrho1}
\sup_{\veps\in\,]0,1]}\left\|\vrho_\veps^{(1)}\right\|_{L^\infty_T(L^2+L^{5/3}({\B_{l}}))}\,\leq\, c \qquad\qquad\mbox{ and }\qquad\qquad
\sup_{\veps\in\,]0,1]}\left\| \wtilde{r}_\veps \right\|_{L^{\infty}(\B_{l})}\,\leq\, c \,.
\end{equation}
As usual, here above the radius $l>0$ is fixed (and the constants $c$ depend on it). In addition, in the case $F=0$, there is no need of localising in $\B_l$, and one gets instead
\begin{equation*}
\sup_{\veps\in\,]0,1]}\left\|\vrho_\veps^{(1)}\right\|_{L^\infty_T(L^2+L^{5/3}(\Omega_\veps))}\,\leq\, c \qquad\qquad\mbox{ and }\qquad\qquad
\sup_{\veps\in\,]0,1]}\left\| \wtilde{r}_\veps \right\|_{L^{\infty}(\Omega_\veps)}\,\leq\,\sup_{\veps\in\,]0,1]}\left\| \wtilde{r}_\veps \right\|_{L^{\infty}(\Omega)}\,\leq\, c \,.
\end{equation*}
In view of the previous properties, there exist $\vrho^{(1)}\in L^\infty_T(L^{5/3}_{\rm loc})$ and
$\wtilde{r}\in L^\infty_{\rm loc}$ such that (up to the extraction of a suitable subsequence)
\begin{equation} \label{conv:rr}
\vrho_\veps^{(1)}\,\weakstar\,\vrho^{(1)}\qquad\quad \mbox{ and }\qquad\quad \wtilde{r}_\veps\,\weakstar\,\wtilde{r}\,,
\end{equation}
where we understand that limits are taken in the weak-$*$ topology of the respective spaces.
Therefore
	\begin{equation}\label{conv:r}
	R_\veps\, \weakstar\,R\,:=\,\vrho^{(1)}\,+\,\wtilde{r}\qquad\qquad\qquad \mbox{ weakly-$*$ in }\quad L^\infty\bigl([0,T]; L^{5/3}_{\rm loc}(\Omega)\bigr)\,.
	\end{equation}
Observe that $\wtilde r$ can be interpreted as a datum of our problem.
Moreover, owing to Proposition~\ref{p:target-rho_bound} and \eqref{est:rho_ess}, we also get
$$
	\left[R_\veps \right]_{\rm  ess}\weakstar R\qquad\qquad \mbox{ weakly-$*$ in }\quad L^\infty\bigl([0,T]; L^2_{\rm loc}(\Omega)\bigr)\,.
$$

In a pretty similar way, we also find that
\begin{align}
\Theta_\veps\,:=\,\frac{\vtheta_\veps\,-\,\oline{\vtheta}}{\veps^m}\,&\rightharpoonup\,\Theta
\qquad\qquad\mbox{ in }\qquad L^2\bigl([0,T];W^{1,2}_{\rm loc}(\Omega)\bigr) \label{conv:theta} \\
\vec{u}_\veps\,&\weak\,\vec{U}\qquad\qquad\mbox{ in }\qquad L^2\bigl([0,T];W_{\rm loc}^{1,2}(\Omega)\bigr)\,. \label{conv:u}
\end{align}

Let us infer now some properties that these weak limits have to satisfy, starting with the case of anisotropic scaling, namely, in view of \eqref{eq:choice-m}, either $m\geq2$, or $m>1$ and $F=0$.

\subsubsection{The case of anisotropic scaling} \label{ss:constr_2}

When $m\geq 2$, or $m>1$ and $F=0$, the system presents multiple scales, which act (and interact) at the same time; however, the low Mach number limit has a predominant effect.
As established in the next proposition, this fact imposes some rigid constraints on the target profiles.


\begin{proposition} \label{p:limitpoint}
Let $m\geq2$, or $m>1$ and $F=0$ in (NSF)$_\veps$.
Let $\left( \vre, \ue, \tem\right)_{\veps}$ be a family of weak solutions, related to initial data $\left(\vrho_{0,\veps},\vec u_{0,\veps},\vtheta_{0,\veps}\right)_\veps$
verifying the hypotheses of Paragraph \ref{sss:data-weak}. Let $(\vrho^{(1)},R, \vec{U},\Theta )$ be a limit point of the sequence
$\left(\vrho_\veps^{(1)}, R_\veps, \ue,\Theta_\veps\right)_{\veps}$, as identified in Subsection \ref{sss:constr_prelim}. Then
\begin{align}
&\vec{U}\,=\,\,\Big(\vec{U}^h\,,\,0\Big)\,,\qquad\qquad \mbox{ with }\qquad \vec{U}^h\,=\,\vec{U}^h(t,x^h)\quad \mbox{ and }\quad \div_{\!h}\,\vec{U}^h\,=\,0\,, \label{eq:anis-lim_1} \\[1ex]
&\nabla_x\Big(\d_\varrho p(1,\oline{\vtheta})\,R\,+\,\d_\vtheta p(1,\oline{\vtheta})\,\Theta\Big)\,=\,\nabla_x G\,+\,\delta_2(m)\nabla_x F
\qquad\qquad\mbox{ a.e. in }\;\,\R_+\times \Omega\,, \label{eq:anis-lim_2} \\[1ex]
&\d_{t} \Upsilon +\div_{h}\left( \Upsilon \vec{U}^{h}\right) -\frac{\kappa(\oline\vtheta)}{\oline\vtheta} \Delta \Theta =0\,,\qquad\qquad
\mbox{ with }\qquad \Upsilon\,:=\,\d_\vrho s(1,\oline{\vtheta})R + \d_\vtheta s(1,\oline{\vtheta})\,\Theta\,,
\label{eq:anis-lim_3}
\end{align}
where 
the last equation is supplemented with the initial condition
$\Upsilon_{|t=0}=\d_\vrho s(1,\oline\vtheta)\,R_0\,+\,\d_\vtheta s(1,\oline\vtheta)\,\Theta_0$.
\end{proposition}

\begin{proof}
Let us focus here on the case $m\geq 2$ and $F\neq 0$. A similar analysis yields the result also in the case $m>1$, provided we take $F=0$.

First of all, let us consider the weak formulation of the mass equation \eqref{ceq}: for any test function $\varphi\in C_c^\infty\bigl(\R_+\times\Omega\bigr)$, denoting $[0,T]\times K\,=\,{\rm supp} \, \varphi$, with $\vphi(T,\cdot)\equiv0$, we have
$$
-\int^T_0\int_K\bigl(\vrho_\veps-1\bigr)\,\d_t\varphi \dxdt\,-\,\int^T_0\int_K\vrho_\veps\,\vec{u}_\veps\,\cdot\,\nabla_{x}\varphi \dxdt\,=\,
\int_K\bigl(\vrho_{0,\veps}-1\bigr)\,\varphi(0,\,\cdot\,)\dx\,.
$$
We can easily pass to the limit in this equation, thanks to the strong convergence $\vrho_\veps\longrightarrow1$ provided by \eqref{rr1} and the weak convergence of
$\vec{u}_\veps$ in $L_T^2\bigl(L^6_{\rm loc}\bigr)$ (by \eqref{conv:u} and Sobolev embeddings): we find
$$
-\,\int^T_0\int_K\vec{U}\,\cdot\,\nabla_{x}\varphi \dxdt\,=\,0
$$
for any test function $\varphi \, \in C_c^\infty\bigl([0,T[\,\times\Omega\bigr)$, which in particular implies
\begin{equation} \label{eq:div-free}
\div \U = 0 \qquad\qquad\mbox{ a.e. in }\; \,\R_+\times \Omega\,.
\end{equation}

Let us now focus on the momentum equation \eqref{meq}, or rather on its weak formulation \eqref{weak-mom}.
First of all, we test the momentum equation on $\veps^m\,\vec\phi$, for a smooth compactly supported $\vec\phi$.
By use of the uniform bounds we got in Subsection \ref{ss:unif-est}, it is easy to see that the only terms which do not converge to $0$ are the ones involving the pressure and
the gravitational force; in the endpoint case $m=2$, we also have the contribution of the centrifugal force. Hence, let us focus on them, and more precisely on the quantity
$$
\Xi\,:=\,\frac{\nabla_x p(\vrho_\veps,\vtheta_\veps)}{\veps^m}\,-\,\veps^{m-2}\,\vrho_\veps\nabla_x F\,-\,\vrho_\veps\nabla_x G\,.
$$
Owing to relation \eqref{prF}, we can write 
\begin{equation}\label{eq:mom_rest_1}
\Xi\,=\,
\frac{1}{\veps^m}\nabla_x\left(p(\vrho_\veps,\vtheta_\veps)\,-\,p(\wtilde{\vrho}_\veps,\oline{\vtheta})\right)\,-\,
\veps^{m-2}\,\left(\vrho_\veps-\wtilde{\vrho}_\veps\right)\nabla_x F\,-\,\left(\vrho_\veps-\wtilde{\vrho}_\veps\right)\nabla_x G\,.
\end{equation}
By uniform bounds and \eqref{conv:r}, it is easy to see that the second and third terms in the right-hand side of the previous relation converge to $0$,
whenever tested against any smooth compactly supported $\vec\phi$; notice that this is true actually for any $m>1$.
On the other hand, for the first item we can use the decomposition
$$
\frac{1}{\veps^m}\,\nabla_x\left(p(\vrho_\veps,\vtheta_\veps)\,-\,p(\wtilde{\vrho}_\veps,\oline{\vtheta})\right)\,=\,
\frac{1}{\veps^m}\,\nabla_x\left(p(\vrho_\veps,\vtheta_\veps)\,-\,p(1,\oline{\vtheta})\right)\,-\,
\frac{1}{\veps^m}\,\nabla_x\left(p(\wtilde{\vrho}_\veps,\oline{\vtheta})\,-\,p(1,\oline{\vtheta})\right)\,.
$$

Due to the smallness of the residual set \eqref{est:M_res-measure} and to estimates \eqref{est:rho_res} and \eqref{est:theta_res}, decomposing $p$ into essential and residual part
and then applying Proposition \ref{p:prop_5.2}, give us the convergence
$$
\frac{1}{\veps^m}\,\nabla_x\left(p(\vrho_\veps,\vtheta_\veps)\,-\,p(1,\oline{\vtheta})\right)\;\stackrel{*}{\rightharpoonup}\;
\nabla_x\left(\d_\varrho p(1,\oline{\vtheta})\,R\,+\,\d_\vtheta p(1,\oline{\vtheta})\,\Theta\right)
$$
in $L_T^\infty(H^{-1}_{\rm loc})$, for any $T>0$.
On the other hand, a Taylor expansion of $p(\,\cdot\,,\oline{\vtheta})$ up to the second order around $1$ gives, together with Proposition \ref{p:target-rho_bound}, the bound
$$
\left\|\frac{1}{\veps^m}\,\left(p(\wtilde{\vrho}_\veps,\oline{\vtheta})\,-\,p(1,\oline{\vtheta})\right)\,-\,
\d_\varrho p(1,\oline{\vtheta})\,\wtilde{r}_\veps\right\|_{L^\infty(K)}\,\leq\,C(K)\,\veps^m
$$
for any compact set $K\subset\Omega$. From the previous estimate we deduce that
$\left(p(\wtilde{\vrho}_\veps,\oline{\vtheta})\,-\,p(1,\oline{\vtheta})\right)/\veps^m\,\longrightarrow\,
\d_\varrho p(1,\oline{\vtheta})\,\wtilde{r}$ in e.g. $\mc{D}'\bigl(\R_+\times\Omega\bigr)$.

Putting all these facts together and keeping in mind relation \eqref{conv:r}, thanks to \eqref{eq:mom_rest_1} we finally find the celebrated \emph{Boussinesq relation}
	\begin{equation} \label{eq:rho-theta}
	\nabla_x\left(\d_\varrho p(1,\oline{\vtheta})\,\vrho^{(1)}\,+\,\d_\vtheta p(1,\oline{\vtheta})\,\Theta\right)\,=\,0
	\qquad\qquad\mbox{ a.e. in }\; \R_+\times \Omega\,.
	\end{equation}

\begin{remark} \label{r:F-G}
Notice that, dividing \eqref{prF} by $\veps^m$ and passing to the limit in it, one gets the identity
$$
\d_\varrho p(1,\oline{\vtheta})\,\nabla_x\wtilde{r} \,=\,\nabla_x G\,+\,\delta_2(m)\nabla_x F\,,
$$
where we have set $\delta_2(m)=1$ if $m=2$, $\delta_2(m)=0$ otherwise.

Therefore, the previous relation \eqref{eq:rho-theta} is actually equivalent to equality \eqref{eq:anis-lim_2}, 
which might be more familiar to the reader (see formula (5.10), Chapter 5 in \cite{F-N}).
\end{remark}



Up to now, the contribution of the fast rotation in the limit has not been seen: this is due to the fact that the incompressible
limit takes place faster than the high rotation limit, because $m>1$. Roughly speaking, the rotation term enters into the singular
perturbation operator as a ``lower order'' part; nonetheless, being singular, it will impose some conditions on the limit dynamics,
and one has to deal with it in the convergence process.

Let us make rigorous what we have just said. We test \eqref{meq} on $\veps\,\vec\phi$, where this time we take $\vec\phi\,=\,\curl\vec\psi$, for some smooth compactly supported $\vec\psi\,\in C^\infty_c\bigl([0,T[\,\times\Omega\bigr)$.
Once again, by uniform bounds we infer that the $\d_t$ term, the convective term and the viscosity term all converge to $0$ when $\veps\ra0$.
As for the pressure and the external forces, we repeat the same manipulations as before: making use of relation \eqref{prF} again, we are reconducted to work on
$$
\int^T_0\int_K\left(\frac{1}{\veps^{2m-1}}\nabla_x\left(p(\vrho_\veps,\vtheta_\veps)\,-\,p(\wtilde{\vrho}_\veps,\oline{\vtheta})\right)\,-\,
\frac{\vrho_\veps-\wtilde{\vrho}_\veps}{\veps}\,\nabla_x F\,-\,\frac{\vrho_\veps-\wtilde{\vrho}_\veps}{\veps^{m-1}}\nabla_x G\right)\cdot\vec\phi\,\dx\,dt\,,
$$
where we have supposed that $\Supp\vec\phi\subset[0,T[\,\times K$, for some compact set $K\subset\Omega$, and $\veps>0$ is small enough.
According  to \eqref{rr1}, the two forcing terms converge to $0$, in the limit for $\veps\ra0$; on the other hand, the first term (which has no chance to be bounded uniformly in $\veps$)
simply vanishes, due to the fact that $\vec\phi\,=\,{\rm curl}\,\vec\psi$.

Finally, using a priori bounds and properties \eqref{conv:r} and
\eqref{conv:u}, it is easy to see that the rotation term converges to $\int^T_0\int_K\e_3\times\vec{U}\cdot\vec\phi$.
In the end, passing to the limit for $\veps\ra0$ we find
$$
\mbb{H}\left(\e_3\times\vec{U}\right)\,=\,0\qquad\qquad\Longrightarrow\qquad\qquad \e_3\times\vec{U}\,=\,\nabla_x\Phi\,
$$
for some potential function $\Phi$. From this relation, it is standard to deduce that $\Phi=\Phi(t,x^h)$, i.e. $\Phi$ does not depend
on $x^3$, and that the same property is inherited by $\vec{U}^h\,=\,\bigl(U^1,U^2\bigr)$, i.e. $\vec{U}^h\,=\,\vec{U}^h(t,x^h)$. Furthermore, it is also easy to see that
the $2$-D flow given by $\vec{U}^h$ is incompressible, namely $\div_{\!h}\,\vec{U}^h\,=\,0$.
Combining this fact with \eqref{eq:div-free}, we infer that $\d_3 U^3\,=\,0$; on the other hand, thanks to the boundary condition
\eqref{bc1-2} we must have $\bigl(\vec{U}\cdot\vec{n}\bigr)_{|\d\Omega}\,=\,0$. Keeping in mind that
$\d\Omega\,=\,\bigl(\R^2\times\{0\}\bigr)\cup\bigl(\R^2\times\{1\}\bigr)$, we finally get $U^3\,\equiv\,0$,
whence \eqref{eq:anis-lim_1} finally follows.

Next, we observe that we can by now pass to the limit in the weak formulation \eqref{weak-ent} of equation \eqref{eiq}.
The argument being analogous to the one used in \cite{F-N} (see Paragraph 5.3.2), we will only sketch it.
First of all, testing \eqref{eiq} on $\varphi/\veps^m$, for some $\varphi\in C^\infty_c\bigl([0,T[\,\times\Omega\bigr)$, and using \eqref{ceq}, for $\veps>0$ small enough we get
\begin{align}
&-\int^T_0\!\!\int_K\vrho_\veps\left(\frac{s(\vrho_\veps,\vtheta_\veps)-s(1,\oline{\vtheta})}{\veps^m}\right)\d_t\varphi -
\int^T_0\!\!\int_K\vrho_\veps\left(\frac{s(\vrho_\veps,\vtheta_\veps)-s(1,\oline{\vtheta})}{\veps^m}\right)\vec{u}_\veps\cdot\nabla_x\varphi \label{weak:entropy}  \\
&+\int^T_0\!\!\int_K\frac{\kappa(\vtheta_\veps)}{\vtheta_\veps}\,\frac{1}{\veps^m}\,\nabla_x\vtheta_\veps\cdot\nabla_x\varphi-
\frac{1}{\veps^m}\,\langle\sigma_\veps,\varphi\rangle_{[\mc{M}^+,C]([0,T]\times K)} =
\int_K\vrho_{0,\veps}\left(\frac{s(\vrho_{0,\veps},\vtheta_{0,\veps})-s(1,\oline{\vtheta})}{\veps^m}\right)\varphi(0)\,.  \nonumber
\end{align}

To begin with, let us decompose
\begin{align}
&\vrho_\veps\left(\frac{s(\vrho_\veps,\vtheta_\veps)-s(1,\oline{\vtheta})}{\veps^m}\right) =  \label{eq:dec_rho-s} \\
&=[\vrho_\veps]_{\ess}\left(\frac{[s(\vrho_\veps,\vtheta_\veps)]_{\ess}-s(1,\oline{\vtheta})}{\veps^m}\right) + 
\left[\frac{\vrho_\veps}{\veps^m}\right]_{\res}\left([s(\vrho_\veps,\vtheta_\veps)]_{\ess}-s(1,\oline{\vtheta})\right) +
\left[\frac{\vrho_\veps\,s(\vrho_\veps,\vtheta_\veps)}{\ep^m}\right]_{\res}\,.  \nonumber
\end{align}
Thanks to \eqref{est:rho_res}, we discover that the second term in the right-hand side strongly converges to $0$ in
$L_T^\infty(L^{5/3}_{\rm loc})$. 
Also the third term converges to $0$ in the space $ L_T^2(L^{30/23}_{\rm loc})$, as a consequence of \eqref{est:M_res-measure} and \eqref{est:rho-s_res}.
Notice that these terms converge to $0$ even when multiplied by $\vec{u}_\veps$: to see this, it is enough to put \eqref{est:M_res-measure},
\eqref{est:rho-s_res}, \eqref{est:u-H^1} and the previous properties together.

As for the first term in the right-hand side of \eqref{eq:dec_rho-s}, Propositions \ref{p:prop_5.2} and \ref{p:target-rho_bound} and estimate \eqref{rr1} imply that it
weakly converges to $\d_\vrho s(1,\oline{\vtheta})\,R\,+\,\d_\vtheta s(1,\oline{\vtheta})\,\Theta$, where $R$ and $\Theta$ are defined respectively in \eqref{conv:r}
and \eqref{conv:theta}. On the other hand, an application of the Div-Curl Lemma (we refer to Paragraph 5.3.2 of \cite{F-N} for details) gives
$$
[\vrho_\veps]_{\ess}\left(\frac{[s(\vrho_\veps,\vtheta_\veps)]_{\ess}-s(1,\oline{\vtheta})}{\veps^m}\right)\,\vec{u}_\veps\,\rightharpoonup\,
\Bigl(\d_\vrho s(1,\oline{\vtheta})\,R\,+\,\d_\vtheta s(1,\oline{\vtheta})\,\Theta\Bigr)\,\vec{U}
$$
in the space $L_T^2(L^{3/2}_{\rm loc})$.  
In addition, from estimate \eqref{est:sigma} we deduce the convergence
$$
\frac{1}{\veps^m}\,\langle\sigma_\veps\,,\,\varphi\rangle_{[\mc{M}^+,C]([0,T]\times\Omega)}\,\longrightarrow\,0\,.
$$
Finally, a separation into essential and residual part of the coefficient $\kappa(\vtheta_\veps)/\vtheta_\veps$, together with \eqref{mu}, \eqref{est:theta-ess},
\eqref{est:theta_res},   \eqref{est:theta-Sob}  and \eqref{est:Dtheta_res} gives
$$
\frac{\kappa(\vtheta_\veps)}{\vtheta_\veps}\,\frac{1}{\veps^m}\,\nabla_x\vtheta_\veps\,\rightharpoonup\,
\frac{\kappa(\oline\vtheta)}{\oline\vtheta}\,\nabla_x\Theta\qquad\qquad\mbox{ in }\qquad L^2\bigl([0,T];L^{1}_{\rm loc}(\Omega)\bigr)\,.
$$

In the end, we have proved that equation \eqref{weak:entropy} converges, for $\veps\ra0$, to equation
\begin{align}
&-\int^T_0\int_\Omega\Bigl(\d_\vrho s(1,\oline{\vtheta})R + \d_\vtheta s(1,\oline{\vtheta})\,\Theta\Bigr)\left(\d_t\varphi + 
\vec{U}\cdot\nabla_x\varphi\right)\dxdt \,+ \label{eq:ent_bal_lim_1} \\
&\qquad\qquad+ \int^T_0\int_\Omega\frac{\kappa(\oline\vtheta)}{\oline\vtheta} \nabla_x\Theta\cdot\nabla_x\varphi\dxdt =
\int_\Omega\Bigl(\d_\vrho s(1,\oline{\vtheta})\,R_0\,+\,\d_\vtheta s(1,\oline{\vtheta})\,\Theta_0\Bigr)\,\varphi(0)\dx \nonumber
\end{align}
for all $\varphi \in C_c^\infty([0,T[\,\times\Omega)$, with $T>0$ any arbitrary time.
Relation \eqref{eq:ent_bal_lim_1} means that the quantity $\Upsilon$, defined in \eqref{eq:anis-lim_3}, is a weak solution of that equation,
related to the initial datum $\Upsilon_0:=\d_\vrho s(1,\oline\vtheta)\,R_0\,+\,\d_\vtheta s(1,\oline\vtheta)\,\Theta_0$.
Equation \eqref{eq:anis-lim_3} is in fact an equation for $\Theta$ only, keep in mind Remark \ref{r:lim delta theta}.
\qed
\end{proof}

\subsubsection{The case of isotropic scaling} \label{ss:constr_1}

We focus now on the case of isotropic scaling, namely $m=1$. Recall that, in this instance, we also set $F=0$. In this case, the fast rotation and weak compressibility
effects are of the same order; in turn, this allows to reach the so-called \emph{quasi-geostrophic balance} in the limit (see equation \eqref{eq:for q} below).
\begin{proposition}  \label{p:limit_iso}
Take $m=1$ and $F=0$ in system (NSF)$_\veps$.
Let $\left( \vre, \ue, \tem\right)_{\veps}$ be a family of weak solutions to (NSF)$_\veps$, associated with initial data
$\left(\vrho_{0,\veps},\vec u_{0,\veps},\vtheta_{0,\veps}\right)$ verifying the hypotheses fixed in Paragraph \ref{sss:data-weak}.
Let $(R, \vec{U},\Theta )$ be a limit point of the sequence $\left(R_{\veps} , \ue,\Theta_\veps\right)_{\veps}$, as identified in Subsection
\ref{sss:constr_prelim}.
Then
\begin{align}
&\vec{U}\,=\,\,\Big(\vec{U}^h\,,\,0\Big)\,,\qquad\qquad \mbox{ with }\qquad \vec{U}^h\,=\,\vec{U}^h(t,x^h)\quad \mbox{ and }\quad \div_{\!h}\,\vec{U}^h\,=\,0\,, \nonumber \\[1ex] 
&\vec{U}^h\,=\,\nabla^\perp_hq
\;\mbox{ a.e. in }\;\,]0,T[\, \times \Omega\,,
\quad\mbox{ with }\quad q\,=\,q(t,x^h)\,:=\,\d_\varrho p(1,\oline{\vtheta})R+\d_\vtheta p(1,\oline{\vtheta})\Theta-G-1/2\,,  \label{eq:for q} \\[1ex]
&\d_{t} \Upsilon +\divh\left( \Upsilon \vec{U}^{h}\right) -\frac{\kappa(\oline\vtheta)}{\oline\vtheta} \Delta \Theta =0\,,\qquad\quad\mbox{ with }\qquad 
\Upsilon_{|t=0}\,=\,\Upsilon_0\,, \nonumber 
\end{align}
where $ \Upsilon$ and $\Upsilon_0$ are the same quantities defined in Proposition \ref{p:limitpoint}.
\end{proposition}

\begin{proof}
Arguing as in the proof of Proposition \ref{p:limitpoint}, it is easy to pass to the limit in the continuity equation and in the entropy balance. In particular, we obtain again
equations \eqref{eq:div-free} and \eqref{eq:ent_bal_lim_1}.

The only changes concern the analysis of the momentum equation, written in its weak formulation \eqref{weak-mom}.
We start by testing it on $\veps\,\vec\phi$, for a smooth compactly supported $\vec\phi$. Similarly to what done above, the uniform bounds of Subsection \ref{ss:unif-est}
allow us to say that the only quantity which does not vanish in the limit is the sum of the terms involving the Coriolis force, the pressure and the gravitational force:
$$
\vec{e}_{3}\times \vrho_{\veps}\ue\,+\frac{\nabla_x \left( p(\vrho_\veps,\vtheta_\veps)-p(\widetilde{\vrho}_\veps,\vtheta_\veps)\right)}{\veps}\,-\,
\left(\vrho_\veps-\widetilde{\vrho}_\veps \right)\nabla_x G\,=\,\mc O(\veps)\,.
$$
From this relation, following the same computations performed in the proof of Proposition \ref{p:limitpoint}, in the limit $\veps\ra0$ we obtain that
$$ 
\vec{e}_{3}\times \vec{U}+\nabla_x\left(\d_\varrho p(1,\oline{\vtheta})\,\vrho^{(1)}\,+\,\d_\vtheta p(1,\oline{\vtheta})\,\Theta\right)\,=\,0 \qquad\qquad\mbox{ a.e. in }\; \R_+\times \Omega\,.
$$ 
After defining $q$ as in \eqref{eq:for q} and keeping Remark \ref{r:F-G} in mind, this equality can be equivalently written in the following way:
$$ 
\vec{e}_{3}\times \vec{U}+\nabla_xq\,=\,0 \qquad\qquad\mbox{ a.e. in }\; \R_+ \times \Omega\,.
$$ 
As done in the proof to Proposition \ref{p:limitpoint}, from this relation we immediately deduce that $q=q(t,x^h)$ and $\vec{U}^h=\vec{U}^h(t,x^h)$.
In addition, we get $\vec{U}^h=\nabla^\perp_hq$, whence we gather that $q$ can be viewed as a stream function for $\vec U^h$. 
Using \eqref{eq:div-free}, we infer that $\d_{3}U^{3}=0$, which in turn implies that $U^{3}\equiv0$, thanks to \eqref{bc1-2}.
The proposition is thus proved.
\qed
\end{proof}

\begin{remark} \label{r:q}
Notice that $q$ is defined up to an additive constant. We fix it to be $-1/2$, in order to compensate the vertical mean of $G$ and have a cleaner expression for $\lan q\ran$ (see 
Theorem \ref{th:m=1_F=0}). As a matter of fact, it is $\lan q\ran$ the natural quantity to look at, see also Subsection \ref{ss:limit_1} in this respect.
\end{remark}

\section{Convergence in presence of the centrifugal force}\label{s:proof}

In this section we complete the proof of Theorem \ref{th:m-geq-2}, in the case when $m\geq2$ and $F\neq0$. In the case $m>1$ and $F=0$, some arguments of the proof slightly change,
due to the absence of the (unbounded) centrifugal force: we refer to Section \ref{s:proof-1} below for more details.

\medbreak
The uniform bounds established in Subsection \ref{ss:unif-est} allow us to pass to the limit in the mass and entropy equations. Nonetheless, they are
not enough for proving convergence in the weak formulation of the momentum equation, the main problem relying on identifying the weak limit of the convective term
$\vrho_\veps\,\vec u_\veps\otimes\vec u_\veps$.
For this, we need to control the strong time oscillations of the solutions: this is the aim of Subsection \ref{ss:acoustic}. 
In the following subsection, by using a compensated compactness argument together with Aubin-Lions Lemma, we establish strong convergence of suitable quantities related to the velocity fields.
This property, which deeply relies on the structure of the wave system, allows us to pass to the limit in our equations (see Subsection \ref{ss:limit}).

\subsection{Analysis of the acoustic waves} \label{ss:acoustic}

The goal of the present subsection is to describe oscillations of solutions. First of all, we recast our equations into a wave system; there we also implement a localisation
procedure, due to the presence of the centrifugal force. Then, we establish uniform bounds for the quantities appearing in the wave system. Finally, we apply a regularisation
in space for all the quantities, which is preparatory in view of the computations in Subsection \ref{ss:convergence}.

\subsubsection{Formulation of the acoustic equation} \label{sss:wave-eq}

Let us define 
$$
\vec{V}_\veps\,:=\,\vrho_\veps\vec{u}_\veps\,.
$$
We start by writing the continuity equation in the form
\begin{equation} \label{eq:wave_mass}
\veps^m\,\d_t\vrho^{(1)}_\veps\,+\,\div\vec{V}_\veps\,=\,0\,.
\end{equation}
Of course, this relation, as well as the other ones which will follow, has to be read in the weak form.

Using continuity equation and resorting to the time lifting \eqref{lift0} of the measure $\sigma_\veps$, straightforward computations
lead us to the following form of the entropy balance:
$$
\veps^m\d_t\!\left(\vrho_\veps\,\frac{s(\vrho_\veps,\vtheta_\veps)-s(\wtilde{\vrho}_\veps,\oline{\vtheta})}{\veps^m}-
\frac{1}{\veps^m}\Sigma_\veps\right)\,=\,\veps^m\,\div\!\!\left(\frac{\kappa(\vtheta_\veps)}{\vtheta_\veps}
\frac{\nabla_x\vtheta_\veps}{\veps^m}\right)+s(\wtilde{\vrho}_\veps,\oline{\vtheta})\div\!\!\left(\vrho_\veps\,\vec{u}_\veps\right)-
\div\!\!\left(\vrho_\veps s(\vrho_\veps,\vtheta_\veps)\vec{u}_\veps\right),
$$
where, with a little abuse of notation, we use the identification
$\int_{\Omega_\veps}\Sigma_\veps\,\varphi\,dx\,=\,\langle\Sigma_\veps,\varphi\rangle_{[\mc{M}^+,C]}$. Next, since $\wtilde{\vrho}_\veps$
is smooth (recall relation \eqref{eq:target-rho} above), the previous equation can be finally written as
\begin{align}
&\veps^m\,\d_t\left(\vrho_\veps\,\frac{s(\vrho_\veps,\vtheta_\veps)-s(\wtilde{\vrho}_\veps,\oline{\vtheta})}{\veps^m}\,-\,
\frac{1}{\veps^m}\Sigma_\veps\right)\,=  \label{eq:wave_entropy} \\
&\qquad\quad=\,
\veps^m\,\biggl(\div\!\left(\frac{\kappa(\vtheta_\veps)}{\vtheta_\veps}\,\frac{\nabla_x\vtheta_\veps}{\veps^m}\right)\,-\,
\vrho_\veps\,\vec{u}_\veps\,\cdot\,\frac{1}{\veps^m}\,\nabla_x s(\wtilde{\vrho}_\veps,\oline{\vtheta})\,-\,
\div\!\left(\vrho_\veps\,\frac{s(\vrho_\veps,\vtheta_\veps)-s(\wtilde{\vrho}_\veps,\oline{\vtheta})}{\veps^m}\,\vec{u}_\veps\right)\biggr)\,.
\nonumber
\end{align}

Now, we turn our attention to the momentum equation. By \eqref{prF} we easily find
\begin{align}
&\veps^m\,\d_t\vec{V}_\veps\,+\,\nabla_x\left(\frac{p(\vrho_\veps,\vtheta_\veps)-p(\wtilde{\vrho}_\veps,\oline{\vtheta})}{\veps^m}\right)\,+\,\veps^{m-1}\,\e_3\times \vec V_\veps\,=\,
\veps^{2(m-1)}\frac{\vrho_\veps-\wtilde{\vrho}_\veps}{\veps^m}\nabla_x F\,+ \label{eq:wave_momentum} \\
&\qquad\qquad\qquad\qquad\qquad
+\,\veps^m\left(\div\mbb{S}\!\left(\vtheta_\veps,\nabla_x\vec{u}_\veps\right)\,-\,\div\!\left(\vrho_\veps\vec{u}_\veps\otimes\vec{u}_\veps\right)\,+\,
\frac{\vrho_\veps-\wtilde{\vrho}_\veps}{\veps^m}\nabla_x G\right)\,. \nonumber
\end{align}

At this point, let us introduce two real numbers $\mc{A}$ and $\mc{B}$, such that the following relations are satisfied:
\begin{equation} \label{relnum}
\mc{A}\,+\,\mc{B}\,\d_\vrho s(1,\oline{\vtheta})\,=\,\d_\vrho p(1,\oline{\vtheta})\qquad\mbox{ and }\qquad
\mc{B}\,\d_\vtheta s(1,\oline{\vtheta})\,=\,\d_\vtheta p(1,\oline{\vtheta})\,.
\end{equation}
Due to Gibbs' law \eqref{gibbs} and the structural hypotheses of Paragraph \ref{sss:structural} (see also Chapter 8 of \cite{F-N} and \cite{F-Scho}), we notice that
$\mc A$ is given by formula \eqref{def:A}, and $\mc A>0$.

Taking a linear combination of \eqref{eq:wave_mass} and \eqref{eq:wave_entropy}, with coefficients respectively $\mc{A}$ and $\mc{B}$,
and keeping in mind equation \eqref{eq:wave_momentum}, we finally get the wave system
\begin{equation} \label{eq:wave_syst}
\left\{\begin{array}{l}
       \veps^m\,\d_tZ_\veps\,+\,\mc{A}\,\div\vec{V}_\veps\,=\,\veps^m\,\left(\div\vec{X}^1_\veps\,+\,X^2_\veps\right) \\[1ex]
       \veps^m\,\d_t\vec{V}_\veps\,+\,\nabla_x Z_\veps\,+\,\veps^{m-1}\,\e_3\times \vec V_\veps\,=\,\veps^m\,\left(\div\mbb{Y}^1_\veps\,+\,\vec{Y}^2_\veps\,+\,\nabla_x Y^3_\veps\right)\,,
       \qquad\big(\vec{V}_\veps\cdot\vec n\big)_{|\d\Omega_\veps}\,=\,0\,,
       \end{array}
\right.
\end{equation}
where we have defined the quantities
\begin{eqnarray*}
Z_\veps & := & \mc{A}\,\vrho^{(1)}_\veps\,+\,\mc{B}\,\left(\vrho_\veps\,
\frac{s(\vrho_\veps,\vtheta_\veps)-s(\wtilde{\vrho}_\veps,\oline{\vtheta})}{\veps^m}\,-\,\frac{1}{\veps^m}\Sigma_\veps\right) \\
\vec{X}^1_\veps & := & \mc{B}\left(\frac{\kappa(\vtheta_\veps)}{\vtheta_\veps}\,\frac{\nabla_x\vtheta_\veps}{\veps^m}\,-\,
\vrho_\veps\,\frac{s(\vrho_\veps,\vtheta_\veps)-s(\wtilde{\vrho}_\veps,\oline{\vtheta})}{\veps^m}\,\vec{u}_\veps\right) \\
X^2_\veps & := & -\,\mc{B}\,\vrho_\veps\,\vec{u}_\veps\,\cdot\,\frac{1}{\veps^m}\,\nabla_x s(\wtilde{\vrho}_\veps,\oline{\vtheta}) \\
\mbb{Y}^1_\veps & := & \mbb{S}\!\left(\vtheta_\veps,\nabla\vec{u}_\veps\right)\,-\,\vrho_\veps\vec{u}_\veps\otimes\vec{u}_\veps \\
\vec{Y}^2_\veps & := & \frac{\vrho_\veps-\wtilde{\vrho}_\veps}{\veps^m}\nabla_x G\,+\,
\veps^{m-2}\,\frac{\vrho_\veps-\wtilde{\vrho}_\veps}{\veps^m}\nabla_x F \\
Y^3_\veps & := &\frac{1}{\veps^{m}}\left( \mc{A}\,\frac{\vrho_\veps-\wtilde{\vrho}_\veps}{\veps^m}\,+\mc{B}\,\vrho_\veps\,
\frac{s(\vrho_\veps,\vtheta_\veps)-s(\wtilde{\vrho}_\veps,\oline{\vtheta})}{\veps^m}\,-\,\mc{B}\,\frac{1}{\veps^m}\Sigma_\veps\,-\,
\frac{p(\vrho_\veps,\vtheta_\veps)-p(\wtilde{\vrho}_\veps,\oline{\vtheta})}{\veps^m}\right)\,.
\end{eqnarray*}

We remark that system \eqref{eq:wave_syst} has to be read in the weak sense: for any $\varphi\in C_c^\infty\bigl([0,T[\,\times\oline\Omega_\veps\bigr)$, one has
$$
-\,\veps^m\,\int^T_0\int_{\Omega_\veps} Z_\veps\,\d_t\varphi\,-\,\mc{A}\,\int^T_0\int_{\Omega_\veps} \vec{V}_\veps\cdot\nabla_x\varphi\,=\,
\veps^{m}\int_{\Omega_\veps} Z_{0,\veps}\,\varphi(0)\,+\,\veps^m\,\int^T_0\int_{\Omega_\veps}\left(-\,\vec{X}^1_\veps\cdot\nabla_x\varphi\,+\,X^2_\veps\,\varphi\right)\,,
$$
and also, for any $\vec{\psi}\in C_c^\infty\bigl([0,T[\,\times\oline\Omega_\veps;\R^3\bigr)$ such that $\big(\vec\psi \cdot \n_\veps\big)_{|\partial {\Omega_\veps}} = 0$, one has
\begin{align*}
&-\,\veps^m\,\int^T_0\int_{\Omega_\veps}\vec{V}_\veps\cdot\d_t\vec{\psi}\,-\,\int^T_0\int_{\Omega_\veps} Z_\veps\,\div\vec{\psi}\,+\,\veps^{m-1}\int^T_0\int_{\Omega_\veps} \e_3\times\vec V_\veps\cdot\vec\psi \\
&\qquad\qquad\qquad\qquad
=\,\veps^{m}\int_{\Omega_\veps}\vec{V}_{0,\veps}\cdot\vec{\psi}(0)\,+\,\veps^m\,\int^T_0\int_{\Omega_\veps}\left(-\,\mbb{Y}^1_\veps:\nabla_x\vec{\psi}\,+\,\vec{Y}^2_\veps\cdot\vec{\psi}\,-\,
Y^3_\veps\,\div\vec{\psi}\right)\,,
\end{align*}
where we have set
\begin{equation} \label{def:wave-data}
Z_{0,\veps}\,=\,\mc{A}\,\vrho^{(1)}_{0,\veps}\,+\,\mc{B}\,\left(\vrho_{0,\veps}\,
\frac{s(\vrho_{0,\veps},\vtheta_{0,\veps})-s(\wtilde{\vrho}_\veps,\oline{\vtheta})}{\veps^m}\right) \qquad\mbox{ and }\qquad
\vec{V}_{0,\veps}\,=\,\vrho_{0,\veps}\,\vec{u}_{0,\veps}\,.
\end{equation}

At this point, analogously to \cite{F-G-GV-N}, for any fixed $l>0$, let us introduce a smooth cut-off
\begin{align} 
&\chi_l\in C^\infty_c(\R^2)\quad \mbox{ radially decreasing}\,,\qquad \mbox{ with }\quad 0\leq\chi_l\leq1\,, \label{eq:cut-off} \\
&\quad\mbox{ such that }\quad \chi_l\equiv1 \ \mbox{ on }\ \B_l\,, \quad \chi_l\equiv0 \ \mbox{ out of }\ \B_{2l}\,,\quad \left|\nabla_{h}\chi_l(x^h)\right|\,\leq\,C(l)\ \ \forall\,x^h\in\R^2\,. \nonumber
\end{align}
Then we define
\begin{equation} \label{def:L-W}
\Lambda_{\veps,l}\,:=\,\chi_l\,Z_\veps\,=\,\chi_l\,\mc{A}\,\vrho^{(1)}_\veps\,+\,\chi_l\,\mc{B}\,\left(\vrho_\veps\,
\frac{s(\vrho_\veps,\vtheta_\veps)-s(\wtilde{\vrho}_\veps,\oline{\vtheta})}{\veps^m}\,-\,\frac{1}{\veps^m}\Sigma_\veps\right)
\quad\mbox{ and }\quad \vec W_{\veps,l}\,:=\,\chi_l\,\vec V_\veps\,.
\end{equation}
For notational convenience, in what follows we keep using the notation
$\Lambda_{\veps}$ and $\vec W_\veps$ instead of $\Lambda_{\veps,l}$ and $\vec W_{\veps,l}\, $, tacitly meaning the dependence on $l$.  
So system \eqref{eq:wave_syst} becomes
\begin{equation} \label{eq:wave_syst2}
\left\{\begin{array}{l}
       \veps^m\,\d_t\Lambda_\veps\,+\,\mc{A}\,\div\vec{W}_\veps\,=\,
       \veps^m f_\veps \\[1ex]
       \veps^m\,\d_t\vec{W}_\veps\,+\,\nabla_x \Lambda_\veps\,+\,\veps^{m-1}\,\e_3\times \vec W_\veps\,=\, 
       \veps^m \vec{G}_\veps\,,
       \qquad\big(\vec{W}_\veps\cdot\vec n\big)_{|\d\Omega_\veps}\,=\,0\,,
       \end{array}
\right.
\end{equation}
where we have defined $f_\veps\,:=\,\div\vec{F}^1_\veps\,+\,F^2_\veps\;$ and $\;\vec G_\veps\,:=\,\div\mbb{G}^1_\veps\,+\,\vec{G}^2_\veps\,+\,\nabla_x G^3_\veps$, with
\begin{align*}
 & \vec{F}^1_\veps\,=\,\chi_l\,\vec{X}^1_\veps\qquad\qquad\mbox{ and }\qquad\qquad
F^2_\veps\,=\,\chi_l X^2_\veps\,-\,\vec{X}^1_\veps\cdot\nabla_{x}\chi_l\,+\,\mc{A}\,\vec V_\veps\cdot \nabla_{x}\chi_l\,; \\
& \mbb{G}^1_\veps\,=\,\chi_l\,\mbb{Y}^1_\veps\;,\qquad
\vec G^2_\veps\,=\,\chi_l\,\vec Y^2_\veps\,+\,\left(\frac{Z_\veps}{\veps^{m}}-Y^3_\veps \right)\,\nabla_{x}\chi_l\,-\,^t\mbb{Y}^1_\veps\cdot \nabla_{x}\chi_l\qquad\mbox{ and }\qquad
G^3_\veps\,=\,\chi_l\,Y^3_\veps\,.
\end{align*}

\subsubsection{Uniform bounds} \label{sss:w-bounds}

Here we use estimates of Subsection \ref{ss:unif-est} in order to show uniform bounds for the solutions and the data in the wave equation \eqref{eq:wave_syst2}.
We start by dealing with the ``unknowns'' $\Lambda_\veps$ and $\vec W_\veps$.

\begin{lemma} \label{l:S-W_bounds}
Let $\bigl(\Lambda_\veps\bigr)_\veps$  and $\bigl(\vec W_\veps\bigr)_\veps$ be defined as above. Then, for any $T>0$ and all $\ep \in \, ]0,1]$, one has
$$
	\| \Lambda_\veps\|_{L^\infty_T(L^2+L^{5/3}+L^1+\mc{M}^+)} \leq c(l)\, ,\quad\quad
	 \| \vec W_\veps\|_{L^2_T(L^2+L^{30/23})} \leq c(l) \, .
$$
\end{lemma}

\begin{proof}
We start by writing $\vec W_\veps\,=\,\vec W^1_\veps\,+\,\vec W_\veps^2$, where
$$
\vec W^1_\veps\,:=\,\chi_l\,[\vrho_\veps]_{\ess}\,\vec u_\veps\qquad\qquad\mbox{ and }\qquad\qquad
\vec W^2_\veps\,:=\,\chi_l\,[\vrho_\veps]_{\res}\,\vec u_\veps\,.
$$
Since the density and temperature are uniformly bounded on the essential set, by property \eqref{est:u-H^1} we immediately infer that $\vec W_\veps^1$ is uniformly bounded in
$L_T^2(L^2)$. On the other hand, by \eqref{est:rho_res} and \eqref{est:u-H^1} again, we easily deduce that $\vec W_\veps^2$ is uniformly bounded in
$L_T^2(L^p)$, where $3/5+1/6\,=\,1/p$. The claim about $\vec W_\veps$ is hence proved.

Let us now consider $\Lambda_\veps$, defined in \eqref{def:L-W}.
First of all, owing to the bounds $\left\|\Sigma_\veps\right\|_{L^\infty_T(\mc{M}^+)}\,\leq\,C\,\|\sigma_\veps\|_{\mc{M}^+_{t,x}}$ and \eqref{est:sigma}, we have that
$$
\left\|\frac{1}{\veps^{2m}}\,\chi_l\,\Sigma_\veps\right\|_{L^\infty_T(\mc{M}^+)} \leq c(l)\,,
$$
uniformly in $\veps>0$. Next, we can write the following decomposition:
$$
\vrho_\veps\,\chi_l\,\frac{s(\vrho_\veps,\vtheta_\veps)-s(\wtilde{\vrho}_\veps,\oline{\vtheta})}{\veps^m}\,=\,
\frac{1}{\veps^m}\,\chi_l\,\left(\vrho_\veps\,s(\vrho_\veps,\vtheta_\veps)\,-\,\wtilde{\vrho}_\veps\,s(\wtilde{\vrho}_\veps,\oline{\vtheta})\right)\,-\,
\chi_l\,\vrho_\veps^{(1)}\,s(\wtilde{\vrho}_\veps,\oline{\vtheta})\,,
$$
where the latter term in the right-hand side is bounded in $L^\infty_T(L^2+L^{5/3})$ in view of \eqref{uni_varrho1} and Proposition~\ref{p:target-rho_bound}. Concerning the former term,
we can write it as
\begin{equation}\label{eq:ub_1}
\frac{1}{\veps^m}\chi_l \left(\varrho_\ep s(\vrho_\veps,\vtheta_\veps)-\varrho_\ep s( \wtilde{\vrho}_\veps,\oline\vtheta)\right)=
\frac{1}{\veps^m}\chi_l \bigl[\varrho_\ep s(\vrho_\veps,\vtheta_\veps)-\varrho_\ep s(\wtilde{\vrho}_\veps,\oline\vtheta)\bigr]_{\ess}+
\frac{1}{\veps^m}\chi_l \bigl[\varrho_\ep s(\vrho_\veps,\vtheta_\veps)\bigr]_{\res}\,,
\end{equation}
since the support of $\chi_l\varrho_\ep s(\wtilde{\vrho}_\veps,\oline\vtheta)$ is contained in the essential set by Proposition \ref{p:target-rho_bound}, for small enough $\veps$
(depending on the fixed $l>0$).
By \eqref{est:e-s_res}, the last term on the r.h.s. is uniformly bounded in $L^\infty_T(L^1)$; as  for the first term on the r.h.s., a Taylor expansion at the first order, together with
inequalities \eqref{est:rho_ess}, \eqref{est:theta-ess} and the structural restrictions on $s$, immediately yields its uniform boundedness in $L^\infty_T(L^2)$.

The lemma is hence completely proved.
\qed
\end{proof}

\medbreak
In the next lemma, we establish bounds for the source terms in the system of acoustic waves \eqref{eq:wave_syst2}.
\begin{lemma} \label{l:source_bounds}
For any $T>0$ fixed, let us define the following spaces:
\begin{itemize}
 \item $
\mc X_1\,:=\,L^2\Bigl([0,T];\big(L^2+L^{1}+L^{3/2}+L^{30/23}+L^{30/29}\big)(\Omega)\Bigr)$;
\item  $\mc X_2\,:=\,L^2\Bigl([0,T];\big(L^2+L^1+L^{4/3}\big)(\Omega)\Bigr)$;
\item  $\mc X_3\,:=\,\mc X_2\,+\,L^\infty\Bigl([0,T];\big(L^2+L^{5/3}+L^1\big)(\Omega)\Bigr)$;
\item $\mc X_4\,:=\,L^\infty\Bigl([0,T];\big(L^2+L^{5/3}+L^1+\mc{M}^+\big)(\Omega)\Bigr)$.
\end{itemize}
Then, for any $l>0$ fixed, one has the following bounds, uniformly in $\veps\in\,]0,1]$:
$$
\left\|\vec F^1_\veps\right\|_{\mc X_1}\,+\,\left\|F^2_\veps\right\|_{\mc X_1}\,+\,\left\|\mbb{G}^1_\veps\right\|_{\mc X_2}\,+\,\left\|\vec G^2_\veps\right\|_{\mc X_3}\,+\,
\left\|G^3_\veps\right\|_{\mc X_4}\,\leq\,C(l)\,,
$$
where the constant $C(l)>0$ depends only on $l$, but not on $\veps$.

In particular, the sequences $\bigl( f_\veps\bigr)_\veps$ and
$\bigl(\vec{G}_\veps\bigr)_\veps$, defined in system \eqref{eq:wave_syst2}, are uniformly bounded in the space $L^{2}\big([0,T];W^{-1,1}(\Omega)\big)$, 
thus 
in $L^{2}\big([0,T];H^{-s}(\Omega)\big)$, for all $s>5/2$.
\end{lemma}

\begin{proof}
We start by dealing with $\vec F^1_\veps$. By relations \eqref{est:D-theta} and \eqref{est:Dtheta_res}, it is easy to see that
$$
\left\|\frac{1}{\veps^m}\,\chi_l\,\frac{\kappa(\vtheta_\veps)}{\vtheta_\veps}\,\nabla_{x}\vtheta_\veps\right\|_{L^2_T(L^2+L^1)}\,\leq\,c(l)\,.
$$
On the other hand, the analysis of the term
$$
\vrho_\veps\,\chi_l\,\frac{s(\vrho_\veps,\vtheta_\veps)-s(\wtilde{\vrho}_\veps,\oline{\vtheta})}{\veps^m}\,\vec u_\veps
$$
is based on an analogous decomposition as used in the proof of Lemma \ref{l:S-W_bounds} and on uniform bounds of Paragraph \ref{sss:uniform}: these facts allow us to bound
it in $L^2_T(L^{3/2}+L^{30/23}+L^{30/29})$.  

According to its definition, the bounds for $F^2_\veps$ easily follow from the previous ones and Lemma \ref{l:S-W_bounds} (indeed, the analysis performed therein for $\vec{W}_\veps$
applies also to the terms in $\vec{V}_\veps=\vrho_\veps\vec{u}_\veps$ which appear in the definition of $F^2_\veps$), provided we show that
$$
\frac{1}{\veps^m}\,\left|\chi_l\,\nabla_{x}\wtilde{\vrho}_\veps\right|\,\leq\,C(l)\,.
$$
The previous bound immediately follows from the equation
$$
\nabla_{x}\wtilde{\vrho}_\veps\,=\,\frac{\wtilde{\vrho}_\veps}{\d_\vrho p(\wtilde{\vrho}_\veps,\oline\vtheta)}\,\left(\veps^{2(m-1)}\,\nabla_{x} F\,+\,\veps^m\,\nabla_{x} G\right)\,,
$$
which derives from \eqref{prF}, together with Proposition \ref{p:target-rho_bound} and the definitions given in \eqref{assFG}.

The bound on  $\mbb{G}^1_\veps$ is an immediate consequence of \eqref{est:Du} and \eqref{est:momentum}.

Let us focus now on the term $\vec G^2_\veps$. The control of the term $\,^t\mbb{Y}^1_\veps\cdot\nabla_{x}\chi_l$ is the same as above. The control of $\chi_l\vec Y^2_\veps$, instead,
gives rise to a bound in $L^\infty_T(L^2+L^{5/3})$: this is easily seen once we write
$$
\chi_l\,\vec Y^2_\veps\,=\,\chi_l\,\vrho_\veps^{(1)}\,\nabla_{x} G\,+\,\veps^{m-2}\,\chi_l\,\vrho_\veps^{(1)}\,\nabla_{x}F
$$
and we use \eqref{uni_varrho1} and \eqref{assFG}. Finally, we have the equality
\begin{equation*}
\begin{split}
\nabla_x \chi_l\,\left( \frac{Z_\veps}{\veps^{m}}-Y^3_\veps \right)&=\nabla_x \chi_{l}\,\left(\frac{p(\vrho_\veps,\vtheta_\veps)-p(\wtilde{\vrho}_\veps,\oline\vtheta)}{\veps^m}\right)\\
&=\nabla_x \chi_{l}\left[\frac{p(\vrho_\veps,\vtheta_\veps)-p(\wtilde{\vrho}_\veps,\oline\vtheta)}{\veps^m}\right]_\ess+\nabla_x \chi_{l}\left[\frac{p(\vrho_\veps,\vtheta_\veps)}{\veps^m}\right]_\res.
\end{split}
\end{equation*}
The second term in the last line is uniformly bounded in $L^\infty_T(L^1)$, in view of \eqref{est:rho_res} and \eqref{est:theta_res}.
For the first term, instead,
we can proceed as in \eqref{eq:ub_1}.

We switch our attention to the term $G^3_\veps$, whose analysis is more involved. 
%
By definition, we have
\begin{equation*}
\begin{split}
\chi_l\,Y^3_\veps  &:= \frac{1}{\veps^{m}}\,\chi_l\,\left( \mc{A}\,\frac{\vrho_\veps-\wtilde{\vrho}_\veps}{\veps^m}\,+\mc{B}\,\vrho_\veps\,
\frac{s(\vrho_\veps,\vtheta_\veps)-s(\wtilde{\vrho}_\veps,\oline{\vtheta})}{\veps^m}\,-\,\mc{B}\,\frac{1}{\veps^m}\Sigma_\veps\,-\,
\frac{p(\vrho_\veps,\vtheta_\veps)-p(\wtilde{\vrho}_\veps,\oline{\vtheta})}{\veps^m}\right) \\
&=\frac{1}{\veps^{m}}\,\chi_l\,\left( \mc{A}\,\frac{\vrho_\veps-\wtilde{\vrho}_\veps}{\veps^m}\,+\mc{B}\, \frac{s(\vrho_\veps,\vtheta_\veps)-s(\wtilde{\vrho}_\veps,\oline{\vtheta})}{\veps^m}\,-
\frac{p(\vrho_\veps,\vtheta_\veps)-p(\wtilde{\vrho}_\veps,\oline{\vtheta})}{\veps^m}\right)\\
&-\,\mc{B}\,\frac{1}{\veps^{2m}}\,\chi_l\,\Sigma_\veps\,+\,\mc{B}\,\chi_l\,\left(\frac{\vrho_\veps -1}{\veps^{m}}\right)\,
\frac{s(\vrho_\veps,\vtheta_\veps)-s(\wtilde{\vrho}_\veps,\oline{\vtheta})}{\veps^m}\,,
\end{split}
\end{equation*}
with $\mc A$ and $\mc B$ defined in \eqref{relnum}. 
Next, we write
\begin{equation*}
\begin{split}
s(\vrho_\veps,\vtheta_\veps)-s(\widetilde{\vrho}_\veps,\oline{\vtheta})&=s(\vrho_\veps,\vtheta_\veps)-s(1,\oline{\vtheta})+s(1,\oline \vtheta)-s(\widetilde{\vrho}_\veps,\oline{\vtheta})\\
&=\d_\vrho \, s(1,\oline \vtheta )\, (\vrho_\veps -1)+\d_\vtheta \, s(1,\oline \vtheta )\, (\vtheta_\veps -\oline \vtheta)+\frac{1}{2}\,{\rm Hess}(s)[\xi_1 ,\eta_1]\begin{pmatrix}
\vrho_\veps -1 \\ 
\vtheta_\veps -\oline\vtheta
\end{pmatrix}\cdot\begin{pmatrix}
\vrho_\veps -1 \\ 
\vtheta_\veps -\oline\vtheta
\end{pmatrix} \\
&+\d_\vrho \, s(1,\oline \vtheta )\, (1-\widetilde{\vrho}_\veps )+\frac{1}{2}\, \d_{\vrho \vrho}\,s(\xi_{2},\oline\vtheta)\, (\widetilde{\vrho}_\veps -1)^{2}\\
&=\d_\vrho \, s(1,\oline \vtheta )\, (\vrho_\veps - \widetilde{\vrho}_\veps)+\d_\vtheta \, s(1,\oline \vtheta )\, (\vtheta_\veps -\oline \vtheta)\\
&+\frac{1}{2}\left({\rm Hess}(s)[\xi_1 ,\eta_1]\begin{pmatrix}
\vrho_\veps -1 \\ 
\vtheta_\veps -\oline\vtheta
\end{pmatrix}\cdot\begin{pmatrix}
\vrho_\veps -1 \\ 
\vtheta_\veps -\oline\vtheta
\end{pmatrix}+\d_{\vrho \vrho}\,s(\xi_{2},\oline \vtheta)\, (\widetilde{\vrho}_\veps -1)^{2}\right)\,,
\end{split}
\end{equation*}
where $\xi_1,\xi_2,\eta_1$ are suitable points between $1$ and $\vrho_\veps$, $1$ and $\widetilde{\vrho}_\veps$, $\oline \vtheta$ and $ \vtheta_\veps $ respectively, and
we have denoted by ${\rm Hess}(s)[\xi,\eta]$ the Hessian matrix of the function $s$ with respect to its variables $\big(\vrho,\vtheta\big)$, computed at the point $(\xi,\eta)$.
Analogously, for the pressure term we have
\begin{equation*}
\begin{split}
p(\vrho_\veps,\vtheta_\veps)-p(\widetilde{\vrho}_\veps,\oline{\vtheta}) 
&=\d_\vrho \, p(1,\oline \vtheta )\, (\vrho_\veps - \widetilde{\vrho}_\veps)+\d_\vtheta \, p(1,\oline \vtheta )\, (\vtheta_\veps -\oline \vtheta)\\
&+\frac{1}{2}\left({\rm Hess}(p)[\xi_3 ,\eta_2]\begin{pmatrix}
\vrho_\veps -1 \\ 
\vtheta_\veps -\oline\vtheta
\end{pmatrix}\cdot\begin{pmatrix}
\vrho_\veps -1 \\ 
\vtheta_\veps -\oline\vtheta
\end{pmatrix}+\d_{\vrho \vrho}\,p(\xi_{4},\oline \vtheta)\, (\widetilde{\vrho}_\veps -1)^{2}\right)\,,
\end{split}
\end{equation*}
where $\xi_3,\xi_4,\eta_2$ are still between $1$ and $\vrho_\veps$, $1$ and $\widetilde{\vrho}_\veps$, $\oline \vtheta$ and $ \vtheta_\veps $ respectively.
Using now \eqref{relnum}, we find that the first order terms cancel out, and we are left with
\begin{equation*}
\begin{split}
\chi_l\,Y^3_\veps  &=\frac{\mc{B}}{2\veps^{2m}}\,\chi_l\,\left( \,{\rm Hess}(s)[\xi_1 ,\eta_1]\begin{pmatrix}
\vrho_\veps -1 \\ 
\vtheta_\veps -\oline\vtheta
\end{pmatrix}\cdot\begin{pmatrix}
\vrho_\veps -1 \\ 
\vtheta_\veps -\oline\vtheta
\end{pmatrix}+\d_{\vrho \vrho}\,s(\xi_{2},\oline \vtheta)\, (\widetilde{\vrho}_\veps -1)^{2}\right)\\
&-\,\frac{1}{2\veps^{2m}}\,\chi_l\,\left({\rm Hess}(p)[\xi_3 ,\eta_2]\begin{pmatrix}
\vrho_\veps -1 \\ 
\vtheta_\veps -\oline\vtheta
\end{pmatrix}\cdot\begin{pmatrix}
\vrho_\veps -1 \\ 
\vtheta_\veps -\oline\vtheta
\end{pmatrix}+\d_{\vrho \vrho}\,p(\xi_{4},\oline \vtheta)\, (\widetilde{\vrho}_\veps -1)^{2}\right)\\
&-\,\frac{\mc B}{\veps^{2m}}\,\chi_l\,\Sigma_\veps\,+\,\mc{B}\,\chi_l\,\left(\frac{\vrho_\veps -1}{\veps^{m}}\right)\,
\frac{s(\vrho_\veps,\vtheta_\veps)-s(\wtilde{\vrho}_\veps,\oline{\vtheta})}{\veps^m}\, .
\end{split}
\end{equation*} 
Thanks to the uniform bounds establish in Paragraph \ref{sss:uniform} and the decomposition into essential and residual parts, the claimed control in the space $\mc X_4$ follows.
\qed
\end{proof}

\subsubsection{Regularization and description of the oscillations}\label{sss:w-reg}

Before going on, following \cite{F-N_AMPA} and \cite{F-N_CPDE} (see also \cite{Ebin}), it is convenient to reformulate our problem (NSF)$_\veps$, supplemented with complete slip boundary
conditions \eqref{bc1-2} and \eqref{bc3}, in a completely equivalent way, in the domain
$$
\wtilde{\Omega}_\veps\,:=\,{B}_{L_\veps} (0) \times \mbb{T}^1\,,\qquad\qquad\mbox{ with }\qquad\mbb{T}^1\,:=\,[-1,1]/\sim\,,
$$
where $\sim$ denotes the equivalence relation which identifies $-1$ and $1$.
For this, it is enough to extend $\varrho_\veps$, $\vtheta_\veps$, and $\vec u_\veps^h$ as even functions with respect to $x^{3}$, $u_\veps^3$ and $G$ as odd functions.

Correspondingly, we consider also the wave system \eqref{eq:wave_syst2} to be satisfied in the new domain $\wtilde\Omega_\veps$. It goes without saying that the uniform bounds
established above hold true also when replacing $\Omega$ with $\wtilde\Omega$, where we have set
$$\wtilde\Omega\,:=\,\R^2 \times \mbb{T}^1\,.$$
Notice that the wave speed in \eqref{eq:wave_syst2} is proportional to $\veps^{-m}$, while, in view of assumption \eqref{dom}, the domains $\wtilde\Omega_\veps$ are expanding at speed proportional
to $\veps^{-m-\de}$, for some $\delta>0$.
Therefore, no interactions of the acoustic-Poincar\'e waves with the boundary of $\wtilde\Omega_\veps$ take place (see also Remark \ref{r:speed-waves} in this respect),
for any finite time $T>0$ and sufficiently small $\ep>0$.
Thanks to this fact and the spatial localisation given by the cut-off function $\chi_l$, we can assume that the wave system \eqref{eq:wave_syst2} is satisfied
(still in a weak sense) on the whole $\wtilde\Omega$.

Now, for any $M\in\N$ let us consider the low-frequency cut-off operator ${S}_{M}$ of a Littlewood-Paley decomposition, as introduced in \eqref{eq:S_j} below. We define 
\begin{equation*}
\Lambda_{\varepsilon ,M}={S}_{M}\Lambda_{\veps}\qquad\qquad \text{ and }\qquad\qquad \vec{W}_{\veps ,M}={S}_{M}\vec{W}_{\veps}\, .
\end{equation*} 
The following result holds true.
Recall that we are omitting from the notation the dependence of all quantities on $l>0$, due to multiplication by the cut-off function $\chi_l$ fixed above.
\begin{proposition} \label{p:prop approx}
For any $T>0$, we have the following convergence properties, in the limit $M\rightarrow +\infty $:
\begin{equation}\label{eq:approx var}
\begin{split}
&\sup_{0<\veps\leq1}\, \left\|\Lambda_{\varepsilon }-\Lambda_{\varepsilon ,M}\right\|_{L^{\infty}([0,T];H^{s})}\longrightarrow 0\quad  \forall s<-3/2-\delta \\
&\sup_{0<\veps\leq1}\, \left\|\vec{W}_{\varepsilon }-\vec{W}_{\varepsilon ,M}\right\|_{L^{\infty}([0,T];H^{s})}\longrightarrow 0\quad \forall s<-4/5-\delta\,,
\end{split}
\end{equation}
for any $\delta >0$.
Moreover, for any $M>0$, the couple $(\Lambda_{\veps ,M},W_{\veps ,M})$ satisfies the approximate wave equations
\begin{equation}\label{eq:approx wave}
\left\{\begin{array}{l}
       \veps^m\,\d_t\Lambda_{\veps ,M}\,+\,\mc{A}\,\div\vec{W}_{\veps ,M}\,=\,\veps^m\,f_{\veps ,M} \\[1ex]
       \veps^m\,\d_t\vec{W}_{\veps ,M}\,+\veps^{m-1}\,e_{3}\times \vec{W}_{\veps ,M}+\,\nabla_x \Lambda_{\veps,M}\,=\,\veps^m\,\vec G_{\veps ,M}\, ,
       \end{array}
\right.
\end{equation}
where $(f_{\veps ,M})_{\veps}$ and $(\vec G_{\veps ,M})_{\veps}$ are families of smooth (in the space variables) functions satisfying, for any $s\geq0$, the uniform bounds
\begin{equation}\label{eq:approx force}
\sup_{0<\veps\leq1}\, \left\|f_{\veps ,M}\right\|_{L^{2}([0,T];H^{s})}\,+\,\sup_{0<\veps\leq1}\,\left\|\vec G_{\veps ,M}\right\|_{L^{2}([0,T];H^{s})}\,\leq\, C(l,s,M)\,,
\end{equation}
where the constant $C(l,s,M)$ depends on the fixed values of $l>0$, $s\geq 0$ and $M>0$, but not on $\veps>0$.
\end{proposition}

\begin{proof}
Thanks to characterization \eqref{eq:LP-Sob} of $H^{s}$, properties \eqref{eq:approx var} are straightforward consequences of the uniform bounds establish in Subsection \ref{sss:w-bounds}.

Next, applying the operator ${S}_{M}$ to \eqref{eq:wave_syst2} immediately gives us system \eqref{eq:approx wave} where we have set 
\begin{equation*}
f_{\veps ,M}:={S}_{M}\left(\div\vec{F}^1_\veps\,+\,F^2_\veps\right)\qquad \text{ and }\qquad \vec G_{\veps ,M}:={S}_{M}\left(\div\mbb{G}^1_\veps\,+\,\vec{G}^2_\veps\,+\,\nabla_x G^3_\veps\right)\,.
\end{equation*}
By these definitions and the uniform bounds
established in Lemma \ref{l:source_bounds}, thanks to \eqref{eq:LP-Sob} it is easy to verify inequality \eqref{eq:approx force}. 
\qed
\end{proof}

\medbreak
We also have an important decomposition for the approximated velocity fields and their $\curl$.
\begin{proposition} \label{p:prop dec}
For any $M>0$ and any $\veps\in\,]0,1]$, the following decompositions hold true:
\begin{equation*}
\vec{W}_{\veps ,M}\,=\,
\veps^{m}\vec{t}_{\veps ,M}^{1}+\vec{t}_{\veps ,M}^{2}\qquad\mbox{ and }\qquad
\curl_{x}\vec{W}_{\veps ,M}=\veps^{m}\vec{T}_{\veps ,M}^{1}+\vec{T}_{\veps ,M}^{2}\,,
\end{equation*}
where, for any $T>0$ and $s\geq 0$, one has 
\begin{align*}
&\left\|\vec{t}_{\veps ,M}^{1}\right\|_{L^{2}([0,T];H^{s})}+\left\|\vec{T}_{\veps ,M}^{1}\right\|_{L^{2}([0,T];H^{s})}\leq C(l,s,M) \\
&\left\|\vec{t}_{\veps ,M}^{2}\right\|_{L^{2}([0,T];H^{1})}+\left\|\vec{T}_{\veps ,M}^{2}\right\|_{L^{2}\left([0,T];L^2\right)}\leq C(l)\,,
\end{align*}
for suitable positive constants $C(l,s,M)$ and $C(l)$, which are uniform with respect to $\veps\in\,]0,1]$.
\end{proposition}

\begin{proof}
We start by defining
\begin{equation} \label{eq:t-T}
\vec{t}_{\veps,M}^{1}\,:=\,{S}_{M}\left(\chi_{l}\left(\frac{\vrho_\veps -1}{\veps^{m}}\right)\vec{u}_{\veps}\right) \qquad\mbox{ and }\qquad
\vec{t}_{\veps,M}^{2}\,:=\,{S}_{M}\left(\chi_{l}\vec{u}_{\veps}\right)\,.
\end{equation}
Then, it is apparent that $\vec{W}_{\veps ,M}\,=\,\veps^m\vec t_{\veps,M}^{1}\,+\,\vec t_{\veps,M}^{2}$.
The decomposition of $\curl_x\vec W_{\veps,M}$ is also easy to get, if we set $\vec T_{\veps,M}^j\,:=\,\curl_x\vec t_{\veps,M}^j$, for $j=1,2$.
We have to prove uniform bounds for all those terms.
But this is an easy verification, thanks to the $L^\infty_T(L^{2}_{\rm loc}+L^{5/3}_{\rm loc})$ bound on $R_\veps$ and the $L^2_T(H^{1}_{\rm loc})$ bound on $\vec{u}_{\veps}$, for any fixed time $T>0$ 
(recall the estimates obtained in Subsection \ref{ss:unif-est} above).
\qed
\end{proof}


\subsection{Convergence of the convective term} \label{ss:convergence}
We are finally able to handle the convective term. The first step is to reduce the study to the case of smooth vector fields $\vec{W}_{\veps ,M}$.

\begin{lemma} \label{lem:convterm}
Let $T>0$. For any $\vec{\psi}\in C_c^\infty\bigl([0,T[\,\times\wtilde\Omega;\R^3\bigr)$, we have 
\begin{equation*}
\lim_{M\rightarrow +\infty} \limsup_{\veps \rightarrow 0}\left|\int_{0}^{T}\int_{\wtilde\Omega} \vrho_\veps\,\vec{u}_\veps\otimes \vec{u}_\veps: \nabla_{x}\vec{\psi}\, dx \, dt-
\int_{0}^{T}\int_{\wtilde\Omega} \vec{W}_{\veps ,M}\otimes \vec{W}_{\veps,M}: \nabla_{x}\vec{\psi}\, dx \, dt\right|=0\, .
\end{equation*}
\end{lemma}

\begin{proof}
Let $\vec \psi\in C_c^\infty\bigl(\R_+\times\wtilde\Omega;\R^3\bigr)$, and let $\Supp\vec\psi\subset[0,T]\times K$ for some compact set $K\subset\wtilde\Omega$. Then, we take $l>0$ in \eqref{eq:cut-off} so large that
$K\subset \wtilde{\mbb{B}}_{l}\,:=\,B_l(0)\times\T$.
Therefore, using the notation introduced in \eqref{def_deltarho}, we get
$$
\int_{0}^{T}\int_{\wtilde\Omega} \vrho_\veps\,\vec{u}_\veps\otimes \vec{u}_\veps: \nabla_{x}\vec{\psi}\,=\,
\int_{0}^{T}\int_{K}(\chi_l\,\vec{u}_\veps)\otimes\vec{u}_\veps:\nabla_{x}\vec{\psi}\,+\veps^{m}\int_{0}^{T}\int_{K}R_\veps\,\vec{u}_\veps\otimes \vec{u}_\veps:\nabla_{x}\vec{\psi}\,.
$$
Notice that, as a consequence of the uniform bounds for $\big(\vec{u}_{\veps}\big)_\veps$ in $L^{2}_{T}(L^{6}_{\rm loc})$ and for $\big(R_{\veps}\big)_\veps$ in $L^{\infty}_{T}(L_{\rm loc}^{5/3})$
(recall \eqref{uni_varrho1} above), the second integral in the right-hand side is of order $\veps^{m}$. As for the first one, 
using the definition in \eqref{eq:t-T}, we can write
$$
\int_{0}^{T}\int_{K}(\chi_l\,\vec{u}_\veps)\otimes \vec{u}_\veps:\nabla_{x}\vec{\psi}\,=\,\int_{0}^{T}\int_{K}\vec{t}^2_{\veps,M}\otimes\vec{u}_\veps:\nabla_{x}\vec{\psi}
+\int_{0}^{T}\int_{K} \,(Id-{S}_{M})(\chi_l\,\vec{u}_\veps)\otimes\vec{u}_\veps: \nabla_{x}\vec{\psi}\,.
$$
Observe that, in view of characterisation \eqref{eq:LP-Sob}, one has the property
\begin{equation*}
\left\|(Id-{S}_{M})(\chi_l\,\vec{u}_\veps)\right\|_{L_{T}^{2}(L^{2})}\,\leq\,C\,2^{-M}\,\left\|\nabla_{x}(\chi_l\,\vec{u}_\veps)\right\|_{L_{T}^{2}(L^{2})}\,\leq\,C(l)\,2^{-M}\,.
\end{equation*}
Therefore, it is enough to consider the first term in the right-hand side of the last relation: we have
$$
\int_{0}^{T}\int_{K}\vec{t}^2_{\veps,M}\otimes \vec{u}_\veps:\nabla_{x}\vec{\psi}\,=\,\int_{0}^{T}\int_K\vec{t}^2_{\veps,M}\otimes\vec{t}^2_{\veps,M}:\nabla_{x}\vec{\psi}\,+\,
\int_{0}^{T}\int_{K} \,\vec{t}^2_{\veps,M}\otimes(Id-{S}_{M})(\chi_l\,\vec{u}_\veps): \nabla_{x}\vec{\psi}\,,
$$
where, for the same reason as before, we gather that
\begin{equation*}
\lim_{M\rightarrow +\infty}\limsup_{\veps \rightarrow 0}\left|\int_{0}^{T}\int_{K}\vec{t}^2_{\veps,M}\otimes (Id-{S}_{M})(\chi_l\,\vec{u}_\veps): \nabla_{x}\vec{\psi}\right|=0\, .
\end{equation*}

It remains us to consider the integral 
$$
\int_{0}^{T}\int_K\vec{t}^2_{\veps,M}\otimes\vec{t}^2_{\veps,M}:\nabla_{x}\vec{\psi}\,=\,\int_{0}^{T}\int_{K} \vec{W}_{\veps ,M}\otimes \vec t^2_{\veps,M}: \nabla_{x}\vec{\psi}
-\veps^{m}\int_{0}^{T}\int_{K}\vec t^1_{\veps,M}\otimes \vec t^2_{\veps,M}: \nabla_{x}\vec{\psi}\,,
$$
where we notice that, owing to Proposition \ref{p:prop dec}, the latter term in the right-hand side is of order $\veps^{m}$, so it vanishes at the limit.
As a last step, we write
$$
\int_{0}^{T}\int_{K} \vec{W}_{\veps ,M}\otimes \vec t^2_{\veps,M}: \nabla_{x}\vec{\psi}\,=\,
\int_{0}^{T}\int_{K} \vec{W}_{\veps ,M}\otimes \vec W_{\veps,M}: \nabla_{x}\vec{\psi}\,-\,\veps^m\int_{0}^{T}\int_{K} \vec{W}_{\veps ,M}\otimes \vec t^1_{\veps,M}: \nabla_{x}\vec{\psi}\,.
$$
Using Lemma \ref{l:S-W_bounds} together with Bernstein's inequalities of Lemma \ref{l:bern}, we see that the latter integral in the right-hand side is of order $\veps^{m}$.
This concludes the proof of the lemma.
\qed
\end{proof}

\medbreak
From now on, in order to avoid the appearance of (irrelevant) multiplicative constants everywhere, we suppose that the torus $\T$ has been normalised so that its Lebesgue measure is equal to $1$.

In view of the previous lemma and of Proposition \ref{p:limitpoint}, for any test-function
\begin{equation} \label{eq:test-f}
\vec\psi\in C_c^\infty\big([0,T[\,\times\wtilde\Omega;\R^3\big)\qquad\qquad \mbox{ such that }\qquad \div\vec\psi=0\quad\mbox{ and }\quad \d_3\vec\psi=0\,,
\end{equation}
we have to pass to the limit in the term 
\begin{align*}
-\int_{0}^{T}\int_{\wtilde\Omega} \vec{W}_{\veps ,M}\otimes \vec{W}_{\veps ,M}: \nabla_{x}\vec{\psi}\,&=\,\int_{0}^{T}\int_{\wtilde\Omega} \div\left(\vec{W}_{\veps ,M}\otimes
\vec{W}_{\veps ,M}\right) \cdot \vec{\psi}\,.
\end{align*}

Notice that the integration by parts above is well-justified, since all the quantities inside the integrals are smooth with respect to the space variable. At this point, we observe that,
resorting to the notation introduced in \eqref{eq:decoscil}, we can write
$$
\int_{0}^{T}\int_{\wtilde\Omega} \div\left(\vec{W}_{\veps ,M}\otimes \vec{W}_{\veps ,M}\right) \cdot \vec{\psi}\,=\,
\int_{0}^{T}\int_{\R^2} \left(\mc{T}_{\veps ,M}^{1}+\mc{T}_{\veps, M}^{2}\right)\cdot\vec{\psi}^h\,,
$$
where we have defined the terms
\begin{equation} \label{def:T1-2}
\mc T^1_{\veps,M}\,:=\, \divh\left(\langle \vec{W}_{\veps ,M}^{h}\rangle\otimes \langle \vec{W}_{\veps ,M}^{h}\rangle\right)\qquad \mbox{ and }\qquad
\mc T^2_{\veps,M}\,:=\, \divh\left(\langle \widetilde{\vec{W}}_{\veps ,M}^{h}\otimes \widetilde{\vec{W}}_{\veps ,M}^{h}\rangle \right)\,.
\end{equation}
So, it is enough to focus on each of them separately.
For notational convenience, from now on we will generically denote by $\mc{R}_{\veps ,M}$ any remainder term, that is any term satisfying the property
\begin{equation} \label{eq:reminder}
\lim_{M\rightarrow +\infty}\limsup_{\veps \rightarrow 0}\left|\int_{0}^{T}\int_{\wtilde\Omega}\mc{R}_{\veps ,M}\cdot \vec{\psi}\, dx \, dt\right|=0
\end{equation}
for all test functions $\vec{\psi}\in C_c^\infty\bigl([0,T[\,\times\wtilde\Omega;\R^3\bigr)$ as in \eqref{eq:test-f}. 

\subsubsection{The analysis of the $\mc{T}_{\veps ,M}^{1}$ term}\label{sss:term1}
We start by dealing with $\mc T^1_{\veps,M}$. Standard computations give
\begin{align}
\mc{T}_{\veps ,M}^{1}\,&=\,\divh\left(\langle \vec{W}_{\veps ,M}^{h}\rangle\otimes \langle \vec{W}_{\veps ,M}^{h}\rangle\right)=
\divh\langle \vec{W}_{\veps ,M}^{h}\rangle\, \langle \vec{W}_{\veps ,M}^{h}\rangle+\langle \vec{W}_{\veps ,M}^{h}\rangle \cdot \nabla_{h}\langle \vec{W}_{\veps ,M}^{h}\rangle \label{eq:T1} \\
&=\divh\langle \vec{W}_{\veps ,M}^{h}\rangle\, \langle \vec{W}_{\veps ,M}^{h}\rangle+\frac{1}{2}\, \nabla_{h}\left(\left|\langle \vec{W}_{\veps ,M}^{h}\rangle\right|^{2}\right)+
\curlh\langle \vec{W}_{\veps ,M}^{h}\rangle\,\langle \vec{W}_{\veps ,M}^{h}\rangle^{\perp}\,. \nonumber
\end{align}
Notice that we can forget about the second term because it is a perfect gradient and we are testing against divergence-free test functions.

For the first term, we take advantage of system \eqref{eq:approx wave}: averaging the first equation with respect to $x^{3}$ and multiplying it by $\langle \vec{W}_{\veps ,M}\rangle$, we arrive at
$$
\divh\langle \vec{W}_{\veps ,M}^{h}\rangle\,\langle \vec{W}_{\veps ,M}^{h}\rangle\,=\,-\frac{\veps^{m}}{\mc{A}}\d_t\langle \Lambda_{\veps ,M}\rangle \langle \vec{W}_{\veps ,M}^{h}\rangle+
\frac{\veps^{m}}{\mc{A}} \langle f_{\veps ,M}^{h}\rangle \langle \vec{W}_{\veps ,M}^{h}\rangle\,=\,
\frac{\veps^{m}}{\mc{A}}\langle \Lambda_{\veps ,M}\rangle \d_t \langle \vec{W}_{\veps ,M}^{h}\rangle +\mc{R}_{\veps ,M}\,.
$$
We remark that the term presenting the total time derivative is in fact a remainder.
We use now the horizontal part of \eqref{eq:approx wave}, where we take the vertical average and then multiply by $\langle \Lambda_{\veps ,M}\rangle$: we gather
\begin{align*}
\frac{\veps^{m}}{\mc{A}}\langle \Lambda_{\veps ,M}\rangle \d_t \langle \vec{W}_{\veps ,M}^{h}\rangle &=
-\frac{1}{\mc{A}} \langle \Lambda_{\veps ,M}\rangle \nabla_{h}\langle \Lambda_{\veps ,M}\rangle+\frac{\veps^{m}}{\mc{A}}\langle \Lambda_{\veps ,M}\rangle \langle \vec G_{\veps ,M}^{h}\rangle-
\frac{\veps^{m-1}}{\mc{A}}\langle \Lambda_{\veps ,M}\rangle\langle \vec{W}_{\veps ,M}^{h}\rangle^{\perp}\\
&=-\frac{\veps^{m-1}}{\mc{A}}\langle \Lambda_{\veps ,M}\rangle\langle \vec{W}_{\veps ,M}^{h}\rangle^{\perp}-\frac{1}{2\mc{A}}  \nabla_{h}\left( \left| \langle \Lambda_{\veps ,M}\rangle \right|^{2}\right)+\mc{R}_{\veps ,M}\\
&=-\frac{\veps^{m-1}}{\mc{A}}\langle \Lambda_{\veps ,M}\rangle\langle \vec{W}_{\veps ,M}^{h}\rangle^{\perp}+\mc{R}_{\veps ,M}\, ,
\end{align*}
where we repeatedly exploited the properties proved in Proposition \ref{p:prop approx} and we included in the remainder term also the perfect gradient.
Inserting this relation into \eqref{eq:T1}, we find
\begin{equation*}
\mc{T}_{\veps ,M}^{1}= \g_{\veps,M}\,\langle\vec{W}_{\veps,M}^{h}\rangle^{\perp}+\mc{R}_{\veps,M}\,,
\qquad\qquad\mbox{ with }\qquad \gamma_{\veps, M}:=\curlh\langle \vec{W}_{\veps ,M}^{h}\rangle\,-\,\frac{\veps^{m-1}}{\mc{A}}\langle \Lambda_{\veps ,M}\rangle\,.
\end{equation*}


We observe that, for passing to the limit in $\mc{T}_{\veps ,M}^{1}$, there is no other way than finding some strong convergence property for $\vec W_{\veps,M}$. 
Such a property is in fact hidden in the structure of the wave system \eqref{eq:approx wave}: in order to exploit it, some work on the term $\g_{\veps,M}$ is needed.
We start by rewriting the vertical average of the first equation in \eqref{eq:approx wave} as
\begin{equation*}
\frac{\veps^{2m-1}}{\mc{A}}\,\d_t \langle \Lambda_{\veps ,M} \rangle\,+\,\veps^{m-1}\div_{h} \langle \vec{W}_{\veps ,M}^{h}\rangle\,=\,\frac{\veps^{2m-1}}{\mc{A}}\, \langle f_{\veps ,M}^{h}\rangle\,.
\end{equation*}
On the other hand, taking the vertical average of the horizontal components of \eqref{eq:approx wave} and then applying $\curlh$, we obtain the relation
\begin{equation*}
\veps^m\,\d_t\curlh\langle \vec{W}_{\veps ,M}^{h}\rangle\,+\veps^{m-1}\,\divh\langle \vec{W}_{\veps ,M}^{h}\rangle\, =\,\veps^m \curlh\langle\vec G_{\veps ,M}^{h}\rangle\, .
\end{equation*}
Summing up the last two equations, we discover that
\begin{equation} \label{eq:gamma}
\d_{t}\gamma_{\veps,M}\,=\,\curlh\langle \vec G_{\veps ,M}^{h}\rangle\,-\,\frac{\veps^{m-1}}{\mc{A}}\,\langle f_{\veps ,M}^{h}\rangle \, .
\end{equation}
Thanks to estimate \eqref{eq:approx force} in Proposition \ref{p:prop approx}, we discover that (for any $M>0$ fixed) the family 
$\left(\d_{t}\,\gamma_{\veps,M}\right)_{\veps}$
is uniformly bounded (with respect to $\veps$) in e.g. $L_{T}^{2}(L^{2})$. 
On the other hand, thanks to Lemma \ref{l:S-W_bounds} and Sobolev embeddings, we have that (for any $M>0$ fixed)
the sequence $(\gamma_{\veps,M})_{\veps}$ is uniformly bounded (with respect to $\veps$) in the space $L_{T}^{2}(H^{1})$.
Since the embedding $H_{\rm loc}^{1}\hookrightarrow L_{\rm loc}^{2}$ is compact, the Aubin-Lions Lemma implies that, for any $M>0$ fixed, the family $(\gamma_{\veps,M})_{\veps}$ is compact
(in $\veps$) in $L_{T}^{2}(L_{\rm loc}^{2})$. Then, it converges strongly (up to extracting a subsequence) to a tempered distribution $\gamma_M$ in the same space. 
Of course, by definition of $\g_{\veps,M}$, this tells us that also $\big(\curlh\lan\vec W_{\veps,M}^h\ran\big)_\veps$ is compact in $L^2_T(L^2_{\rm loc})$.

Now, we have that $\gamma_{\veps ,M}$ converges strongly to $\gamma_M$ in $L_{T}^{2}(L_{\rm loc}^{2})$ and $\langle \vec{W}_{\veps ,M}^{h}\rangle$ converges weakly to
$\langle \vec{W}_{M}^{h}\rangle$ in $L_{T}^{2}(L_{\rm loc}^{2})$ (owing to Proposition \ref{p:prop dec}, for instance). Then, we deduce that
\begin{equation*}
\gamma_{\veps,M}\langle \vec{W}_{\veps ,M}^{h}\rangle^{\perp}\longrightarrow \gamma_M \langle \vec{W}_{M}^{h}\rangle^{\perp}\qquad \text{ in }\qquad \mc{D}^{\prime}\big(\R_+\times\R^2\big)\,.
\end{equation*}
Observe that, by definition of $\g_{\veps,M}$, we must have $\gamma_M=\curlh\langle \vec{W}_{M}^{h}\rangle$. On the other hand, by Proposition \ref{p:prop dec} and the definitions given in \eqref{eq:t-T},
one has $\langle \vec{W}_{M}^{h}\rangle= \lan{S}_{M}(\chi_l\vec{U}^{h})\ran$.

In the end, we have proved that, for any $T>0$ and any test-function $\vec \psi$ as in \eqref{eq:test-f},
one has the convergence (at any $M\in\N$ fixed, when $\veps\ra0$)
\begin{equation} \label{eq:limit_T1}
\int_{0}^{T}\int_{\R^2}\mc{T}_{\veps ,M}^{1}\cdot\vec{\psi}^h\,dx^h\,dt\,\longrightarrow\,
\int^T_0\int_{\R^2}\curlh\lan{S}_{M}(\chi_l\vec{U}^{h})\ran\; \lan{S}_{M}\big(\chi_l(\vec{U}^{h})^{\perp}\big)\ran\cdot\vec\psi^h\,dx^h\,dt\,.
\end{equation}

\subsubsection{Dealing with the term $\mc{T}_{\veps ,M}^{2}$}\label{sss:term2}
Let us now consider the term $\mc{T}_{\veps ,M}^{2}$, defined in \eqref{def:T1-2}. By the same computation as above, we infer that
\begin{align}
\mc{T}_{\veps ,M}^{2}\,
&=\,\langle \divh (\widetilde{\vec{W}}_{\veps ,M}^{h})\;\;\widetilde{\vec{W}}_{\veps ,M}^{h}\rangle+\frac{1}{2}\, \langle \nabla_{h}| \widetilde{\vec{W}}_{\veps ,M}^{h}|^{2} \rangle+
\langle \curlh\widetilde{\vec{W}}_{\veps ,M}^{h}\,\left( \widetilde{\vec{W}}_{\veps ,M}^{h}\right)^{\perp}\rangle\, . \label{eq:T2} 
\end{align}

Let us now introduce now the quantities
$$
\widetilde{\Phi}_{\veps ,M}^{h}\,:=\,( \widetilde{\vec{W}}_{\veps ,M}^{h})^{\perp}-\d_{3}^{-1}\nabla_{h}^{\perp}\widetilde{\vec{W}}_{\veps ,M}^{3}\qquad\mbox{ and }\qquad
\widetilde{\omega}_{\veps ,M}^{3}\,:=\,\curlh \widetilde{\vec{W}}_{\veps ,M}^{h}\,.
$$
Then we can write
\begin{equation*}
\left( \curl \widetilde{\vec{W}}_{\veps ,M}\right)^{h}\,=\,\d_3 \widetilde{\Phi}_{\veps ,M}^{h}\qquad \text{ and }\qquad
\left( \curl \widetilde{\vec{W}}_{\veps ,M}\right)^{3}\,=\,\widetilde{\omega}_{\veps ,M}^{3}\,.
\end{equation*}
In addition, from the momentum equation in \eqref{eq:approx wave}, where we take the mean-free part and then the $\curl$, we deduce the equations
\begin{equation} \label{eq:eq momentum term2}
\begin{cases}
\veps^{m}\d_t\widetilde{\Phi}_{\veps ,M}^{h}-\veps^{m-1}\widetilde{\vec{W}}_{\veps ,M}^{h}=\veps^m\left(\d_{3}^{-1}\curl\widetilde{\vec G}_{\veps ,M} \right)^{h}\\[1ex]
\veps^{m}\d_t\widetilde{\omega}_{\veps ,M}^{3}-\veps^{m-1}\divh\widetilde{\vec{W}}_{\veps ,M}^{h}=\veps^m\,\curlh\widetilde{\vec G}_{\veps ,M}^{h}\, .
\end{cases}
\end{equation}
Making use of the relations above and of Propositions \ref{p:prop approx} and \ref{p:prop dec}, we get
\begin{equation*}
\begin{split}
\curlh\widetilde{\vec{W}}_{\veps ,M}^{h}\;\left(\widetilde{\vec{W}}_{\veps ,M}^{h}\right)^{\perp}&=\widetilde{\omega}_{\veps ,M}^{3}\left(\widetilde{\vec{W}}_{\veps ,M}^{h}\right)^{\perp}\,=\,
\veps \d_t\!\left( \widetilde{\Phi}_{\veps ,M}^{h}\right)^{\perp}\widetilde{\omega}_{\veps ,M}^{3}-
\veps\widetilde{\omega}_{\veps ,M}^{3}\left(\left(\d_{3}^{-1}\curl\widetilde{\vec G}_{\veps ,M}\right)^{h}\right)^\perp  \\
&=-\veps \left( \widetilde{\Phi}_{\veps ,M}^{h}\right)^{\perp}\d_t\widetilde{\omega}_{\veps ,M}^{3}+\mc{R}_{\veps ,M}=
\left( \widetilde{\Phi}_{\veps ,M}^{h}\right)^{\perp}\,\divh\widetilde{\vec{W}}_{\veps ,M}^{h}+\mc{R}_{\veps ,M}\, .
\end{split}
\end{equation*}
Hence, including also the gradient term into the remainders, from \eqref{eq:T2} we arrive at 
\begin{align*}
\mc{T}_{\veps ,M}^{2}\,&=\,\langle \divh\widetilde{\vec{W}}_{\veps ,M}^{h}\,\left(\widetilde{\vec{W}}_{\veps ,M}^{h}+\left(\widetilde{\Phi}_{\veps ,M}^{h}\right)^{\perp}\right) \rangle+\mc{R}_{\veps ,M} \\
&=\,\langle \div \widetilde{\vec{W}}_{\veps ,M}\left(\widetilde{\vec{W}}_{\veps ,M}^{h}+\left(\widetilde{\Phi}_{\veps ,M}^{h}\right)^{\perp}\right) \rangle -
\langle \d_3 \widetilde{\vec{W}}_{\veps ,M}^{3}\left(\widetilde{\vec{W}}_{\veps ,M}^{h}+\left(\widetilde{\Phi}_{\veps ,M}^{h}\right)^{\perp}\right) \rangle+\mc{R}_{\veps ,M}\, .
\end{align*}
The second term on the right-hand side of the last line is actually another remainder. Indeed, using the definition of the function $\widetilde{\Phi}_{\veps ,M}^{h}$ and the fact
that the test function $\vec\psi$ does not depend on $x^3$, one has
\begin{equation*}
\begin{split}
\d_3 \widetilde{\vec{W}}_{\veps ,M}^{3}\left(\widetilde{\vec{W}}_{\veps ,M}^{h}+\left(\widetilde{\Phi}_{\veps ,M}^{h}\right)^{\perp}\right)&=\d_3 \left(\widetilde{\vec{W}}_{\veps ,M}^{3}\left(\widetilde{\vec{W}}_{\veps ,M}^{h}+\left(\widetilde{\Phi}_{\veps ,M}^{h}\right)^{\perp}\right)\right) - \widetilde{\vec{W}}_{\veps ,M}^{3}\, \d_3\left(\widetilde{\vec{W}}_{\veps ,M}^{h}+\left(\widetilde{\Phi}_{\veps ,M}^{h}\right)^{\perp}\right)\\
&=\mc{R}_{\veps ,M}-\frac{1}{2}\nabla_{h}\left|\wtilde{\vec{W}}_{\veps ,M}^{3}\right|^{2}=\mc{R}_{\veps ,M}\, .
\end{split}
\end{equation*}
As for the first term, instead, we use the first equation in \eqref{eq:approx wave} to obtain
\begin{equation*}
\begin{split}
\div \widetilde{\vec{W}}_{\veps ,M}\left(\widetilde{\vec{W}}_{\veps ,M}^{h}+\left(\widetilde{\Phi}_{\veps ,M}^{h}\right)^{\perp}\right)&=-\frac{\veps^{m}}{\mc{A}} \d_t \widetilde{\Lambda}_{\veps ,M}\left(\widetilde{\vec{W}}_{\veps ,M}^{h}+\left(\widetilde{\Phi}_{\veps ,M}^{h}\right)^{\perp}\right)+\mc{R}_{\veps ,M}\\
&=\frac{\veps^{m}}{\mc{A}} \widetilde{\Lambda}_{\veps ,M}\, \d_t\left(\widetilde{\vec{W}}_{\veps ,M}^{h}+\left(\widetilde{\Phi}_{\veps ,M}^{h}\right)^{\perp}\right)+\mc{R}_{\veps ,M}\, .
\end{split}
\end{equation*}
Now, equations \eqref{eq:approx wave} and \eqref{eq:eq momentum term2} immediately yield that
\begin{equation*}
\frac{\veps^{m}}{\mc{A}} \widetilde{\Lambda}_{\veps ,M}\, \d_t\left(\widetilde{\vec{W}}_{\veps ,M}^{h}+\left(\widetilde{\Phi}_{\veps ,M}^{h}\right)^{\perp}\right)=
\mc{R}_{\veps ,M}-\frac{1}{\mc{A}}\widetilde{\Lambda}_{\veps ,M}\, \nabla_{h}\left(\widetilde{\Lambda}_{\veps ,M}\right)=
\mc{R}_{\veps ,M}-\frac{1}{2\mc{A}}\nabla_{h}\left|\widetilde{\Lambda}_{\veps ,M}\right|^{2}=\mc{R}_{\veps ,M}\,,
\end{equation*}
and this relation finally implies that $\mc{T}_{\veps ,M}^{2}\,=\,\mc R_{\veps,M}$ is a remainder, in the sense of relation \eqref{eq:reminder}.

To sum up, we have proved that, for any $T>0$ and any test-function $\vec \psi$ as in \eqref{eq:test-f},
one has the convergence
(at any $M\in\N$ fixed, when $\veps\ra0$)
\begin{equation} \label{eq:limit_T2}
\int_{0}^{T}\int_{\R^2}\mc{T}_{\veps ,M}^{2}\cdot\vec{\psi}^h\,dx^h\,dt\,\longrightarrow\,0\,.
\end{equation}

We conclude this part with a remark.
\begin{remark}\label{r:T1-T2}
Due to the presence of the term $\vec Y^2_\veps$ in \eqref{eq:wave_syst}, the choice $m\geq2$ is fundamental.
However, as soon as $F=0$, our analysis applies also in the case when $1<m<2$.
\end{remark}
 
\subsection{The limit system} \label{ss:limit} 
With the convergence established in \eqref{conv:r} to \eqref{conv:u} and in Subsection \ref{ss:convergence}, we can pass to the limit in equation \eqref{weak-mom}.
Since all the integrals will be made on $\R^2$ (in view of the choice of the test functions in \eqref{eq:test-f} above), we can safely come back to the notation on $\Omega$ instead of $\wtilde\Omega$.

To begin with, we take a test-function $\vec\psi$ as in \eqref{eq:test-f}, specifically
$$
\vec{\psi}=\big(\nabla_{h}^{\perp}\phi,0\big)\,,\qquad\qquad\mbox{ with }\qquad \phi\in C_c^\infty\big([0,T[\,\times\R^2\big)\,,\quad \phi=\phi(t,x^h)\,.
$$
For such a $\vec\psi$, all the gradient terms vanish identically, as well as all the contributions
due to the vertical component of the equation. Hence, after using also \eqref{prF}, equation \eqref{weak-mom} becomes
\begin{align}
\int_0^T\!\!\!\int_{\Omega}  
& \left( -\vre \ue^h \cdot \partial_t \vec\psi^h -\vre \ue^h\otimes\ue^h  : \nabla_h \vec\psi^h
+ \frac{1}{\ep}\vre\big(\ue^{h}\big)^\perp\cdot\vec\psi^h\right)\, dx \, dt \label{eq:weak_to_conv}\\
& =\int_0^T\!\!\!\int_{\Omega} 
\left(-\mbb{S}(\vtheta_\veps,\nabla_x\vec\ue): \nabla_x \vec\psi+\frac{1}{\veps^{2}}(\vrho_\veps -\widetilde{\vrho}_\veps)\, \nabla_x F\cdot \vec\psi\right)\,dx\,dt+
\int_{\Omega}\vrez \uez  \cdot \vec\psi(0,\cdot)\,dx\,. \nonumber
\end{align}

Making use of the uniform bounds established in Subsection \ref{ss:unif-est}, it is an easy matter to pass to the limit in the $\d_t$ term, in the viscosity term and in the centrifugal force.
Moreover, thanks to the assumptions on the initial data, we have that $\vrho_{0,\veps}\vec{u}_{0,\veps}\rightharpoonup \vec{u}_0$ in $L_{\rm loc}^2$. 

Let us consider now the Coriolis term. We can write
\begin{align*}
\int_0^T\!\!\!\int_{\Omega}\frac{1}{\ep}\vre\big(\ue^{h}\big)^\perp\cdot\vec\psi^h\,&=\,\int_0^T\!\!\!\int_{\mbb{R}^2}\frac{1}{\ep}\langle\vre \ue^{h}\rangle \cdot \nabla_{h}\phi\,=\,
-\veps^{m-1}\int_0^T\!\!\!\int_{\mbb{R}^2}\langle R_\veps\rangle\, \d_t\phi\,-\,\veps^{m-1}\int_{\mbb{R}^2}\langle  R_{0,\veps}\rangle\, \phi(0,\cdot )\,, 
\end{align*}
which of course converges to $0$ when $\veps\ra0$. Notice that the second equality derives from the mass equation \eqref{weak-con}, tested against $\phi$: namely,
\begin{equation*}
-\veps^m\int_0^T\!\!\!\int_{\mbb{R}^2}\langle\frac{\vrho_\veps-1}{\veps^m}\rangle\, \d_t\phi\,-\,\int_0^T\!\!\!\int_{\mbb{R}^2}\langle \vrho_\veps \ue^h\rangle \cdot \nabla_{h}\phi\,=\,
\veps^m\int_{\mbb{R}^2}\langle\frac{\vrho_{0,\veps}-1}{\veps^m}\rangle\, \phi(0,\cdot )\,.
\end{equation*}

It remains to deal with the convective term $\vrho_\veps \ue^h \otimes \ue^h$. For this, we take  advantage of
the analysis of Subsection \ref{ss:convergence}, and in particular of Lemma \ref{lem:convterm} and relations \eqref{eq:limit_T1} and \eqref{eq:limit_T2}.
Next, we remark that, since $\vec U^h\in L^2_T(H_{\rm loc}^1)$ by \eqref{conv:u}, from \eqref{eq:LP-Sob} we gather the strong convergence
$S_M(\chi_l\vec U^h)\longrightarrow\chi_l\vec{U}^{h}$ in $L_{T}^{2}(H^{s})$ for any $s<1$ and any $l>0$ fixed, in the limit for $M\rightarrow +\infty$.
Therefore,  in the term on the right-hand side of \eqref{eq:limit_T1}, we can perform equalities \eqref{eq:T1} backwards, and then pass to the limit also for $M\ra+\infty$.
Using that $\chi_l\equiv1$ on $\Supp\vec\psi$ by construction, we finally get the convergence (for $\veps\ra0$)
\begin{equation*}
\int_0^T\int_{\Omega} \vre \ue^h\otimes\ue^h  : \nabla_h \vec\psi^h\, \longrightarrow\, \int_0^T\int_{\R^2}\vec{U}^h\otimes\vec{U}^h  : \nabla_h \vec\psi^h\,.
\end{equation*}

In the end, letting $\varepsilon \rightarrow 0$ in \eqref{eq:weak_to_conv}, we may infer that 
\begin{align*}
&\int_0^T\!\!\!\int_{\R^2} \left(\vec{U}^{h}\cdot \d_{t}\vec\psi^h+\vec{U}^{h}\otimes \vec{U}^{h}:\nabla_{h}\vec\psi^h\right)\, dx^h\, dt\\
&\qquad\qquad= \int_0^T\!\!\!\int_{\R^2} \left(\mu(\oline\vartheta )\nabla_{h}\vec{U}^{h}:\nabla_{h}\vec\psi^h-\delta_2(m)\lan\varrho^{(1)}\ran\nabla_{h}F\cdot \vec\psi^h\right)\, dx^h\, dt-
\int_{\R^2}\lan\vec{u}_{0}^{h}\ran\cdot \vec\psi^h(0,\cdot)\,dx^h\,,
\end{align*}
where $\delta_2(m)=1$ if $m=2$, $\delta_2(m)=0$ otherwise. At this point, Remark \ref{r:F-G} applied to the case $m=2$ yields the equality
$\d_\vrho p(1,\oline\vtheta)\,\nabla_h\lan\wtilde r\ran\,=\,\nabla_hF$. Therefore, keeping in mind that $R=\vrho^{(1)}+\wtilde r$, we get
$$
\lan\varrho^{(1)}\ran\nabla_{h}F\,=\,\lan R\ran\nabla_{h}F\,-\,\lan\wtilde r\ran\nabla_{h}F\,=\,\lan R\ran\nabla_{h}F\,-\,\frac{\d_\vrho p(1,\oline\vtheta)}{2}\,\nabla_h\left|\lan\wtilde r\ran\right|^2\,.
$$
Of course, the perfect gradient disappears from the weak formulation. Using this observation in the target momentum equation written above, we finally deduce \eqref{eq_lim_m:momentum}.
This completes the proof of Theorem \ref{th:m-geq-2}, in the case when $m\geq2$ and $F\neq0$.

When $m>1$ and $F=0$, most of the arguments above still apply. We refer to the next section for more details.

\section{Proof of the convergence in the case when $F=0$} \label{s:proof-1}
In the present section we will prove the convergence result in the case $F=0$. For the sake of brevity, we will focus on the case $m=1$, completing in this way the proof to Theorem \ref{th:m=1_F=0}.
The case $m>1$ follows by a similar analysis, using at the end the compensated compactness argument depicted in Subsection \ref{ss:convergence} (recall also Remark \ref{r:T1-T2} above).



\subsection{Analysis of the acoustic-Poincar\'e waves}\label{ss:unifbounds_1} 

We start by remarking that system (NSF)$_{\veps}$ can be recasted in the form \eqref{eq:wave_syst}, with $m=1$: with the same notation introduced in Paragraph \ref{sss:wave-eq},
and after setting $X_\veps\,:=\,\div\vec{X}^1_\veps\,+\,X^2_\veps$ and $\vec Y_\veps\,:=\,\div\mbb{Y}^1_\veps\,+\,\vec{Y}^2_\veps\,+\,\nabla_x Y^3_\veps$,
we have
\begin{equation} \label{eq:wave_m=1}
\left\{\begin{array}{l}
       \veps\,\d_tZ_\veps\,+\,\mc{A}\,\div\vec{V}_\veps\,=\,\veps\,X_\veps \\[1ex]
       \veps\,\d_t\vec{V}_\veps\,+\,\nabla_x Z_\veps\,+\,\,\e_3\times \vec V_\veps\,=\,\veps\,\vec Y_\veps\,,\qquad\qquad\big(\vec{V}_\veps\cdot\vec n\big)_{|\d\Omega_\veps}\,=\,0\,.
       \end{array}
\right.
\end{equation}
This system is supplemented with the initial datum $\big(Z_{0,\veps},\vec V_{0,\veps}\big)$, where these two functions are defined as in relation \eqref{def:wave-data} above.


\subsubsection{Uniform bounds and regularisation}

In the next lemma, we establish uniform bounds for $Z_\veps$ and $\vec V_\veps$. Its proof is an easy adaptation of the one given in Lemma \ref{l:S-W_bounds}, hence omitted. One has
to use the fact that, since $F=0$, all the bounds obtained in the previous sections hold now on the whole $\Omega_\veps$, with constants which are uniform in $\veps\in\,]0,1]$;
therefore, one can abstain from using the cut-off functions $\chi_l$.
\begin{lemma} \label{l:S-W_bounds_1}
Let $\bigl(Z_\veps\bigr)_\veps$  and $\bigl(\vec V_\veps\bigr)_\veps$ be defined as in Paragraph \ref{sss:wave-eq}. Then, for any $T>0$ and all $\ep \in \, ]0,1]$, we have
$$
\sup_{\veps\in\,]0,1]}\| Z_\veps\|_{L^\infty_T((L^2+L^{5/3}+L^1+\mc{M}^+)(\Omega_\veps))} \leq c\, ,\quad\quad
\sup_{\veps\in\,]0,1]}\| \vec V_\veps\|_{L^2_T((L^2+L^{30/23})(\Omega_\veps))} \leq c \, .
$$

\end{lemma}

\medbreak
Now, we state the analogous of Lemma \ref{l:source_bounds} for $m=1$ and $F=0$.
\begin{lemma} \label{l:source_bounds_1}
For $\veps\in\,]0,1]$, let us introduce the following spaces:
\begin{enumerate}[(i)]
 \item $\mc X^\veps_1\,:=\,L^2_{\rm loc}\Bigl(\R_+;\big(L^2+L^{1}+L^{3/2}+L^{30/23}+L^{30/29}\big)(\Omega_\veps)\Bigr)$;
\item $\mc X^\veps_2\,:=\,L^2_{\rm loc}\Bigl(\R_+;\big(L^2+L^1+L^{4/3}\big)(\Omega_\veps)\Bigr)$;
\item $\mc X^\veps_3\,:=\,L^\infty_{\rm loc}\Bigl(\R_+;\big(L^2+L^{5/3}\big)(\Omega_\veps)\Bigr)$;
\item $\mc X^\veps_4\,:=\,L^\infty_{\rm loc}\Bigl(\R_+;\big(L^2+L^{5/3}+L^1+\mc{M}^+\big)(\Omega_\veps)\Bigr)$.
\end{enumerate}

Then, one has the following uniform bound, for a constant $C>0$ independent of $\veps\in\,]0,1]$:
$$
\left\|\vec X^1_\veps\right\|_{\mc X_1^\veps}\,+\,\left\|X^2_\veps\right\|_{\mc X_1^\veps}\,+\,\left\|\mbb Y^1_\veps\right\|_{\mc X_2^\veps}\,+\,
\left\|\vec{Y}^2_\veps\right\|_{\mc X_3^\veps}\,+\,\left\|Y^3_\veps\right\|_{\mc X_4^\veps}\,\leq\,C\,.
$$
In particular, one has that the sequences $(X_\veps)_\veps$ and $(\vec Y_\veps)_\veps$, defined in system \eqref{eq:wave_m=1}, verify\footnote{For any $s\in\R$, we denote by $[s]$ the entire part of $s$,
i.e. the greatest integer smaller than or equal to $s$.}
$$
\left\|X_\veps\right\|_{L^2_T(H^{-[s]-1}(\Omega_\veps))}\,+\,\left\|\vec Y_\veps\right\|_{L^2_T(H^{-[s]-1}(\Omega_\veps))}\,\leq\,C
$$
for all $s>5/2$, for a constant $C>0$ independent of $\veps\in\,]0,1]$.
\end{lemma}

\begin{proof}
The proof follows the main lines of the proof to Lemma \ref{l:source_bounds}. Here, we limit ourselves to point out that we have a slightly better control on
$\vec Y^2_\veps\,=\,\vrho_\veps^{(1)}\,\nabla_{x} G$, whose boundedness in $\mc X^\veps_3$ follows from \eqref{assFG} and the estimate analogous to \eqref{uni_varrho1} for the case $F=0$.
\qed
\end{proof}

 \medbreak
The next step is to regularise all the terms appearing in the wave system \eqref{eq:wave_m=1}. Here we have to pay attention: since the domains $\Omega_\veps$ are bounded,
we cannot use the Littlewood-Paley operators $S_M$ directly.
Rather than multiplying by a cut-off function $\chi_l$ as done in the previous section (a procedure which would create more complicated forcing terms in the wave system), we use here the arguments
of Chapter 8 of \cite{F-N} (see also \cite{F-Scho}, \cite{WK}), based on finite propagation speed properties for \eqref{eq:wave_m=1}.

First of all, similarly to Paragraph \ref{sss:w-reg} above, we extend our domains $\Omega_\veps$ and $\Omega$ by periodicity in the third variable and denote
$$
\wtilde\Omega_\veps\,:=\,{B}_{L_\veps} (0) \times \mbb{T}^1\qquad\qquad\mbox{ and }\qquad\qquad
\wtilde\Omega\,:=\,\R^2 \times \mbb{T}^1\,.
$$
Thanks to complete slip boundary conditions \eqref{bc1-2} and \eqref{bc3}, system (NSF)$_\veps$ can be equivalently reformulated in $\wtilde\Omega_\veps$. 
Analogously, the wave system \eqref{eq:wave_m=1} can be recasted in $\wtilde\Omega_\veps$ in a completely equivalent way. From now on, we will focus on
the equations satisfied on the domain $\wtilde\Omega_\veps$.

Next, we fix a smooth radially decreasing function $\omega\in{C}^\infty_c(\mbb{R}^3)$, such that $0\leq\omega\leq1$, $\omega(x)=0$ for $|x|\geq1$ and
$\int_{\R^3}\omega(x)dx=1$. Next, we define the mollifying kernel $\big(\omega_M\big)_{M\in\N}$ by the formula
$$
\omega_M(x)\,:=\,2^{3M}\,\,\omega\!\left(2^Mx\right)\qquad\qquad\forall\,M\in\N\,,\quad\forall\,x\in\R^3\,.
$$
Then, for any tempered distribution $\mf S=\mf S(t,x)$ on $\R_+\times\wtilde\Omega$ and any $M\in\N$, we define
$$
\mf S_M\,:=\,\omega_M\,*\,\mf S\,,
$$
where the convolution is taken only with respect to the space variables.

Applying the mollifier $\omega_M$ to \eqref{eq:wave_m=1}, we deduce that $Z_{\veps,M}\,:=\,\omega_M*Z_\veps$ and $\vec V_{\veps,M}\,:=\,\omega_M*\vec V_\veps$ satisfy the regularised
wave system
\begin{equation} \label{eq:reg-wave}
\left\{\begin{array}{l}
       \veps\,\d_tZ_{\veps,M}\,+\,\mc{A}\,\div\vec{V}_{\veps,M}\,=\,\veps\,X_{\veps,M} \\[1ex]
       \veps\,\d_t\vec{V}_{\veps,M}\,+\,\nabla_x Z_{\veps,M}\,+\,\,\e_3\times \vec V_{\veps,M}\,=\,\veps\,\vec Y_{\veps,M}
       \end{array}
\right.
\end{equation}
in the domain $\R_+\times\wtilde\Omega_{\veps,M}$, where we have defined
\begin{equation} \label{def:O_e-M}
\wtilde\Omega_{\veps,M}\,:=\,\left\{x\in\wtilde\Omega_\veps\;:\quad{\rm dist}(x,\d\wtilde\Omega_\veps)\geq2^{-M} \right\}\,.
\end{equation}
Since the mollification commutes with standard derivatives, we notice that
$X_{\veps,M}\,=\,\div\vec{X}^1_{\veps,M}\,+\,X^2_{\veps,M}$ and $\vec Y_{\veps,M}\,=\,\div\mbb{Y}^1_{\veps,M}\,+\,\vec{Y}^2_{\veps,M}\,+\,\nabla_x Y^3_{\veps,M}$. 
Moreover, system \eqref{eq:reg-wave} is supplemented with the initial data
$$
Z_{0,\veps,M}\,:=\,\omega_M*Z_{0,\veps}\qquad\qquad\mbox{ and }\qquad\qquad \vec V_{0,\veps,M}\,:=\,\omega_M*\vec V_{0,\veps}\,.
$$

In accordance with Lemmas \ref{l:S-W_bounds_1} and \ref{l:source_bounds_1}, by standard properties of mollifying kernels (see e.g. Section 10.1 in \cite{F-N}), we get the following properties: for all
$k\in\N$, one has
\begin{align*}
\left\|Z_{\veps,M}\right\|_{L^\infty_T(H^k(\wtilde\Omega_{\veps,M}))}\,+\,\left\|\vec V_{\veps,M}\right\|_{L^2_T(H^k(\wtilde\Omega_{\veps,M}))}\,\leq\,C(k,M) \\
\left\|X_{\veps,M}\right\|_{L^2_T(H^k(\wtilde\Omega_{\veps,M}))}\,+\,\left\|\vec Y_{\veps,M}\right\|_{L^2_T(H^k(\wtilde\Omega_{\veps,M}))}\,\leq\,C(k,M)\,,
\end{align*}
for some positive constants $C(k,M)$, only depending on the fixed $k$ and $M$. Of course, the constants blow up when $M\ra+\infty$, but they are uniform for $\veps\in\,]0,1]$.

We also have the following statement, analogous to Proposition \ref{p:prop dec} above. Its proof is also similar, hence omitted; notice that the strong convergence result easily follows
from standard properties of the mollifying kernel.
\begin{proposition} \label{p:dec_1}
For any $M>0$ and any $\veps\in\,]0,1]$, we have
\begin{equation*}
\vec{V}_{\veps ,M}\,=\,
\veps\,\vec{v}_{\veps ,M}^{(1)}\,+\,\vec{v}_{\veps ,M}^{(2)}\,,
\end{equation*}
together with the following bounds: for any $T>0$, any compact set $K\subset\wtilde\Omega$ and any $s\in\N$, one has 
(for $\veps>0$ small enough, depending only on $K$) the bounds
\begin{align*}
\left\|\vec{v}_{\veps ,M}^{(1)}\right\|_{L^{2}([0,T];H^{s}(K))}\,\leq\,C(K,s,M) \qquad\qquad\mbox{ and }\qquad\qquad
\left\|\vec{v}_{\veps ,M}^{(2)}\right\|_{L^{2}([0,T];H^{1}(K))}\,\leq\,C(K)\,,
\end{align*}
for suitable positive constants $C(K,s,M)$ and $C(K)$ depending only on the quantities in the brackets, but uniform with respect to $\veps\in\,]0,1]$.
\end{proposition}

In particular, we deduce the following fact: for any $T>0$ and any compact $K\subset\wtilde\Omega$, there exist $\veps_K>0$ and $M_K\in\N$ such that, for all $\veps\in\,]0,\veps_K]$ and all $M\geq M_K$,
there are positive constants $C(K)$ and $C(K,M)$ for which
\begin{equation} \label{est:V_e-M_conv}
\left\|\vec V_\veps\,-\,\vec V_{\veps,M}\right\|_{L^2_T(L^2(K))}\,\leq\,C(K,M)\,\veps\,+\,C(K)\,2^{-M}\,.
\end{equation}

\subsubsection{Finite propagation speed and consequences}
Take now smooth enough initial data $\mc Z_0$ and $\vec{\mc V_0}$ and external forces $\mf X$ and $\vec{\mc Y}$. Consider, in $\R_+\times\wtilde\Omega$, the wave system
\begin{equation} \label{eq:wave_Omega}
\left\{\begin{array}{l}
       \veps\,\d_t\mc Z\,+\,\mc{A}\,\div\vec{\mc V}\,=\,\veps\,\mf X \\[1ex]
       \veps\,\d_t\vec{\mc V}\,+\,\nabla_x\mc Z\,+\,\,\e_3\times \vec{\mc V}\,=\,\veps\,\vec{\mc Y}\,, 
       \end{array}
\right.
\end{equation}
supplemented with initial data  $\mc Z_{|t=0}\,=\,\mc Z_0$ and $\vec{\mc V}_{|t=0}\,=\,\vec{\mc V_0}$.

System \eqref{eq:wave_Omega} is a symmetrizable (in the sense of Friedrichs) first-order
hyperbolic system with a skew-symmetric $0$-th order term. Therefore, 
classical arguments based on energy methods (see e.g. Chapter 3 of \cite{M-2008} and Chapter 7 of \cite{Ali}) allow to establish finite propagation speed and domain of dependence properties
for solutions to \eqref{eq:wave_Omega}.

Namely, set $\lambda\,:=\,\sqrt{\mc A}/\veps$ to be the propagation speed of acoustic-Poincar\'e waves.
Let $\mf B$ be a cylinder 
included in $\wtilde\Omega$.
Then one has the following two properties.
\begin{enumerate}[(i)]
 \item \emph{Domain of dependence}: assume that 
$$
\Supp\mc Z_0\,,\;\Supp\vec{\mc V_0}\,\subset\,\mf B\,,\qquad\qquad\qquad \Supp\mf X(t)\,,\;\Supp\vec{\mc Y}(t)\,\subset\,\mf B\quad\mbox{ for a. a. }t\in[0,T]\,;
$$
then the corresponding solution $\big(\mc Z,\vec{\mc V}\big)$ to \eqref{eq:wave_Omega} is \emph{identically zero} outside the cone
$$
\Big\{(t,x)\in\,]0,T[\,\times\,\wtilde\Omega\;:\quad {\rm dist}\big(x,\mf B\big)\,<\,\lambda\,t\Big\}\,.
$$
\item \emph{Finite propagation speed}: define the set
$$
\mf B_{\lambda T}\,:=\,\Big\{x\in\wtilde\Omega\;:\quad {\rm dist}\big(x,\mf B\big)\,<\,\lambda\,T\Big\}
$$
and assume that
$$
\Supp\mc Z_0\,,\;\Supp\vec{\mc V_0}\,\subset\,\mf B_{\lambda T}\,,\qquad\qquad\qquad \Supp\mf X(t)\,,\;\Supp\vec{\mc Y}(t)\,\subset\,\mf B_{\lambda T}\quad\mbox{ for a. a. }t\in[0,T]\,;
$$
then the solution $\big(\mc Z,\vec{\mc V}\big)$ is \emph{uniquely determined} by the data inside the cone
$$
\mc C_{\lambda T}\,:=\,\Big\{(t,x)\in\,]0,T[\,\times\mf B_{\lambda T}\;:\quad {\rm dist}\big(x,\d\mf B_{\lambda T}\big)\,>\,\lambda\,t\Big\}\,,
$$
and in particular in the space-time cylinder $\,]0,T[\,\times\,\mf B$.
\end{enumerate}

\medbreak
Next, fix any test-function $\vec\psi\in C^\infty_c\big(\R_+\times\wtilde\Omega;\R^3\big)$, and let $T>0$ and the compact set $K\subset\wtilde\Omega$ be such that $\Supp\vec\psi\subset[0,T[\,\times K$.
Take a cylindrical neighborhood $\mf B$ of $K$ in $\wtilde\Omega$. 
It goes without saying that there exist an $\veps_0=\veps_0(\mf B)\in\,]0,1]$ and a $M_0=M_0(\mf B)\in\N$ such that
$$
\mf B\,\subset\,\wtilde\Omega_{\veps,M}\qquad\qquad \mbox{ for all }\qquad 0<\veps\leq\veps_0\quad\mbox{ and }\quad M\geq M_0\,,
$$
where the set $\wtilde\Omega_{\veps,M}$ has been defined in \eqref{def:O_e-M} above. Take now a cut-off function $\mf h\in C^\infty_c(\wtilde\Omega)$ such that
$\mf h\equiv1$ on $\mf B$ (and hence on $K$), and solve problem \eqref{eq:wave_Omega} with data
$$
\mc Z_0\,=\,\mf h\,Z_{0,\veps,M}\,,\qquad \vec{\mc V_0}\,=\,\mf h\,\vec V_{0,\veps,M}\,,\qquad
\mf X\,=\,\mf h\,X_{\veps,M}\,,\qquad \vec{\mc Y}\,=\,\mf h\,\vec{Y}_{\veps,M}\,.
$$
In view of assumption \eqref{dom}, the previous discussion implies that, up to take a smaller $\veps_0$, for any $\veps\leq\veps_0$ the corresponding solution
$\big(\mc Z,\vec{\mc V}\big)$ of \eqref{eq:wave_Omega}
has support in a cylinder $\wtilde\B_L\,:=\,B_L(0)\times\T\subset\wtilde\Omega_\veps$, for some $L=L(T,K,\lambda)>0$, and it must coincide with the solution $\big(Z_{\veps,M},\vec V_{\veps,M}\big)$
of \eqref{eq:reg-wave} on the set $\,]0,T[\,\times\,\mf B$, for all $0<\veps\leq \veps_0$ and all $M\geq M_0$.
In particular, for all $0<\veps\leq \veps_0$ and all $M\geq M_0$ we have
$$
\mc Z\,\equiv\,Z_{\veps,M}\quad\mbox{ and }\quad \vec{\mc V}\,\equiv\,\vec V_{\veps,M}\qquad\qquad\mbox{ on }\qquad \Supp\vec\psi\,.
$$

The previous argument shows that, without loss of generality, we can assume that the regularised wave system \eqref{eq:reg-wave} is verified on the whole $\wtilde\Omega$,
with compactly supported initial data and forces and with solutions supported on some cylinder $\wtilde\B_L$.
In particular, we can safely work with system \eqref{eq:reg-wave} and its smooth solutions $\big(Z_{\veps,M},\vec V_{\veps,M}\big)$ in the computations below.

\subsection{Convergence of the convective term} \label{ss:convergence_1}
Here we tackle the convergence of the convective term. The first step is to reduce the study to the case of smooth vector fields $\vec{V}_{\veps ,M}$.
Arguing as in Lemma \ref{lem:convterm}, and using Proposition \ref{p:dec_1} and property \eqref{est:V_e-M_conv}, one can easily prove the following approximation result.
Again, the proof is omitted.

\begin{lemma} \label{lem:convterm_1}
Let $T>0$. For any $\vec{\psi}\in C_c^\infty\bigl([0,T[\,\times\wtilde\Omega;\R^3\bigr)$, we have 
\begin{equation*}
\lim_{M\rightarrow +\infty} \limsup_{\veps \rightarrow 0}\left|\int_{0}^{T}\int_{\wtilde\Omega} \vrho_\veps\,\vec{u}_\veps\otimes \vec{u}_\veps: \nabla_{x}\vec{\psi}\, dx \, dt-
\int_{0}^{T}\int_{\wtilde\Omega} \vec{V}_{\veps ,M}\otimes \vec{V}_{\veps,M}: \nabla_{x}\vec{\psi}\, dx \, dt\right|=0\, .
\end{equation*}
\end{lemma}


\medbreak
Assume now $\vec\psi\in C_c^\infty\big([0,T[\,\times\wtilde\Omega;\R^3\big)$ is such that $\div\vec\psi=0$ and $\d_3\vec\psi=0$.
Thanks to the previous lemma, it is enough to pass to the limit in the smooth term 
\begin{align*}
-\int_{0}^{T}\int_{\wtilde\Omega} \vec{V}_{\veps ,M}\otimes \vec{V}_{\veps ,M}: \nabla_{x}\vec{\psi}\,&=\,
\int_{0}^{T}\int_{\wtilde\Omega}\div\left(\vec{V}_{\veps ,M}\otimes \vec{V}_{\veps ,M}\right) \cdot \vec{\psi}\,=\,
\int_{0}^{T}\int_{\R^2} \left(\mc{T}_{\veps ,M}^{1}+\mc{T}_{\veps, M}^{2}\right)\cdot\vec{\psi}^h\,,
\end{align*}
where, for simplicity, we agree that the torus $\T$ has been normalised so that its Lebesgue measure is equal to $1$ and, analogously to \eqref{def:T1-2}, we have introduced the quantities
$$ 
\mc T^1_{\veps,M}\,:=\, \divh\left(\langle \vec{V}_{\veps ,M}^{h}\rangle\otimes \langle \vec{V}_{\veps ,M}^{h}\rangle\right)\qquad \mbox{ and }\qquad
\mc T^2_{\veps,M}\,:=\, \divh\left(\langle \widetilde{\vec{V}}_{\veps ,M}^{h}\otimes \widetilde{\vec{V}}_{\veps ,M}^{h}\rangle \right)\,.
$$ 

We notice that the analysis of $\mc{T}_{\veps ,M}^{2}$ is similar to the one performed in Paragraph \ref{sss:term2}, up to taking $m=1$ and replacing $\vec{W}_{\veps ,M}$ and
$\Lambda_{\veps ,M}$ by $\vec{V}_{\veps ,M}$ and $Z_{\veps ,M}$ respectively.
Indeed, it deeply relies on system \eqref{eq:eq momentum term2}, which remains unchanged when $m=1$.
Also in this case, we find \eqref{eq:limit_T2}.

Therefore, we can focus on the term $\mc{T}_{\veps ,M}^{1}$ only. Its study presents some differences with respect to Paragraph \ref{sss:term1}, so let us give the full details.
To begin with, like in \eqref{eq:T1}, we have
\begin{equation*}
\mc{T}_{\veps ,M}^{1}\,=\,
\divh\langle \vec{V}_{\veps ,M}^{h}\rangle\;\; \langle \vec{V}_{\veps ,M}^{h}\rangle+\frac{1}{2}\, \nabla_{h}\left(\left|\langle \vec{V}_{\veps ,M}^{h}\rangle\right|^{2}\right)+
\curlh\langle \vec{V}_{\veps ,M}^{h}\rangle\;\;\langle \vec{V}_{\veps ,M}^{h}\rangle^{\perp}\,.
\end{equation*}

Of course, we can forget about the second term, because it is a perfect gradient.
For the first term, we use system \eqref{eq:reg-wave}: averaging the first equation with respect to $x^{3}$ and multiplying it by $\langle \vec{V}_{\veps ,M}\rangle$, we get
\begin{equation*}
\divh\langle \vec{V}_{\veps ,M}^{h}\rangle\;\;\langle \vec{V}_{\veps ,M}^{h}\rangle\,=\,-\frac{\veps}{\mc{A}}\d_t\langle Z_{\veps ,M}\rangle \langle \vec{V}_{\veps ,M}^{h}\rangle+
\frac{\veps}{\mc{A}} \langle X_{\veps ,M}\rangle \langle \vec{V}_{\veps ,M}^{h}\rangle\,=\,
\frac{\veps}{\mc{A}}\langle Z_{\veps ,M}\rangle \d_t \langle \vec{V}_{\veps ,M}^{h}\rangle +\mc{R}_{\veps ,M}\,.
\end{equation*}
We now use the horizontal part of \eqref{eq:reg-wave}, multiplied by $\langle Z_{\veps ,M}\rangle$, and we gather
\begin{equation*}
\begin{split}
\frac{\veps^{m}}{\mc{A}}\langle Z_{\veps ,M}\rangle \d_t \langle \vec{V}_{\veps ,M}^{h}\rangle &=-\frac{1}{\mc{A}} \langle Z_{\veps ,M}\rangle \nabla_{h}\langle Z_{\veps ,M}\rangle-
\frac{1}{\mc{A}}\langle Z_{\veps ,M}\rangle\langle \vec{V}_{\veps ,M}^{h}\rangle^{\perp}+\frac{\veps}{\mc{A}}\langle Z_{\veps ,M}\rangle \langle \vec Y_{\veps ,M}^{h}\rangle\\
&=-\frac{1}{\mc{A}}\langle Z_{\veps ,M}\rangle\langle \vec{V}_{\veps ,M}^{h}\rangle^{\perp}+\mc{R}_{\veps ,M}\, .
\end{split}
\end{equation*}
This latter relation yields that
\begin{equation*}
\mc{T}_{\veps ,M}^{1}\,=\,\left(\curlh\langle \vec{V}_{\veps ,M}^{h}\rangle-\frac{1}{\mc{A}}\langle Z_{\veps ,M}\rangle \right)\langle \vec{V}_{\veps ,M}^{h}\rangle^{\perp}+\mc{R}_{\veps ,M} .
\end{equation*}

Now we use the horizontal part of \eqref{eq:reg-wave}: 
averaging it with respect to the vertical variable and
applying the operator $\curlh$, we find
\begin{equation*}
\veps\,\d_t\curlh\langle \vec{V}_{\veps ,M}^{h}\rangle\,+\,\divh\langle \vec{V}_{\veps ,M}^{h}\rangle \,=\,\veps\, \curlh\langle \vec Y_{\veps ,M}^{h}\rangle\, .
\end{equation*}
Taking the difference of this equation with the first one in \eqref{eq:reg-wave}, we discover that
\begin{equation*}
\d_t\g_{\veps,M}
\,=\,\curlh\langle \vec Y_{\veps ,M}^{h}\rangle\,-\,\frac{1}{\mc{A}}\,\langle X_{\veps ,M}\rangle\,,\qquad\qquad \mbox{ with }\qquad
\gamma_{\veps, M}:=\curlh\langle \vec{V}_{\veps ,M}^{h}\rangle\,-\,\frac{1}{\mc{A}}\langle Z_{\veps ,M}\rangle\,.
\end{equation*}
An argument analogous to the one used after \eqref{eq:gamma} above, based on Aubin-Lions Lemma, shows also in this case that
$(\gamma_{\veps,M})_{\veps}$ is compact (in $\veps$) in $L_{T}^{2}(L_{\rm loc}^{2})$. Then, this sequence converges strongly (up to extraction of a suitable subsequence)
to a tempered distribution $\gamma_M$ in the same space. 

At this point, since $\gamma_{\veps ,M}\rightarrow\gamma_M$ strongly in $L_{T}^{2}(L_{\rm loc}^{2})$ and
$\langle \vec{V}_{\veps ,M}^{h}\rangle\rightharpoonup\langle \vec{V}_{M}^{h}\rangle$ weakly in $L_{T}^{2}(L_{\rm loc}^{2})$, we deduce that
\begin{equation*}
\gamma_{\veps,M}\,\langle \vec{V}_{\veps ,M}^{h}\rangle^{\perp}\,\longrightarrow\, \gamma_M\, \langle \vec{V}_{M}^{h}\rangle^{\perp}\qquad \text{ in }\qquad \mc{D}^{\prime}\big(\R_+\times\R^2\big),
\end{equation*}
where $\langle \vec{V}_{M}^{h}\rangle=\lan{\omega}_{M}*\vec{U}^{h}\ran$ and $\gamma_M=\curlh\lan{\omega}_{M}*\vec{U}^{h}\ran-(1/\mc{A})\langle Z_{M}\rangle$.
Notice that, in view of \eqref{conv:rr}, \eqref{conv:theta}, \eqref{est:sigma}, Proposition \ref{p:prop_5.2} and the definitions given in \eqref{relnum}, we have
$$
Z_M\,=\,\d_\vrho p(1,\oline\vtheta)\,\omega_M*\vrho^{(1)}\,+\,\d_\vtheta p(1,\oline\vtheta)\,\omega_M*\Theta\,=\,\omega_M*q\,,
$$
where $q$ is the quantity defined in \eqref{eq:for q}. 
Owing to the regularity of the target velocity $\vec U^h$, we can pass to the limit also for $M\ra+\infty$, thus finding that
\begin{equation} \label{eq:limit_T1-1}
\int^T_0\!\!\!\int_{\wtilde\Omega}\vrho_\veps\,\vec{u}_\veps\otimes \vec{u}_\veps: \nabla_{x}\vec{\psi}\, dx \, dt\,\longrightarrow\,
\int^T_0\!\!\!\int_{\R^2}\big(\vec U^h\otimes\vec U^h:\nabla_h\vec\psi^h\,+\,\frac{1}{\mc A}\,q\,(\vec U^h)^\perp\cdot\vec\psi^h\big)\,dx^h\,dt
\end{equation}
for all test functions $\vec\psi$ such that $\div\vec\psi=0$ and $\d_3\vec\psi=0$. Recall the convention $|\T|=1$.
Notice that, since $\vec U^h=\nabla_h^\perp q$, the last term in the integral on the right-hand side is actually zero.

\subsection{End of the proof} \label{ss:limit_1}
Thanks to the previous analysis, we are now ready to pass to the limit in equation \eqref{weak-mom}.
As done above, we take a test-function $\vec\psi$ such that
$$
\vec{\psi}=\big(\nabla_{h}^{\perp}\phi,0\big)\,,\qquad\qquad\mbox{ with }\qquad \phi\in C_c^\infty\big([0,T[\,\times\R^2\big)\,,\quad \phi=\phi(t,x^h)\,.
$$
Notice that $\div\vec\psi=0$ and $\d_3\vec\psi=0$. Then, all the gradient terms and all the contributions coming from the vertical
component of the momentum equation vanish identically, when tested against such a $\vec\psi$. So, equation \eqref{weak-mom} reduces
to\footnote{Remark that, in view of our choice of the test-functions, we can safely come back to the notation on $\Omega$ instead of $\wtilde\Omega$.}
\begin{align*}
\int_0^T\!\!\!\int_{\Omega}  \left( -\vre \ue \cdot \partial_t \vec\psi -\vre \ue\otimes\ue  : \nabla \vec\psi
+ \frac{1}{\ep}\vre\big(\ue^{h}\big)^\perp\cdot\vec\psi^h+\mbb{S}(\vtheta_\veps,\nabla_x\vec\ue): \nabla_x \vec\psi\right)
 =\int_{\Omega}\vrez \uez  \cdot \vec\psi(0,\cdot)\,.
\end{align*}

As done in Subsection \ref{ss:limit}, we can limit ourselves to consider the rotation and convective terms only.
As for the former term, we start by using the first equation in \eqref{eq:wave_m=1} against $\phi$: we get
\begin{equation*} 
\begin{split}
-\int_0^T\!\!\!\int_{\R^2} \left( \lan Z_{\varepsilon}\ran\, \d_{t}\phi +\frac{\mc{A}}{\veps}\, \lan\vrho_{\veps}\ue^{h}\ran\cdot \nabla_{h}\phi\right)=
\int_{\R^2}\lan Z_{0,\varepsilon }\ran\, \phi (0,\cdot ) +\varepsilon \int_0^T\!\!\!\int_{\R^2} \lan X_{\varepsilon}\ran\cdot \phi\, ,
\end{split}
\end{equation*}
whence we deduce that
\begin{align*}
\int_0^T\!\!\!\int_{\Omega}\frac{1}{\ep}\vre\big(\ue^{h}\big)^\perp\cdot\vec\psi^h\,&=\,\int_0^T\!\!\!\int_{\mbb{R}^2}\frac{1}{\ep}\langle\vre \ue^{h}\rangle \cdot \nabla_{h}\phi\, \\
&=\,-\,\frac{1}{\mc A}\int_0^T\!\!\!\int_{\mbb{R}^2}\langle Z_\veps\rangle\, \d_t\phi\,-\,\frac{1}{\mc A}\int_{\mbb{R}^2}\langle Z_{0,\veps}\rangle\, \phi(0,\cdot )\,-\,
\frac{\veps}{\mc A}\int_0^T\!\!\!\int_{\R^2}\lan X_{\varepsilon}\ran\cdot \phi\,. 
\end{align*}

Letting now $\varepsilon \rightarrow 0$, thanks to the previous relation and \eqref{eq:limit_T1-1}, we finally gather
\begin{align*}
&-\int_0^T\!\!\!\int_{\R^2} \left(\vec{U}^{h}\cdot \d_{t}\nabla_{h}^{\perp} \phi+ \vec{U}^{h}\otimes \vec{U}^{h}:\nabla_{h}(\nabla_{h}^{\perp}\phi )+\frac{1}{\mc{A}}q\, \d_t \phi \right)\, dx^h\, dt\\
&=-\int_0^T\!\!\!\int_{\R^2} \mu (\oline\vartheta )\nabla_{h}\vec{U}^{h}:\nabla_{h}(\nabla_{h}^{\perp}\phi ) \, dx^h\, dt+\int_{\R^2}\left(\lan\vec{u}_{0}^{h}\ran\cdot \nabla _{h}^{\perp}\phi (0,\cdot )+
\frac{1}{\mc{A}}\lan q_{0}\ran\phi (0,\cdot )\right) \, dx^h\, dt\, ,
\end{align*}
where $q$ is defined as in \eqref{eq:for q} and we have set $q_0\,=\,\d_\vrho p(1,\oline\vtheta)R_0+\d_\vtheta p(1,\oline\vtheta)\Theta_0-G-1/2$
(keep in mind \eqref{conv:in_data} above). Theorem \ref{th:m=1_F=0} is finally proved.

\appendix

\section{Appendix -- A few tools from Littlewood-Paley theory} \label{app:LP}

This appendix is devoted to present some tools from Littlewood-Paley theory, which we have exploited in our analysis.
We refer e.g. to Chapter 2 of \cite{B-C-D} for details.
For simplicity of exposition, let us deal with the $\R^d$ case, with $d\geq1$; however, the whole construction can be adapted also to the $d$-dimensional torus $\TT^d$, and to the ``hybrid'' case
$\R^{d_1}\times\TT^{d_2}$.

\medbreak
First of all, let us introduce the \emph{Littlewood-Paley decomposition}. For this
we fix a smooth radial function $\chi$ such that $\Supp\chi\subset B(0,2)$, $\chi\equiv 1$ in a neighborhood of $B(0,1)$
and the map $r\mapsto\chi(r\,e)$ is non-increasing over $\R_+$ for all unitary vectors $e\in\R^d$.
Set $\varphi\left(\xi\right)=\chi\left(\xi\right)-\chi\left(2\xi\right)$ and $\vphi_j(\xi):=\vphi(2^{-j}\xi)$ for all $j\geq0$.
The dyadic blocks $(\Delta_j)_{j\in\Z}$ are defined by\footnote{We agree  that  $f(D)$ stands for 
the pseudo-differential operator $u\mapsto\mc{F}^{-1}[f(\xi)\,\what u(\xi)]$.} 
$$
\Delta_j\,:=\,0\quad\mbox{ if }\; j\leq-2,\qquad\Delta_{-1}\,:=\,\chi(D)\qquad\mbox{ and }\qquad
\Delta_j\,:=\,\varphi(2^{-j}D)\quad \mbox{ if }\;  j\geq0\,.
$$
For any $j\geq0$ fixed, we  also introduce the \emph{low frequency cut-off operator}
\begin{equation} \label{eq:S_j}
S_j\,:=\,\chi(2^{-j}D)\,=\,\sum_{k\leq j-1}\Delta_{k}\,.
\end{equation}
Note that $S_j$ is a convolution operator. More precisely, after defining
$$
K_0\,:=\,\mc F^{-1}\chi\qquad\qquad\mbox{ and }\qquad\qquad K_j(x)\,:=\,\mathcal{F}^{-1}[\chi (2^{-j}\cdot)] (x) = 2^{jd}K_0(2^j x)\,,
$$
for all $j\in\N$ and all tempered distributions $u\in\mc S'$ we have that $S_ju\,=\,K_j\,*\,u$.
Thus the $L^1$ norm of $K_j$ is independent of $j\geq0$, hence $S_j$ maps continuously $L^p$ into itself, for any $1 \leq p \leq +\infty$.

The following classical property holds true: for any $u\in\mc{S}'$, then one has the equality $u=\sum_{j}\Delta_ju$ in the sense of $\mc{S}'$.
Let us also recall the so-called \emph{Bernstein inequalities}.
  \begin{lemma} \label{l:bern}
Let  $0<r<R$.   A constant $C$ exists so that, for any non-negative integer $k$, any couple $(p,q)$ 
in $[1,+\infty]^2$, with  $p\leq q$,  and any function $u\in L^p$,  we  have, for all $\lambda>0$,
$$
\displaylines{
{\Supp}\, \widehat u \subset   B(0,\lambda R)\quad
\Longrightarrow\quad
\|\nabla^k u\|_{L^q}\, \leq\,
 C^{k+1}\,\lambda^{k+d\left(\frac{1}{p}-\frac{1}{q}\right)}\,\|u\|_{L^p}\;;\cr
{\Supp}\, \widehat u \subset \{\xi\in\R^d\,:\, \lambda r\leq|\xi|\leq \lambda R\}
\quad\Longrightarrow\quad C^{-k-1}\,\lambda^k\|u\|_{L^p}\,
\leq\,
\|\nabla^k u\|_{L^p}\,
\leq\,
C^{k+1} \, \lambda^k\|u\|_{L^p}\,.
}$$
\end{lemma}   

By use of Littlewood-Paley decomposition, we can define the class of Besov spaces.
\begin{definition} \label{d:B}
  Let $s\in\R$ and $1\leq p,r\leq+\infty$. The \emph{non-homogeneous Besov space}
$B^{s}_{p,r}$ is defined as the subset of tempered distributions $u$ for which
$$
\|u\|_{B^{s}_{p,r}}\,:=\,
\left\|\left(2^{js}\,\|\Delta_ju\|_{L^p}\right)_{j\geq -1}\right\|_{\ell^r}\,<\,+\infty\,.
$$
\end{definition}
Besov spaces are interpolation spaces between Sobolev spaces. In fact, for any $k\in\N$ and~$p\in[1,+\infty]$
we have the chain of continuous embeddings $ B^k_{p,1}\hookrightarrow W^{k,p}\hookrightarrow B^k_{p,\infty}$,
which, when $1<p<+\infty$, can be refined to
$$
 B^k_{p, \min (p, 2)}\hookrightarrow W^{k,p}\hookrightarrow B^k_{p, \max(p, 2)}\,.
$$
In particular, for all $s\in\R$ we deduce that $B^s_{2,2}\equiv H^s$, with equivalence of norms:
\begin{equation} \label{eq:LP-Sob}
\|f\|_{H^s}\,\sim\,\left(\sum_{j\geq-1}2^{2 j s}\,\|\Delta_jf\|^2_{L^2}\right)^{\!\!1/2}\,.
\end{equation}

As an immediate consequence of the first Bernstein inequality, one gets the following embedding result, which generalises Sobolev embeddings.
\begin{proposition}\label{p:embed}
The space $B^{s_1}_{p_1,r_1}$ is continuously embedded in the space $B^{s_2}_{p_2,r_2}$ for all indices satisfying $p_1\,\leq\,p_2$ and either
$s_2\,<\,s_1-d\big(1/p_1-1/p_2\big)$, or $s_2\,=\,s_1-d\big(1/p_1-1/p_2\big)$ and $r_1\leq r_2$.
\end{proposition}


\end{document}